\documentclass[11pt]{article}
\usepackage{multicol}
\usepackage{amssymb,amsmath,amsthm,bm}
\usepackage[colorlinks=true, urlcolor=blue,linkcolor=blue, citecolor=blue]{hyperref}
\usepackage{mathrsfs,verbatim}
\usepackage{caption,cleveref }
\usepackage{graphicx, float}
\usepackage{color, mathtools, tikz, xcolor}
\usepackage{enumerate,enumitem}
    \setlist{nosep}
\usepackage{tikz}
\usetikzlibrary{shapes.geometric}
\usepackage{pdfpages}
\usepackage{bbm}
\usepackage{algorithm}
\usepackage{algorithmic}
\usetikzlibrary{calc}

\newcommand{\dist}{\text{dist}}
\usepackage[mathscr]{euscript}
\usepackage{appendix}
\usepackage{geometry}
\numberwithin{equation}{section}

\renewcommand{\S}{\mathcal{S}}

\renewcommand{\l}{\left}
\renewcommand{\r}{\right}

\usepackage{svg}
\usepackage{multicol}
\usepackage{amssymb,amsmath,amsthm,bm}
\usepackage[colorlinks=true, urlcolor=blue,linkcolor=blue, citecolor=blue]{hyperref}

\numberwithin{equation}{section}

	\newtheorem{theorem}{Theorem}[section]
	\newtheorem{problem}[theorem]{Problem}
        \newtheorem{ques}[theorem]{Question}
	
	\newtheorem{coro}[theorem]{Corollary}
	\newtheorem{conjecture}[theorem]{Conjecture}
	\newtheorem{claim}[theorem]{Claim}
	\newtheorem{proposition}[theorem]{Proposition}
	\newtheorem{lemma}[theorem]{Lemma}
 
	\newtheorem{fact}[theorem]{Fact}

        \theoremstyle{definition}
	\newtheorem{defn}[theorem]{Definition}

\newenvironment{proofclaim}[1][Proof of claim]{\begin{proof}[#1]}{\end{proof}}

\setlist{nolistsep}

\title{Embedding loose trees in~$k$-uniform hypergraphs}

\author{Yaobin Chen\thanks{Shanghai Center for Mathematical Sciences,~Fudan University,~Shanghai,~200438,~China.~{\tt ybchen21@m.fudan.edu.cn}.
Supported in part by National Natural Science Foundation of China grant 123B2012 and China Scholarship Council (CSC) No. 202306100218.}
\and Allan Lo\thanks{School of Mathematics, University of Birmingham, Birmingham, B15 2TT, UK. {\tt s.a.lo@bham.ac.uk}. The research leading to these results was supported by EPSRC, grant no. EP/V002279/1 and EP/V048287/1. There are no additional data beyond that contained within the main manuscript.}
}

\begin{document}

\maketitle
\begin{abstract}
    
A classical result of Koml{\'o}s, S{\'a}rk{\"o}zy and Szemer{\'e}di shows that every large $n$-vertex graph with minimum degree at least $(1/2+\gamma)n$ contains all spanning trees of bounded degree. We generalised this result to loose spanning hypertrees in~$k$-uniform hypergraphs, that is, linear hypergraphs obtained by subsequently adding edges sharing a single vertex with a previous edge.

    We give a general sufficient condition for embedding loose trees with bounded degree. In particular, we show that for all $k\ge 4$, every $n$-vertex $k$-uniform hypergraph with~$n\ge n_0(k,\gamma, \Delta)$ and minimum $(k-2)$-degree at least $(1/2+\gamma)\binom{n}{k-2}$ contains every spanning loose tree with maximum vertex degree at most $\Delta$. This bound is asymptotically tight. This generalises a result of Pehova and Petrova, who proved the case when $k=3$ and of Pavez-Signé, Sanhueza-Matamala and Stein, who considered the codegree threshold for bounded degree tight trees.

\end{abstract}

\section{Introduction}\label{intro}
  The study of spanning trees of bounded degree in graphs and random graphs has been an active area of research for many years. A classical result of Koml{\'o}s, S{\'a}rk{\"o}zy and Szemer{\'e}di~\cite{Tree} proved that every large $n$-vertex graph $G$ with minimum degree at least $(1/2+\gamma)n$ contains all spanning trees of bounded degree. The result is clearly the best possible, as exemplified by two disconnect complete graphs on~$n/2$ vertices. Csaba, Levitt, Nagy-Gy{\"o}rgy and Szemer{\'e}di~\cite{csaba2010tight} showed that $\gamma n$ can be replaced by~$C\Delta \log n$ improving a previous result of~\cite{komlos2001spanning}. This result has also been generalized to random graphs, pseudorandom graphs and directed graphs, see~\cite{krivelevich2010embedding,krivelevich2017bounded,han2022spanning,kathapurkar2022spanning,mycroft2020trees,Im2024}. 

  In this paper, we consider an analogue of this problem for~$k$-uniform hypergraphs. Since there is no single way to define a hypertree. We introduce the definition about hypertrees due to Kalai~\cite{kalai1983enumeration}.  A \emph{$k$-uniform $s$-hypertree} is defined iteratively as follows: a single $k$-uniform edge is a $k$-uniform $s$-hypertree; any $k$-graph can be obtained from a $k$-uniform~$s$-hypertree $T$ by adding a new edge~$e$ such that there exists $e' \in E(T)$ with~$| e \cap e' | = s = | e \cap V(T)|$. Note that a $n$-vertex $k$-uniform $s$-hypertree satisfies $n\equiv s \mod k-s$, which we will always assume.

  Throughout this paper, we refer to the $k$-uniform $1$-tree as \emph{$k$-loose tree} (also known in the literature as a \emph{linear tree}). Observe that $2$-loose tree is the usual tree in graphs, since a tree can be defined as adding leaves iteratively. Also, a loose path in~$k$-uniform graph is a $k$-loose tree. The \emph{minimum~$\ell$-degree} $\delta_\ell(H)$ of a $k$-uniform hypergraph $H$ is the minimum number of edges containing any given set of~$\ell$ vertices. Moreover, the \emph{minimum relative} $\ell$\emph{-degree}~$\overline{\delta}_\ell(G)$ is $\delta_\ell(G)/\binom{n}{k-\ell}$. The \emph{maximum} (\emph{relative}) \emph{$\ell$-degree}, denoted by~$\Delta_\ell(G)$ (and $\overline \Delta_{\ell}(G)$, respectively) is defined analogously.

  Extremal problem about hypertrees has a long history. In 1995 Kalai~\cite[Conjecture 3.6]{frankl1987exact}, conjectured that every $k$-graph with more than $\frac{t-1}{k}\binom{n}{k-1}$ edges contains every $k$-loose tree with~$t$~edges. In general, Kalai's conjecture is still open but there are many partial results. Regarding spanning $k$-loose trees, Georgakopoulos, Haslegrave, Montgomery and Narayanan~\cite{georgakopoulos2022spanning} proved that every large $n$-vertex $3$-graph with minimum $2$-degree at least $n/3+o(n)$ have a spanning triangulation of a $2$-sphere, which in particular contains some spanning $3$-loose tree. Pavez-Signé, Sanhueza-Matamala and Stein~\cite{pavez2024dirac} proved that the minimum $(k-1)$-degree threshold for the existence of any bounded vertex degree spanning $(k-1)$-hypertree.
  \begin{theorem}[Pavez-Signé, Sanhueza-Matamala and Stein~\cite{pavez2024dirac}]\label{Thm1}
      For all $k\ge 2, \gamma>0$ and $\Delta\in \mathbb{N}$, there exists an $n_0$ such that any $k$-graph~$G$ on~$n\ge n_0$ vertices with~$\overline{\delta}_{k-1}(G)\ge 1/2+\gamma$ contains every $n$-vertex $(k-1)$-hypertree $T$ with~$\Delta_1(T)\le \Delta$ and the bound is asymptotic tight. In particular, $G$ contains all $n$-vertex $k$-loose trees $T$ with~$\Delta_1(T)\le \Delta$.
  \end{theorem}

 For general $\ell$ and $k$, what is the minimum (relative) $\ell$-degree threshold which forces the existence of any bounded vertex degree spanning $k$-loose trees? 
    We say a $k$-graph~$G$ on~$n$ vertices is \emph{$\Delta$-loose-tree-universal}, if $G$ contains every spanning $k$-loose tree~$T$ with~$\Delta_1(T)\le \Delta$.
    Moreover, we define the \emph{loose tree embeddable $(k,\ell)$-threshold}, denoted by~$\delta_{k,\ell}^{T}$, to be the infimum $\delta$ such that for any $\gamma>0$ and $\Delta\in \mathbb{N}$, there exists an integer $n_0$ such that any $k$-graph~$G$ on~$n\ge n_0$ vertices with~$n \equiv 1 \mod k-1$ and $\overline{\delta}_{\ell}(G)>\delta+\gamma$ is $\Delta$-loose-tree-universal. 
    Thus Theorem~\ref{Thm1} implies that $\delta_{k,k-1}^{T} = 1/2$.

  The following example shows that $\delta_{k,\ell}^{T}$ should also guarantee a perfect matching in~$k$-graphs. We define the \emph{binary $k$-loose tree} $T_{k,r}$ with depth~$r$ as follows. Let $T_{k,0}$ be a loose tree with one edge and root it at an arbitrary vertex. We say a vertex is a \emph{leaf vertex} if it is not the root vertex and in only one edge. So $T_{k,0}$ has $k-1$ leaf vertices. Then for any $r\ge 1$, let $T_{k,r}$ be the loose tree by adding one edge on each leaf vertex of~$T_{k,r-1}$. If $r$ is even, then $T_{k,r}$ contains a perfect matching and the maximum vertex degree of~$T_{k,r}$ is~$2$, see Figure~\ref{fig:enter-label1}.

  \begin{figure}
      \centering
      \includegraphics[scale=0.4]{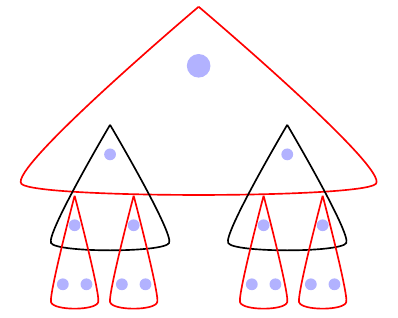}
      \caption{The binary $3$-loose tree~$T_{3,2}$.}
      \label{fig:enter-label1}
  \end{figure}

   We denote the \emph{perfect matching $(k,\ell)$-threshold} $\delta_{k,\ell}^{PM}$ to be the infimum $\delta$ such that for any $\gamma>0$, there exists an integer $n_0$ such that any $k$-graph $G$ on~$n\ge n_0$ vertices with~$k|n$ and $\overline{\delta}_\ell(G)>\delta+\gamma$ contains a perfect matching. 
   Ruci\'nski and Szemer\'edi~\cite{Rodlalmost} showed that $\delta_{k,k-1}^{PM}=1/2$ and H\'an, Person and Schacht~\cite{han2009perfect} showed that $\delta_{3,1}^{PM}=5/9$. 
   See \cite{Yisurvey} for a survey.

   Recently, Pehova and Petrova~\cite{pehova2024embedding} showed that every $3$-graph $G$ with~$\overline{\delta}_1(G)> 5/9$ contains all bounded degree $3$-loose trees.
   \begin{theorem}[Pehova and Petrova~\cite{pehova2024embedding}]\label{Thm2}
    We have $\delta_{3,1}^{T} = 5/9$.
    In particular,  for all $\gamma>0$ and $\Delta\in \mathbb{N}$, there exists an $n_0$ such that any $3$-graph~$G$ on~$n\ge n_0$ vertices with~$n$ odd and $\overline{\delta}_{1}(G)\ge 5/9+\gamma$ contains every spanning $3$-loose tree~$T$ with~$\Delta_1(T)\le \Delta$.
  \end{theorem}

 Pehova and Petrova~\cite{pehova2024embedding} further conjectured that $\delta_{k,\ell}^{T} = \delta_{k,\ell}^{PM}$ in general.

 \begin{conjecture}[Pehova and Petrova~\cite{pehova2024embedding}]\label{Conjecture1}
 For all $1 \le \ell < k$, $\delta_{k,\ell}^{T} = \delta_{k,\ell}^{PM}$.
  \end{conjecture}

By Theorems~\ref{Thm1} and~\ref{Thm2}, this conjecture holds for~$(k,\ell)=(k,k-1)$ and $(k,\ell)=(3,1)$, respectively. 
Our main result shows that Conjecture~\ref{Conjecture1} indeed holds for all $\ell=k-2$ with~$k\ge 4$.
Note that $\delta_{k,k-2}^{PM}=1/2$ by Pikhurko~\cite{pikhurko2008perfect}. 

\begin{theorem}\label{Theorem Matching threshold}
    For all $k \ge 4$, $\delta_{k,k-2}^{T} = 1/2 =\delta_{k,k-2}^{PM}$.
    In particular, for all $\gamma>0$ and $\Delta\in \mathbb{N}$, there exists an~$n_0$ such that any $k$-graph~$G$ on~$n\ge n_0$ vertices with~$\overline{\delta}_{k-2}(G)\ge 1/2+\gamma$ and $n\equiv 1\mod k-1$ contains every spanning $k$-loose tree~$T$ with~$\Delta_1(T)\le \Delta$.
\end{theorem}

Note that in the proof of Theorems~\ref{Thm1} and~\ref{Thm2}, one ingredient is to find a tight Hamilton cycle. 
Tight Hamilton cycles exists by the minimum degree condition by~\cite{reiher2019minimum} and~\cite{RSS-Hamilton}.
Lang and Sanhueza-Matamala~\cite{lang2022minimum} and independently, Polcyn, Reiher, R\"odl and Sch\"ulke~\cite{polcyn2021Hamiltonian} showed that the corresponding tight Hamilton cycle $(k,k-2)$-threshold is $5/9 > \delta_{k,k-2}^{PM}$.
One of our main contributions is to provide a necessary condition, which we called the \emph{robust framework}, for~$\Delta$-loose-tree-universality, see Section~\ref{embed }.

\textbf{Paper organization:} In Section~\ref{preli}, we introduce some notations and terminology that we will use throughout the paper. In Section~\ref{sketch}, we give a sketch proof for Theorem~\ref{Theorem Matching threshold}. In Section~\ref{embed }
we give a robust framework for embedding loose trees. A robust framework will consist of three key properties, robust fractional matching, reachable and rotatable. In Section~\ref{Regular}, we introduce the hypergraph regularity lemma and related embedding method in the reduced graph. In Section~\ref{proof of general theorem}, we prove that the robust framework is sufficient for embedding loose trees. In Section~\ref{defn}, we construct a subgraph $G^*$ which will turn out to be a robust framework. We will verify that $G^*$ satisfies the three key properties in Sections~\ref{matching},~\ref{reachable } and~\ref{rotatable  }. We conclude the paper with a discussion and open problems in Section~\ref{remark}.

\section{Preliminaries}\label{preli}
\subsection{Notation}\label{Notations}

 We write $v_1\dots v_k$ for~$\{v_1,\dots,v_k \}$. 
 We omit floors and ceilings whenever they are not crucial.
  We say that a statement holds for~$\alpha \ll\beta $ if there exists a non-decreasing function~$f$ such that the statement holds for~$\alpha <f(\beta)$. We write $x=y\pm \gamma$ for~$x\in [y-\gamma,y+\gamma]$ and $[n]$ to denote the set of integers from~$1$ to~$n$. We denote the  symmetric group of~$[n]$ by~$S_n$.

 A \emph{$k$-graph} $G$ consists of a set of vertices~$V(G)$ and a set of edges $E(G)$, where each edge consists of~$k$ vertices. Let $H$ be a $k$-graph. We write $G-H$ for the subgraph of~$G$ by removing edges of~$H$ from~$G$ and~$G\backslash H$ for the subgraph of~$G$ by removing vertices of~$H$ from~$G$.
 For~$U \subseteq V(G)$, $G[U]$ is the subgraph induced on~$U$. 
 The degree $\deg_G(U)$ of~$U$ is the number of edges in~$G$ containing~$U$.

 Let $A_1,\dots,A_{\ell}\subseteq V(G)$. 
 An edge $v_1\dots v_k$ is an $A_1\dots A_\ell$\emph{-edge} if $v_i\in A_i$ for all $i\in [\ell]$.
 If $A_{i+1}=\dots=A_{\ell}=V(G)$, then an $A_1\dots A_i$-edge is also an $A_1\dots A_{\ell}$-edge.
 We write $\deg_{G}(A_1;A_2;\dots ;A_\ell)$ for the number of~$A_1\dots A_\ell$-edges in~$G$. 
 Note that $\deg_G(v_1\dots v_i)=\deg_G(v_1;\dots ;v_i)$ for all~$i\in [k]$.  
 Let $d_G(A_1;A_2;\dots;A_k) = \deg_{G}(A_1;A_2\dots ;A_k)/ \l(|A_1||A_2|\dots|A_k|\r)$. 
 
 For~$2\le j\le k$, let the \emph{$j$th-shadow graph} $\partial_j(G)$ be the $j$-graph on~$V(G)$ whose edges are the $j$-sets contained in some edge of~$G$. For any set $S\subseteq V(G)$ with size smaller than $k$, the \emph{link graph of~$S$}, denoted by~$L_{G}(S)$, is the $(k-|S|)$-graph on vertex set $V(G)$ whose edges are all $(k-|S|)$-tuples~$T$ for which $T\cup S\in E(G)$.  Write~$\partial_j(S)$ for~$\partial_j(L_G(S))$. 
 
 Given edges $e$ and $e'$ of~$G$, a \emph{tight walk} from~$e$ to~$e'$ is a sequence of edges $e_0,e_1,\dots,e_{\ell}$ such that $e_0 = e$, $e_\ell=e'$ and $|e_{i-1}\cap e_i|=k-1$ for each $i\in [\ell]$. If the vertices on the walk are distinct, we say it is a \emph{tight path}. Note that connected by tight walks gives an equivalence relation on the edge set. The \emph{tight components} of~$G$ are the equivalence classes of this relation. We say that $G$ is \emph{tight connected} if it has only one tight component. 

 For~$k$-graphs $H_1$ and $H_2$, a \emph{hypergraph homomorphism} from~$H_1$ to~$H_2$ is a function~$\phi:V(H_1)\rightarrow V(H_2)$ such that for all $e\in E(H_1)$, $\phi(e)\in E(H_2)$. An \emph{embedding} is an injective homomorphism. Throughout the paper, we will usually use $\phi$ to denote homomorphism and $\psi$ to denote embedding.

 The next property shows that relative minimum $j$-degree is no smaller than relative minimum $1$-degree.
\begin{proposition}\label{j-degree}
    Let $G$ be a $k$-graph on~$n$ vertices and $\ell\in [k-1]$. Then $\overline{\delta}_1(G)\ge \overline{\delta}_\ell(G)$.
\end{proposition}
\begin{proof}
    For any vertex $v\in G$, $$\deg(v)=\sum_{v\in J\in \binom{V(G)}{\ell}} \deg(J) /\binom{k-1}{\ell-1}\ge \binom{n-1}{\ell-1}\cdot \binom{n-\ell}{k-j}\overline{\delta}_\ell(G)/\binom{k-1}{\ell-1}=\binom{n-1}{k-1}\overline{\delta}_\ell(G).$$
    Hence $\overline{\delta}_1(G)\ge \overline{\delta}_\ell(G)$.
\end{proof}
 \subsection{Structure of loose trees}\label{sturc of tree}

 In this subsection, we give some basic facts and definitions about loose trees. Consider a $k$-graph $G$. We say a vertex colouring $c:V(G)\rightarrow \mathbb{N}$ is a \emph{proper colouring} if for any adjacent vertex pair $u,v$, $c(u)\neq c(v)$.
It is easy to see that any $k$-loose tree~$T$ can be properly coloured with colours in~$[k]$. 
We denote its colour classes by~$C_1(T),\dots,C_k(T)$.

Let $T$ be a $k$-loose tree with a root vertex $r$, colour classes $C_1(T),\dots,C_k(T)$ and $r\in C_1(T)$. A \emph{layering~$L(T)$ of~$T$} is a partition of~$V(T)$ into layers $L_1,\dots, L_p$ such that $L_1=\{r\}$,  $C_{i}(T)=\bigcup_{j\equiv i \mod k} L_j$, each edge is an $L_i L_{i+1}\dots L_{i+k-1}$-edge for some~$i$ and each $v\in L_i$ is contained in at most one non-$L_i L_{i+1}\dots L_{i+k-1}$-edge. We say a \emph{rooted k-loose tree~$T$ at~$r$} is a $k$-loose tree~$T$ together with a proper $k$-vertex colouring and a layering $L(T)$ of~$T$ rooted at~$r$. 

Let $T$ be a rooted $k$-loose tree at~$r$. For a vertex $v\in L_i\subseteq V(T)$, the subtree $T(v)$ rooted at~$v$ is defined to be the loose tree containing $v$ in~$T[\bigcup_{j\ge i} L_j]$. We say the vertices in~$V(T(v))\backslash \{v\}$ to be the \emph{descendants} of~$v$.

Note that if we remove the root vertex from a loose tree, then the remaining parts can be viewed as subtrees rooted at neighbours of the root vertex.
\begin{fact}\label{partition tree}
    Let $T$ be a rooted $k$-loose tree at~$r$ and $\tau$ be permutation~$(12\dots k)$. Let $\mathcal{T}_j$ be all rooted subtrees~$T(v)$ with~$v\in C^1_{\tau^j(1)}(T)$. Then $r,V(\mathcal{T}_1),\dots V(\mathcal{T}_{k-1})$ partition~$V(T)$. 
\end{fact}

 We define a distance function to describe the location of each edge. 
 Let $T$ be a $k$-loose tree. 
 Between a vertex and an edge or between two edges, there exists only one path connecting them in~$T$.     
 For vertex~$v\in V(T)$ and edge~$e\in E(T)$, $\text{dist}_T(v,e)$ is the length of the path in~$T$ between $e$ and~$v$ minus~$1$. For two edges $e,e'\in E(T)$, $\text{dist}_T(e,e')$ is the length of the path in~$T$ between~$e$ and~$e'$ minus~$1$.
Moreover, if $v\in e$ or~$e=e'$, then $\dist(v,e)$ and $\dist(e,e')$ is zero by our definitions.

\subsection{$\alpha$-perturbed graphs}

The reduced graph~$R$ obtained after the regularity lemma inherits some key features of the host graph~$G$, see Section~\ref{Regular} for more details. 
In particular, if $\overline{\delta}_{\ell}(G) \ge \delta$, then almost all $\ell$-vertex subsets of~$R$ will also have degree at least $\delta - \varepsilon$. 
We use the following definition to quantify this property. 

\begin{defn}[$\alpha$-perturbed degree]\label{defn:perturbed degree}
    Let $1\le \ell <k$ and $\alpha,\delta>0$. We say that a $k$-graph $G$ has \emph{$\alpha$-perturbed minimum relative} $\ell$-\emph{degree}, denoted by~$\overline{\delta}^\alpha_{\ell}(G)$, at least $\delta$, if the following hold for every~$j\in [\ell]$
    \begin{enumerate}[label=(P\arabic*)]
        \item\label{itm:defn:perturbed degree1} every edge of~$E(\partial_j(G))$ has relative degree at least $\delta$ in~$G$;
        \item \label{itm:defn:perturbed degree2}$\overline{\partial_j(G)}$ has edge density at most $\alpha$;
        \item \label{itm:defn:perturbed degree3} each $(j-1)$-edge of~$\partial_{j-1}(G)$ has relative degree less than $\alpha$ in~$\overline{\partial_j(G)}$.        
    \end{enumerate}
\end{defn}

We obtain the following proposition.
\begin{proposition}\label{perturb prop}
    Let $\alpha\ll \delta\ll 1/k$. Let $G$ be a $k$-graph on~$n$ vertices with  $\overline{\delta}^{\alpha}_{k-2}(G)\ge \delta$.
    \begin{enumerate}[label=\rm{(L\arabic*)}]
        \item\label{itm:perturb prop 1} Let $S\in E(\partial_{j}(G))$ with~$j\le k-2$. Then the link graph $L_G(S)$ of~$S$ satisfies  $\overline{\delta}^{\alpha}_{k-j-2}(L_G(S))\ge \delta$.
        \item\label{itm:perturb prop 2} Let $A_1,\dots, A_{\ell}$ be subsets of~$V(G)$ with each of size~$m\le k-4$. Let $B\subseteq V(G)$ with size larger than $\alpha m n  $. Then there exists an edge $uv$ with~$u,v\in B$ and $uv\in \bigcap_{i=1}^{\ell}E(\partial_2(A_i))$. 
    \end{enumerate}
\end{proposition}
\begin{proof}
    Note that~\ref{itm:perturb prop 1} follows from Definition~\ref{defn:perturbed degree}. For~\ref{itm:perturb prop 2}, by Definition~\ref{defn:perturbed degree}, there are at most $\alpha n$ isolated vertices in~$\partial_2(A_i)$ for~$i\in [\ell]$. This implies that there is a vertex~$v\in V(G)$ which is not isolated in all~$\partial_2(A_i)$. By Definition~\ref{defn:perturbed degree}~\ref{itm:defn:perturbed degree3}, the relative degree of~$v$ in~$V(\overline{\partial_2(A_i)})$ is at most $\alpha$. Note that 
    \begin{align*}
        |B|- m \cdot \alpha n > \alpha m n - \alpha m n= 0.
    \end{align*}
    Hence, there exists an edge $uv\in \bigcap_{i=1}^{\ell}E(\partial_2(A_i))$ with~$u,v\in B$. 
\end{proof}
We remark that the transition from perturbed edges to perturbed minimum degree comes at negligible costs. We get the following lemma.
\begin{lemma}[{\cite[Lemma~2.3]{lang2022minimum}}]\label{perturbed degree}
Let $1/n\ll \varepsilon\ll \alpha\ll \delta,1/k$. Let $G$ be a $k$-graph on~$n$ vertices with~$\overline{\delta_\ell}(G)\ge \delta$. Let $I$ be a subgraph of~$G$ of edge density at most $\varepsilon$. Then there exists a spanning subgraph $G'$ of~$G-I$ with~$\overline{\delta}^{\alpha}_{\ell}(G')\ge \delta-\alpha$.
\end{lemma}

\section{Proof sketch of Theorem~\ref{Theorem Matching threshold}}\label{sketch}

In this section, we sketch the proof of Theorem~\ref{Theorem Matching threshold} and discuss some new ideas we used.

Let $T$ be a $k$-loose tree on $n$ vertices. 
Suppose that we are given a large $k$-graph $G$~with $n$ vertices with $\overline{\delta}_{k-2}(G)\ge 1/2+2\gamma$.
We now find an embedding of~$T$ in~$G$ using the absorption technique, which splits the proof into the following three steps. 

\noindent
\textbf{Step 1: find an absorber.} Fix a root vertex~$r$ in~$T$ such that $T$ is union of two $k$-loose trees $T_1$ and $T_2$ both rooted at~$r$, where $T_1$ is small.
Find an embedding~$\psi_1$ from~$T_1$ to~$G$ such that any almost embedding of~$T$ can be extended to a full embedding (see Lemma~\ref{Absorption lemma}).

\noindent
\textbf{Step 2: embed an almost spanning tree.} 
Remove the vertices in $G$ that is used to embed~$T_1 \setminus r$ and call the resulting graph~$G_2$.
Since $T_1$ is small, $\overline{\delta_\ell}(G_2)\ge 1/2+\gamma$. 
Find an almost embedding~$\psi_2$ from~$T_2$ into~$G_2$ such that $\psi_1(r) = \psi_2(r)$, that is, the root vertices of $T_1$ and $T_2$ are mapped to the same vertex in~$G$.  
Together with $\psi_1$, we obtain an almost embedding from~$T$ to~$G$. 

\noindent
\textbf{Step 3: complete the embedding.} Extend the embedding of Step 2 to be a spanning tree by our absorber.

Clearly, Step~3 follows immediately from Steps~1 and~2.

Our proof of Step~1 follows a similar approach to~\cite{pehova2024embedding}, which was inspired by~\cite{Bottcher2019,Bottcher2020,stein2020tree}.
We describe the key ingredient for the absorber. 
Suppose we have already had an almost embedding~$\psi$ of~$T$ and we would like to extend the embedding by attaching an edge to~$\psi(x)$.
Pick vertices $w_2,\dots, w_k$ in~$G$ that are not covered. 
The embedding~$\psi_1$ of~$T_1$ contains some (unused) absorbing tuple, which consists of $k-1$ vertex~disjoint loose stars.
More importantly, the centres of these star together with $\psi(x)$ is an edge~in $G$ and (in~$\psi$) we can replace centres of these stars with $w_2,\dots, w_k$.
Thus we have `enlarged' the embedding~$\psi$. 
See Section~\ref{sec:absorptionlemma} for further details. 
Note that $\overline{\delta_1}(G)\ge 1/2+\gamma$ is already sufficient for Step~1.

Therefore, we now focus on Step~2. 
In particular, we aim to find an almost spanning $k$-loose tree~$T_2$ in the $k$-graph~$G_2$.
We first apply the weak hypergraph regularity lemma and obtain a reduce graph~$R_2$. 
To embed~$T_2$ into~$G_2$, we find a homomorphism~$\phi$ from~$T_2$ to~$R_2$ and then turn it into an embedding of~$T_2$ in~$G_2$. 
Note that we need to ensure that $\phi$ does not map too many vertices of~$T_2$ into one particular cluster of~$R_2$. 
We now give a sketch on how to construct $\phi$ and highlight the necessary properties of~$R_2$ that we need (see Section~\ref{embed } for their formal definitions).

First we find a perfect matching~$M = \{ e_1, \dots, e_\ell\}$ in $R_2$. 
This exists by our assumption as $\overline{\delta}_{k-2}(R_2) \ge \overline{\delta}_{k-2}(G_2) - \gamma/2 \ge (1 + \gamma)/2 \ge \delta_{k,k-2}^{PM}$.
We break $T_2$ into smaller trees $T_{2,1}, \dots, T_{2,s}$ with a large constant $s  \gg \ell$, where $T_{2,i}$ is a subtree of $\bigcup_{i' \in [i]} T_{2,i'}$.
We construct our homomorphism~$\phi$ by first mapping (most of) $T_{2,1}, \dots T_{2,s_1}$ into~$e_1$, $T_{2,s_1+1}, \dots T_{2,s_2}$ into~$e_2$ and so on. 

Since $T_2$ is connected, we require $e_1$ to be connected to~$e_2$ or else we cannot find a homomorphism. 
However just being in the same loose component is not sufficient as illustrated by the following example. 
Suppose that $R$ is a $4$-graph with vertex set $V_1 \cup V_2$ with $V_1 \cap V_2 = \emptyset$ and that all edges in~$R$ contains even number of vertices in~$V_1$.
Observe that there is no homomorphism of a $4$-loose path that maps edges to an edge in~$V_1$ and an edge in~$V_2$. 
Thus, we require a stronger notation, which we call `reachable'.
Roughly speaking, an edge~$e$ is reachable from an edge~$f$ if we can find a homomorphism of~$T_2$ that maps its root to a given vertex in~$f$ and maps almost all vertices of~$T_2$ to~$e$.
Note that reachable may not be symmetric but induces a partial ordering. 
We require that reachable induced a linear ordering on~$M$, namely, $e_j$ is reachable from $e_i$ for $i \le j$. 

Recall that $\phi$ will map most of $T_{2,1}, \dots T_{2,s_1}$ into~$e_1$.
Suppose for each such $T_{2,i}$, majority of its vertices lie in the same colour class that contains in its root. 
Thus a naive homomorphism from $\bigcup_{i \in [s_1]} T_{2,i}$ into~$e_1$ will map most of the vertices onto one vertex of~$e_1$.
In order to ensure that each vertex of~$e_1$ has similar number of preimages, we require each edge to be `rotatable'. 

We say that $R_2$ is `robust' if it satisfies the three mentioned properties of having a  perfect (fractional) matching, reachable and rotatable.
In fact, one of our contribution is to show that being robust is already suffices to Step~2, that is, finding a almost spanning tree, see Section~\ref{sec:almostembedding}.

Recall in the outline of Step~$2$ that we require both roots of $T_1$ and $T_2$ to be mapped onto the same vertex~$x^*$ in~$G$. 
Hence, we actually find a robust subgraph in the reduce graph before Step~$1$.
In particular, we pick~$x^*$ to be in the `first edge' of~$M$ (in the ordering induced by reachability). 
After embedding~$T_1$, the clusters of the reduce graph may have different sizes. 
Since $T_1$ is small, our robust condition still ensures the existence of a perfect fractional matching.

\subsection{Verifying robustness}
Marjory of our paper is dedicated to verifying the reduce graph contains a spanning robust subgraph.

We first review the known cases when $(k,\ell) = (k,k-1)$~\cite{pavez2024dirac} and $(k,\ell) = (3,1)$~\cite{pehova2024embedding}. 
In both cases, we have $\delta_{k,\ell}^T=\delta^{THC}_{k,\ell}$.
This ensures that there is a tight Hamilton cycle~$C$ in the reduce graph. 
Note that $C$ contains a perfect fractional matching. 
By winding around~$C$, one can show that $C$ is reachable.
When $(k,\ell) = (3,1)$, $C$ is also rotatable (c.f.~\cite[Lemma~3.1]{pehova2024embedding}). 
Therefore $C$ is robust. 

In this paper, we consider the case $(k,\ell) = (k,k-2)$ with $k \ge 4$.  
Most importantly, $\delta^{PM}_{k,k-2}<\delta^{THC}_{k,k-2}$ as discussed in Section~\ref{intro}.
Thus, one cannot guarantee a tight Hamiltonian cycle in the reduce graph and so various new ideas are needed to find a spanning robust subgraph.

For simplicity, we focus on the case when $(k,\ell) = (4,2)$.
We first describe the structure of the robust subgraph.
Consider $A \in \binom{V(R)}{2}$.
Note that its link graph is a $2$-graph and let $C_A$ be the largest component. 
Let $E_A$ be the set of $4$-edges of the form $e\cup A$ for $e\in E(C_A)$.
Note that edges in $E_A$ are tight connected (and being tight connected is ideal for being reachable).
Let $R^*$ be the subgraph induced by union of $E_A$ for all $A \in \binom{V(R)}{2}$.
See  Section~\ref{defn} for formal definition.
In Section~\ref{matching}, we show that $R^*$ satisfies the perfect matching property. 

Recall that $E_A$ is tight connected for $A\in \binom{V(R)}{2}$. 
Thus edges in $E_A$ are reachable within themselves. 
We can define an equivalent relationship on~$\binom{V(R)}{2}$ such that $A\sim A'$ if and only if $E_A\cup E_{A'}$ are reachable within themselves.
This can be represented by a colouring of~$\binom{V(R)}{2}$, i.e. an edge-colouring of the complete graph on $V(R)$. 
Furthermore, we prove that the colouring is Gallai (i.e., no rainbow triangle) and locally $2$-edge colouring, see Lemma~\ref{Gallai coloured}.
By analysing monochromatic components, we obtain the desired enumeration of~$E_{R^*}$.
See Section~\ref{reachable } for further discussions and its proof.

Roughly speaking, for a permutation $\sigma \in S_4$, we say that an edge $e = x_1x_2x_3x_4$ is $\sigma$-rotatable if for any $4$-loose tree $T$ with root $r$ and $i \in [4]$, there is a homomorphism~$\psi$ such that $r$ maps to $x_i$ and most of $C_i(T)$ is mapped to $x_{\sigma(i)}$.
Thus $e$ is rotatable if it is $\sigma$-rotatable for all $\sigma \in S_4$.
We will show that $e$ is $\sigma$-rotatable for $\sigma \in \{(12), (13), (14)\}$.
Since $\{(12), (13), (14)\}$ is a generator set of~$S_4$, we can then deduce that $e$ is rotatable. 
It should be noted that the general case when $k \ge 4$ can be deduced to the case when $k = 4$. 
See Section~\ref{rotatable  } for further discussions and its proof.


\section{A sufficient condition for almost spanning trees embedding}\label{embed }
 In this section, we discuss some natural conditions and give a sufficient condition for embedding almost spanning loose trees.

 We now consider a $k$-graph $G$ with its reduced graph $R$ obtained by hypergraph regularity lemma (see Section~\ref{Regular}). Let $T$ be an almost spanning $k$-loose tree. To embed~$T$, we will construct a homomorphism $\phi$ from~$T$ to~$R$ and then turn it into an embedding in~$G$.  Next, we discuss three natural conditions that $R$ needs to satisfy in order to construct $\phi$.

Recall that there exists a bounded vertex degree loose tree with a perfect matching.
Thus, we seek a (fractional) perfect matching structure in~$R$. Secondly, any two edges in a loose tree are connected by a loose path. This requires us to connect edges in the matching of~$R$. It is called \emph{reachable}. Meanwhile, to construct homomorphism in reduced graph $R$, we need to guarantee that there is no cluster overused. This is called \emph{rotatable}. Formally, we have the following definitions.

\begin{defn}($\omega$-fractional matching)\label{defn:fractional matching}
    Let $G$ be a $k$-graph and a vertex weighting $\omega:V(G)\rightarrow [0,1]$. We say $\omega^*:E(G)\rightarrow [0,1]$ is an \emph{$\omega$-fractional matching} if $\sum_{e\ni v}\omega^*(e)\le \omega(v)$ for all $v\in V(G)$. The \emph{size of~$\omega^*$} is $\sum_{e\in E(G)}\omega^*(e)$. We say $\omega^*$ is \emph{perfect} if the size of~$\omega^*$ is $\sum_{v\in V(G)}\omega(v)/k$.
\end{defn}
 The next definition will enable us to find a tree homomorphism on the reduced graph from one edge to another edge.
\begin{defn}\label{Defn:reachable}(Reachable)
    Let $G$ be a $k$-graph and $C \in \mathbb{N}$. We say that an edge $e\in E(G)$ is \emph{$C$-reachable} from a vertex $u\in V(G)$ if for any rooted $k$-loose tree~$T$ at~$r$, there exists a homomorphism~$\phi$ from~$T$ to~$G$ such that $\phi(r)=u$ and $\phi(v)\in e$ for all $v$ with~$\dist(r,v)>C$. Moreover, we say $e$ is \emph{$C$-reachable} from~$e'\in E(G)$ if $e$ is $C$-reachable from all $u\in e'$. 
\end{defn}

Meanwhile, we define a rotation property to embed vertices on each cluster in the reduced graph balanced. Recall that $S_k$ is the set of permutation of~$[k]$.

\begin{defn}\label{defn:rotatable}(Rotatable)
Let $C \in \mathbb{N}$, $G$ be a $k$-graph and $e=u_1\dots u_k\in E(G)$. We say $e$ is \emph{$C$-rotatable} if for any $u\in e$, any $\sigma\in S_k$ and any rooted $k$-loose tree~$T$ at~$r$, there exists a homomorphism~$\phi$ from~$T$ to~$G$ such that $\phi(r)=u$ and $\phi(v)=u_{\sigma(s)}$ for all $v\in C_s(T)$ with~$\dist(r,v)\ge C$.
\end{defn}

The next property shows that reachable and rotatable imply that we can map most edges in a tree to some specific edge while the root vertex is mapped to another edge. 
\begin{proposition}\label{prop:homo exist}
    Let $G$ be a $k$-graph, $v\in V(G)$ and $e=x_1\dots x_k\in E(G)$. Suppose that $e$ is $C$-reachable from~$v$ and $C'$-rotatable. Let $T$ be a rooted $k$-loose tree at~$r$ with~$\Delta_1(T)\le \Delta$. Then there exists a homomorphism~$\phi$ from~$T$ to~$G$ such that $\phi(r)=v$ and for all $j\in [k]$, $$|C_j(T)\backslash \phi^{-1}(x_j)|\le (k\Delta)^{C+C'}.$$
\end{proposition}
\begin{proof}
     By Definition~\ref{Defn:reachable}, there is a homomorphism~$\phi'$ from~$T$ to~$G$ such that for~$w\in V(T)$ with~$\dist(r,w)\ge C$, $\phi'(w)\in e$. Let $\tau=(12\dots k)\in S_k$. Then for vertex $w\in V(T)$ with~$\dist(r,w)=C$ and $w\in C_i(T)$, consider the subtree $T(w)$ with colour classes $$C_j(T(w))=C_{\tau^{i-1}(j)}(T)\cap V(T(w)).$$ 
     Hence, $w\in C_i(T)\cap V(T(w))=C_1(T(w))$. By Definition~\ref{defn:rotatable}, there is a homomorphism~$\phi_w$ from~$T(w)$ to~$G$ such that $\phi_w(w)=\phi'(w)\in e$ and, for each vertex $x\in C_{\tau^{-(i-1)}(j)}(T(w))=C_j(T)\cap V(T(w))$ with~$\dist(w,x)\ge C'$, we have  $\phi_w(x)=x_j$. Define $\phi$ to be such that
  \[
\phi(u) =
\begin{cases}
\phi'(u)&\quad \text{ if }\dist(r,u)\le C, \\
\phi_w(u)&\quad \text{ if }u\in T(w).
\end{cases}
\]
 If $u\in C_j(T)$ with~$\dist(r,u)\ge C+C'$, then we have $\phi(u)=x_j$. So $|C_j(T)\backslash \phi^{-1}(x_j)|\le (k\Delta)^{C+C'}$.
\end{proof}
Next, we define the key properties for embedding almost spanning loose trees.
\begin{defn}~\label{robust}
    Let $G$ be a $k$-graph on~$n$ vertices and $\eta >0$. We say $G$ is $\eta$-\emph{robust} if 
    \begin{enumerate}[label={\rm(R\arabic*)}]
        \item\label{robust1} for any vertex weighting $\omega:V(G)\rightarrow [0,1]$ with~$\sum_{v\in V(G)}\omega(v)\ge (1-\eta)|V(G)|$ and $\omega(v)=0$ if $v$ is an isolated vertex of~$G$, there exists a perfect $\omega$-fractional matching $\omega^{\ast}$ of~$G$;
        \item\label{robust2} there exists an integer $C_1=C_1(n)$ and an enumeration of~$E(G)$, $e_1,e_2,\dots ,e_{|E(G)|}$ such that $e_i$ is $C_1$-reachable from~$e_j$ for all $i\ge j$;
        \item\label{robust3} there exists an integer $C_2=C_2(n)$ such that all edges in~$E(G)$ are $C_2$-rotatable.
\end{enumerate}
     For~$1\le \ell \le k-1$, the \emph{robust $\ell$-degree threshold for~$k$-graph}, denoted by~$\delta_{k,\ell}^R$, is the infimum $\delta>0$ such that for~$1/n\ll \alpha\ll \eta \ll \gamma$, every $k$-graph $G$ on~$n$ vertices with~$\overline{\delta}^\alpha_{\ell }(G)\ge \delta+\gamma$ contains an $\eta$-robust spanning subgraph.
\end{defn}
We remark that our proof shows that $C_1,C_2\le n^{4k}$. Since we will seek an $\eta$-robust subgraph in the reduced graph, $C_1$ and $C_2$ will be treated as large constants independent of the order of host graph.
Note that \ref{robust1} implies that $G$ contains a perfect matching, so we have 
\begin{align}
    \delta_{k,\ell}^R \ge \delta_{k,\ell}^{PM} \ge 1/2. \label{eqn:robustthresholdlowerbound}
\end{align}
Now, we state our general theorem.   

\begin{theorem}\label{generaltheorem}
For~$k\ge 2$ and $\ell\in [k-1]$, we have that $\delta_{k,\ell}^{T}\le \delta_{k,\ell}^{R}$.
\end{theorem}

 By this theorem, to prove Theorem~\ref{Theorem Matching threshold}, it suffices to show that $\delta_{k,k-2}^{R}=1/2$ (see Lemma~\ref{lem:k-2thereshold}). We prove Theorem~\ref{generaltheorem} in Section~\ref{proof of general theorem} after we introduce the hypergraph regularity lemma.

\section{Weak hypergraph regularity lemma}\label{Regular}
A main tool we use to prove Theorem~\ref{generaltheorem} is the \emph{weak hypergraph regularity lemma}. Given $\varepsilon>0$, we say that a $k$-partite $k$-graph $G$ on sets $V_1,\dots ,V_k$ is \emph{$\varepsilon$-regular} if for every $X_i\subseteq V_i$ of size at least $\varepsilon |V_i|$, we have that
\begin{align*}
    |d_G(X_1,\dots,X_k)-d_{G}(V_1,\dots, V_k)|\le \varepsilon.
\end{align*}

A partition $\{V_0,V_1,\dots, V_t\}$ of the vertex set of a $k$-graph $G$ on $n$ vertices is \emph{$\varepsilon$-regular} if $|V_0|\le \varepsilon n$, all other sets $V_i$ with $i\in [t]$ have equal size, and the graph induced by all but at most $\varepsilon \binom{t}{k}$ $k$-tuple of $V_i$ for $i\in [t]$ is $\varepsilon$-regular.

\begin{theorem}[\cite{ChungFan, Frankl}]
    For every $\varepsilon>0$ and $t_0\in \mathbb{N}$, there exist $T_0,n_0\in \mathbb{N}$ such that every $k$-graph $G$ with at least $n_0$ vertices admits an $\varepsilon$-regular partition $\{V_0,V_1,\dots, V_t\}$, where $t_0\le t\le T_0$.
\end{theorem}

Given an $\varepsilon$-regular partition $\mathcal{Q}=\{V_0,\dots, V_t\}$ of the vertex set of a $k$-graph $G$, we define the \emph{$(\mathcal{Q},\varepsilon, d)$-reduced graph} $R$ on vertex set $\{1,\dots, t\}$ corresponding to the sets $\{V_0,\dots,V_t\}$. The edges of $R$ are all $\varepsilon$-regular $k$-tuples of density at least $d$. 
By~Lemma~\ref{perturbed degree}, we can imply the following degree version of hypergraph regularity lemma.

\begin{lemma}\label{regular set up}
    For $0\ll \varepsilon\ll \alpha\ll d \ll \delta$, $k>j\ge 1$ and $t_0,\in \mathbb{N}$, there exist $T_0,n_0\in \mathbb{N}$ such that the following holds. Suppose that $G$ is a $k$-graph on $n\ge n_0$ vertices with minimum degree $\delta_{j}(G)\ge \delta\binom{n}{k-j}$. Then $G$ admits an $\varepsilon$-regular partition $\mathcal{Q}=\{V_0,\dots, V_t\}$ such that $t_0\le t<T_0 $ and the $(\mathcal{Q},\varepsilon,d)$-reduced graph $R$ contains a spanning subgraph $R'$ with minimum perturbed degree $\delta_{j}^{\alpha}(R')\ge (\delta-\alpha-d)\binom{t}{k-j}$
\end{lemma}

Next, we prove an embedding lemma for regular tuple in our reduce graph. This enables us to transform a homomorphism to an almost embedding. Given vertex sets $X_2, \dots ,X_k$ we say that a vertex $v\notin X_2\cup \dots \cup  X_k$ is $d$-\emph{dense} into~$\{X_2,\dots, X_k\}$ if $\deg(v;X_2;\dots; X_k)\ge d|X_2|\cdots |X_k|$. In the process of embedding trees, we aim to embed each edge into a dense $k$-tuple. In order to do this, we have the following lemma.

\begin{lemma}\label{expand}
    Let $1/m\ll \varepsilon\ll d\ll 1/k,1/\Delta$.
    Let $G$ be a $k$-graph.
    Let $\mathcal{Q}$ be a $\varepsilon$-regular partition of~$G$ and $R$ be the corresponding $(\mathcal{Q},\varepsilon,d)$-reduced graph.
    Let $X_1,\dots ,X_k$ and $Y_{j}^1,\dots, Y_{j}^{(k-1)\Delta}$ for~$j\in [2,k]$ (not necessarily distinct) be vertex clusters of~$\mathcal{Q}$ with size $m$ such that, for any~$0\le i < \Delta$ and~$j\in [2,k]$, $X_1X_2\dots X_{k}$ and $X_jY_j^{(k-1)i+1}Y_j^{(k-1)i+2} \dots Y_j^{(k-1)(i+1)}$  are $\varepsilon$-regular with density at least~$d$. 
    For $j\in [2,k]$ and $i \in [(k-1)\Delta]$, let $Z_j \subseteq X_j$ and $Z^i_j \subseteq Y^i_j$  be of size at least~$\sqrt{\varepsilon} m$.
    Suppose that $x\in X_1$ is a vertex which is $d/4$-dense into~$\{Z_2,\dots, Z_{k}\}$. 
    Then there exist vertices~$z_2\in Z_2,\dots ,z_k\in Z_k$ such that $xz_2\dots z_{k}$ is an edge and each $z_j$ is $d/4$-dense into~$\{ Z_j^{(k-1)i+1} , \dots , Z_j^{(k-1)(i+1)} \}$ for every $0\le i <\Delta$.
\end{lemma}

\begin{proof}
    Let $W$ be the set of all vertices $ w \in Z_2$ such that $\deg(x;w;Z_3;\dots; Z_k)\ge {d|Z_3|\cdots |Z_k|/8}$.
    Since the vertex $x\in X_1$ is $d/4$-dense into~$\{Z_2,\dots ,Z_k\}$, we have 
    \begin{align*}
        \frac{d|Z_2|\cdots |Z_k|}{4} & \le 
        \deg(x;Z_2;\dots; Z_k)=\sum_{w \in W} \deg(x;w;Z_3;\dots; Z_k)+ \sum_{w' \in Z_2 \setminus W} \deg(x;w';Z_3;\dots ;Z_k)\\
        & \le |W| |Z_3|\cdots  |Z_k| + |Z_2| \l(\frac{d|Z_3|\cdots |Z_k|}{8}\r),\\
        |W| & \ge d|Z_2|/8 \ge \varepsilon m.
    \end{align*}
    
    Since $X_2Y_2^{1} \dots Y_2^{k-1}$ is $\varepsilon$-regular with density at least~$d$, we obtain that $d(W;Z_2^1;\dots ;Z_2^{k-1}) \ge d/2$.
    Let $W_1$ be the set of all $ w \in W_1$ such that $w$ is $d/4$-dense into $\{ Z_2^{1} , \dots , Z_2^{k-1} \}$.
    Thus, each $w \in W_1$ satisfies $\deg(w;Z_2^1;\dots ;Z_2^{k-1})\ge d|Z_2^1|\dots |Z_2^{k-1}|/4$.
    By a similar calculation as above, we deduce that 
    \begin{align*}
        |W_1|\ge \frac{d|W|}{4}\ge \frac{d^2|Z_2|}{4^3}\ge \varepsilon m.
    \end{align*}
    We repeat this procedure and  obtain $W_1 \supseteq W_2 \supseteq \dots  \supseteq W_{\Delta}$ such that 
    each vertex~$w \in W_i$ is $d/4$-dense into~$\{ Z_2^{(k-1)(i-1)+1} , \dots , Z_2^{(k-1)i} \}$ and $ |W_i| \ge {d^{i+1} |Z_2|}/{ 4^{2+i} } \ge \varepsilon m $.
    We fix $z_2 \in W_{\Delta} \subseteq Z_2$.
    Note that $z_2$ is $d/4$-dense into~$\{ Z_2^{(k-1)i+1} , \dots , Z_2^{(k-1)(i+1)} \}$ for every $0\le i <\Delta$ and 
    $\deg(x;z_2;Z_3;\dots;Z_k )\ge d|Z_3|\dots |Z_k|/8$.

    By repeating the similar construction of~$z_2$, we obtain $z_3, \dots ,z_{k}$ as required.
\end{proof}


\section{From robust framework to embedding: proof of Theorem~\ref{generaltheorem}}\label{proof of general theorem}

The proof of Theorem~\ref{generaltheorem} can be divided into three steps similar to those in Section~\ref{sketch}.
We prove Steps~1 and~2 in Sections~\ref{sec:absorptionlemma} and~\ref{sec:almostembedding}, respectively.
We complete the proof of Theorem~\ref{generaltheorem} in Section~\ref{sec:complete embedding}.

\subsection{Absorption lemma}   \label{sec:absorptionlemma}
~~~~In this subsection, we show our absorption structure and related lemmas. We develop the absorption technique used in~\cite{Bottcher2019,Bottcher2020,stein2020tree,pehova2024embedding}. We start with some definitions.

\begin{defn}\label{absorber star}
    A \emph{$k$-uniform (loose) $d$-star} $S_v$ consists of~$d$ edges such that $V(S_v)=\{v,u_2^j,\dots,u_k^j:j\in [d]\}$ and $E(S_v)=\{vu_2^j\dots u_k^j:j\in [d]\}$. We call $v$ the \emph{centre vertex} of~$S_v$. The \emph{leaf set} of~$S$ is $N_S(v)$ which is a $(k-1)$-uniform matching of size~$d$.
\end{defn}
\begin{defn}\label{absober consturction}
    Let $G$ be a $k$-graph. For a $k$-ordered vertex tuple $(w_1,\dots ,w_k)$, a \emph{$d$-absorbing tuple} for $(w_1, \dots, w_k)$ consists of~$k-1$ vertex-disjoint $d$-stars $S_{v_2},\dots ,S_{v_k}$ such that $w_1v_2\dots v_k\in E(G)$ and for~$i\in [2,k]$, there is a $d$-star $S_{w_i}$ with centre $w_i$ and $N_{S_{w_i}}(w_i)=N_{S_{v_i}}(v_i)$. In other words, we can replace the centre of each $S_{v_i}$ with~$w_i$ for~$i\in [2,k]$.
    
    For~$F\subseteq V(G)$, we denote by~$A_d(w_1,\dots ,w_k,F)$ the set of all $d$-absorbing tuples for~$(w_1,\dots ,w_k)$ disjoint from~$F$, and by~$A_d(F)$ the union of~$A_d(w_1,\dots ,w_k,F)$ for all tuples $(w_1,\dots, w_k)$.

\end{defn}
Let $T$ be a $k$-loose tree and $\psi:V(T)\rightarrow V(G)$ be an embedding of~$T$ in~$G$. We say that a $d$-star~$S_v$ in~$G$ is \emph{immersed} by~$\psi$ if there exists $x\in V(T)$ such that $\psi(x)=v$ and each edge $e\in T$ with~$x\in e$ is mapped to a distinct edge of~$S_v$. We say that a $d$-absorbing tuple $(S_{v_2},\dots ,S_{v_k})$ is immersed by~$\psi$ if $S_{v_2},\dots, S_{v_k}$ are immersed by~$\psi$.

The next lemma can extend an embedding of a loose tree using immersed absorbing tuples. 

\begin{lemma}[Extension lemma]\label{Absorption lemma}
    Let $n,k,\Delta \in \mathbb{N}$ and $p\in \mathbb{N}\cup \{0\}$. Let $G$ be a $k$-graph on~$n$ vertices. Let $T$ be a $k$-loose tree on~$n-p(k-1)$ vertices with~$\Delta_1(T)\le \Delta$. Suppose that there exists an embedding~$\psi$ of~$T$ in~$G$ and a family $\mathcal{A}$ of pairwise vertex-disjoint $\Delta$-absorbing tuples such that 
    \begin{enumerate}[label=\rm (\roman*)]
        \item\label{absorb1} every absorbing tuple of~$\mathcal{A}$ is immersed by~$\psi$;
        \item\label{absorb2} for distinct $w_1,\dots,w_k\in V(G)$, $\mathcal{A}$ contains at least $p$ many $\Delta$-absorbing tuples for~$(w_1,\dots,w_k)$.
    \end{enumerate}
    Then for any  $k$-loose tree~$T^{\ast}$ on~$n$ vertices containing $T$, there is an embedding of~$T^\ast$.
\end{lemma}
\begin{proof}
   We proceed by induction on~$p$. The lemma holds when $p=0$, since $T=T^\ast$. Thus we assume that~$p\ge 1$.

   Consider any $k$-loose tree~$T^\ast$ on~$n$ vertices containing $T$. Let $x_1\dots x_k\in E(T^{\ast})\backslash E(T)$ with~$x_1\in V(T)$, so $x_2,\dots, x_k\notin V(T)$ and let $T'=T\cup \{x_1\dots x_k\}$. Let $u_1=\psi(x_1)$ and $u_2,\dots, u_k$ be distinct vertices in~$V(G)\backslash V(\psi(T))$. Fix a $\Delta$-absorbing tuple $\S=(S_{v_2},\dots, S_{v_k})$ for~$(u_1,\dots,u_k)$ in~$\mathcal{A}$. Then we define $\psi':V(T')\rightarrow V(G)$ be 
\[        
\psi'(x) =
\begin{cases}
\psi(v_i) &\quad \text{ if }x=x_i\text{ for }i\in [2,k], \\
u_i &\quad \text{ if }x=\psi^{-1}(v_i) \text{ for } i\in [2,k],\\
\psi(x)&\quad \text{ otherwise}.
\end{cases}
\]
We deduce that $\psi'$ is an embedding of~$T'$ to~$G$ with~$V(\psi'(T'))=V(\psi(T))\cup \{u_2,\dots,u_k\}$. Let $\mathcal{A}'=\mathcal{A}\backslash \S$. Furthermore,~\ref{absorb1} and~\ref{absorb2} are satisfied with~$(T',p-1,\psi',\mathcal{A}')$ playing the role of~$(T,p,\psi,\mathcal{A})$.
Since $T^\ast$ contains~$T'$, our induction hypothesis implies that $G$ contains an embedding of~$T^\ast$.
\end{proof}

We now find vertex-disjoint absorbing tuples when $\overline{\delta}_{1}(G)\ge 1/2+\gamma$.
\begin{lemma}\label{lemma:intersect}
    Let ${1}/{n}\ll \zeta \ll \alpha  \ll \beta \ll \gamma
\ll 1/\Delta, 1/k$. Let $G$ be a $k$-graph on~$n$ vertices with~$\overline{\delta}_{1}(G)\ge 1/2+\gamma$. Let $F$ be a subset of~$V(G)$ with size of~$\zeta n$. Then there exists a set $\mathcal{A}$ of vertex-disjoint $\Delta$-absorbing tuples $\{S_{v_2^i},\dots ,S_{v_k^i}\}$ with~$i\le \beta n$ such that for distinct $w_1,\dots,w_k \in V(G)$, we have $|A_{\Delta}(w_1,\dots ,w_k, F)\cap \mathcal{A}|\ge \alpha n$.
\end{lemma}

\begin{proof}
    Consider distinct $w_1,\dots, w_k\in V(G)$. Pick an edge $w_1v_2\dots v_k\in E(G)\backslash (F\cup \{w_2,\dots ,w_k\})$. 
    We further pick $(k-1)\Delta$ disjoint $(k-1)$-vertex subsets $U_2^1\dots U_2^\Delta, \dots ,U_k^1\dots U_k^\Delta$ such that $U_j^i\in N_{G\backslash F}(w_j)\cap N_{G\backslash F}(v_j)$ for~$i\in [\Delta]$ and $j\in [2,k]$. Set $E(S_j)=\{v_j\cup U_j^i~:~i\in [\Delta]\}$ for~$j\in [2,k]$. Then $S_2,\dots ,S_k$ is a $\Delta$-absorbing tuple for~$(w_1,\dots ,w_k)$ disjoint from~$F$. This is possible, since for each $j\in [2,k]$, 
    \begin{align*}
        |N_{G\backslash F}(w_j)\cap N_{G\backslash F}(v_j)|&\ge |N_{G}(w_j)\cap N_{G}(v_j)|-|F|n^{k-2}
        \ge \frac{3\gamma}{2}\binom{n}{k-1}-\zeta n^{k-1} \ge \gamma \binom{n}{k-1}.
    \end{align*} 
    Thus we have that
    \begin{align*}
        |A_{\Delta}(w_1,\dots ,w_k,F)| \ge \l( \gamma n /k^2\r)^{((k-1)\Delta +1)(k-1)}.
    \end{align*}
    Let $c$ be such that $\alpha\ll c\ll \beta $.
    Let $\mathcal{A}'$ be a random subset of~$A_{\Delta}(F)$, where each member is chosen independently at random with probability $p=cn^{-(k-1)^2\Delta-(k-2)}$. Note that each $\Delta$-absorbing tuple contains at most $(k-1)^2\Delta+k-1$ vertices and so $|A_\Delta(F)|\le n^{(k-1)^2\Delta+(k-1)}$. By a Chernoff bound and the union bound, w.h.p we have that $|\mathcal{A}'|\le 2cn\le \beta n$ and, for each $(w_1,\dots ,w_k)$, $|A_{\Delta}(w_1,\dots ,w_k,F)\cap \mathcal{A}'|\ge \beta cn/2$. 

    Let $Y$ be the number of pairs of tuples in~$\mathcal{A}'$ which intersect with at least one vertex. Then 
    $$E[Y]\le 2^{2(k-1)^2\Delta+2(k-1)}n^{2(k-1)^2\Delta+2(k-1)-1}p^2\le 2^{k^2\Delta}c^2 n.$$
   By Markov's inequality, with probability at least $3/4$, we have that $Y\le  2^{k^2\Delta+2}c^2 n$. Fix an outcome of~$\mathcal{A}''$ such that all above events hold. If we remove all pairs of intersecting tuples in~$\mathcal{A}'$, we could get a subset~$\mathcal{A}''$ of~$A_{\Delta}(F)$ of at most $\beta n$ vertex-disjoint tuples such that  
   $$|A_{\Delta}(w_1,\dots ,w_k)\cap \mathcal{A}''|\ge \beta c n/2-2^{k^2\Delta+1}c^2 n\ge \alpha n$$
   for each $w_1 \dots w_k\in \binom{V(G)}{k}$. Finally, let $\mathcal{A}=\mathcal{A}''$ as we desire. 
\end{proof}

Next, we show that we could immerse our absorbers by a small tree.
\begin{lemma}[Immersing lemma]\label{robust covering lemma}
    Let $1/n\ll \zeta \ll \beta \ll \eta \ll \gamma\ll 1/\Delta,1/k$. Let $G$ be a $k$-graph on~$n$ vertices with~$\overline{\delta}_{1}(G)\ge 1/2+\gamma$ and $T$ be a rooted $k$-loose tree at~$r$ on~$\eta n$ vertices with~$\Delta_1(T)\le \Delta$. Let $\mathcal{A}$ be a set of vertex-disjoint $d$-stars in~$G$ with~$|\mathcal{A}|\le k \beta n$ and $v\notin V(\mathcal{A})$.  Then there exists an embedding~$\psi: V(T)\rightarrow V(G)$ such that  $\psi(r)=v$ and every absorbing tuple in~$\mathcal{A}$ is immersed by~$\psi$. 
\end{lemma}

\begin{proof}
    Let $\mathcal{A}=\{S_1, \dots, S_m \}$, so $m\le k\beta n$. Let $T_0=r$ and $\psi_0:V(T_0)\rightarrow \{v\}$. Suppose that for some $i\in [m]\cup \{0\}$, we have found a rooted subtree $T_i$ of~$T$ rooted at~$r$ and an embedding~$\psi_i$ from~$T_i$ to~$G$ such that $S_{i'}$ is immersed by~$\psi_i$ for all $i'\le i$, $S_{i''}$ is disjoint from~$\psi_i(V(T_i))$ for all $i<i'' \le m$ and $|V(T_i)|\le 3(i+1)k\Delta$.
    
    If $i=m$, then we greedily extend $\psi_m$ to an embedding of~$T$. Hence, we may assume that $i<m$. Note that there exists a vertex $x\in V(T)\backslash V(T_i)$ such that the loose path from~$x$ to~$T_i$ has length~$3$. Otherwise, all vertices in~$T$ have distance at most $2$ from~$T_i$. It implies that $$|V(T)|\le |V(T_i)|(1+k\Delta+k^2\Delta^2)\ll \eta n=|V(T)|,$$
    a contradiction. Let $P$ be the path from~$T_i$ to~$x$ in~$T$ of length $3$. Let $T_{i+1}$ be the union of~$T_i$, $P$ and all edges incident to~$x$ in~$T$. Then we have $$|V(T_{i+1})|\le |V(T_i)|+|V(P)|+\Delta\cdot k\le 3(i+2)k\Delta.$$

    Let $w$ be the centre vertex of~$S_{i+1}$, $V(S_{i+1})=\{w, u_2^s,\dots ,u_k^s \text{ for }s\in [\Delta]\}$, $V(P)\cap V(T_i)=z$ and $G'=\l(G\backslash (\psi(T_i)\cup \mathcal{A})\r)\cup \{u_k^1\}$. We pick a loose path $P'$ from~$\psi_i(z)$ to~$u_k^1$ with length $2$ in~$G'$ which is possible as $\overline{\delta}_1(G)\ge 1/2+\gamma$\footnote{Note that $|N(x)\cap N(y)|\ge \gamma n^{r-1}$. Then there exists a $K_{1,2,\dots, 2}$ in~$N(x)\cap N(y)$ and so a loose path of length $2$ between $x$ and $y$.}. Let $P''=P'\cup u_k^1\dots u_2^1w$. Then set $\psi_{i+1}(P)=P''$ and for other edges adjacent to~$\psi^{-1}_{i+1}(w)$, set them to be distinct edges of~$S_{i+1}$ in~$\psi_{i+1}$. Hence $S_{i+1}$ is immersed by~$\psi_{i+1}$.
\end{proof}

\subsection{Embedding for almost spanning tree}  \label{sec:almostembedding}
In this subsection, our aim is to prove the second step in our proof. 

Let $R$ be a $k$-graph and $T$ be a rooted $k$-loose tree at~$r$. For~$v\in V(R)$ and $v\in e\in E(R)$, we say a homomorphism~$\phi$ is a~$(T,R,v,e)$-\emph{homomorphism}, if $\phi(r)=v$ and, for all $r\in e'\in E(T)$,  $\phi(e')=e$.

We show that if $R$ is $\eta$-\emph{robust}, then we can find a homomorphism of loose tree with a fixed rooted vertex.

\begin{lemma}\label{lem:assign}
Let $1/m\ll 1/t\ll \theta \ll \zeta\ll \alpha \ll \eta\ll 1/k, 1/\Delta$. Let $R$ be an $\eta$-robust $k$-graph on~$t$ vertices. Let $\omega$ be a vertex weighting of~$R$ with~$\sum_{v\in V(R)}\omega(v)\ge (1-\eta)t$, $\omega(v)=0$ for all isolated vertices~$v\in V(R)$ and $\omega(v)\in [\theta,1]$ for all other vertices~$v\in V(R)$. Let $e_1,\dots,e_{|E(R)|}$ be the edge enumeration of~$E(R)$ satisfying~\ref{robust2} and $v_1\in e_1$. Let $T$ be a rooted $k$-loose tree at~$r$ with~$\Delta_1(T)=\Delta$ and $|V(T)|\le (1-\alpha-\eta)tm$. Then there exists a $(T,R,v_1,e_1)$-homomorphism~$\phi$ such that for all $v\in V(R)$, $|\phi^{-1}(v)|\le (1-\zeta)\omega(v)m$.
\end{lemma}

We sketch its proof. First, we break $T$ into small pieces (see Claim~\ref{decompose}) $T_1,\dots, T_s$ with similar size such that the root of~$T_i$ is a leaf vertex of some tree~$T_j$ with~$j\le i$. Then by~$\eta$-robust, suppose that there exists a perfect matching $e_1,\dots ,e_{|V_{R}|/k}$ in~$R$ and $e_i$ is reachable from~$e_j$ for~$i\ge j$. Next, we embed each tree one by one into~$e_1,\dots ,e_{|V_{R}|/k}$. We aim to map most of~$T_i$ to~$e_{j(i)}$. In particular, $e_{j(i)}$ will be the first edge (in the ordering) has not been overused. Thus, $j(1), \dots ,j(i)$ form a non-decreasing sequence. Since the root of~$T_i$ is a leaf vertex of some tree~$T_j$ which has been embedded to~$e_{j'}$ with~$j'\le j(i)$. By reachable property~\ref{robust2}, we can extend the map such that most of the vertices in~$T_i$ maps to~$e_{j(i)}$. By rotatable property, we ensure that $T_i$ maps to~$e_{j(i)}$ in a balanced way.

\begin{proof}[Proof of Lemma~\ref{lem:assign}]
 Let $\xi$ with~$1/m\ll \xi\ll1/t$.
We decompose $T$ into small subtrees $T_1,\dots, T_s$. 
 \begin{claim}~\label{decompose}
There is a decomposition of~$T$ into~$T_1,  \dots, T_s$ such that $s \le 1/\xi$, $\xi tm\le |V(T_i)|\le 2k\Delta \xi tm$ for all $i\in [s]$ and for any $i'\in [2,s]$, $T_{i'}$ is a leaf vertex of~$T_j$ for some $j<i'$.
 \end{claim}
 \begin{proofclaim}
      Initially, let $T'=T$. Find a vertex $v\in T'$ such that the rooted tree~$\xi tm\le |V(T(v))|\le k\Delta  \xi tm$ and for each $u\in N_T(v)\cap V(T(v))$, $|V(T(u))|\le \xi tm$. Such vertex can be found by a breadth-first search. Set $T'_1=T(v)$ and update $T'=T\backslash (T'_1\backslash v)$. Repeat this procedure on~$T'$ until $|V(T')|\le k\Delta \xi tm$, we have $T_1',\dots ,T_s'$ with~$s\le \frac{tm}{\xi tm}\le \xi^{-1}$. We replace $T'_s$ with~$T'_s\cup T'$. So $T_1',\dots ,T_s'$ is a decomposition of~$T$. We are done by reversing the order of~$T_1',\dots ,T_s',$ that is $T_i=T'_{s+1-i}$ for~$i\in [s]$.
 \end{proofclaim}
 
 Since $R$ is $\eta$-robust, by~\ref{robust1}, there is a perfect $\omega$-fractional matching $\omega^{\ast}$ such that, for all~$v\in V(R)$, $\omega(v)=\sum_{e\ni v}\omega^{\ast}(e)$ and $\sum_{e\in E(R)}\omega^\ast(e)=\sum_{v\in V(R)}\omega(v)/k$. By scaling each edge weight by a factor of~$(1-2\zeta)$ and removing fractional edges with weight smaller than $1/t^{k+1}$, we could get a new fractional matching $\omega'$ of~$R$ such that $\omega'(e)\ge 1/t^{k+1}$ for all edges~$e \in E(R)$, 
 \begin{align}
 \sum_{e\ni v}\omega'(e)\le (1-2\zeta)\omega(v) \text{ for all $v\in V(R)$},  \label{eqn:omega'(e)}
 \end{align}
 and \begin{align}
        \sum_{e\in E(R)} \omega'(e)&\ge (1-2\zeta)\sum_{v\in V(R)}\omega(v)/k-t^k\cdot \frac{1}{t^{k+1}}\nonumber \\&\ge (1-3\zeta)\sum_{v\in V(R)}\omega(v)/k\ge (1-\eta-\alpha/2) t/k.\label{edge total weight}
    \end{align}
   
    Let $T_{\le 0}$ be the single vertex $r$ and for~$i\in [s]$, let $T_{\le i}$ be $\bigcup_{i'\le i}T_{i'}$. Suppose that for some $i\in [s]$, we have already defined a $(T_{\le i-1},R,v_1,e_1)$-homomorphism~$\phi_{i-1}$ such that, for~$i' \in [i-1]$,
   \begin{enumerate} [label=(\roman*)]
       \item\label{map1} there exists $j(i')$ such that the number of vertices in~$T_{i'}$ not mapped to~$e_{j(i')}$ is at most $k(k\Delta)^{C_1+C_2}$, that is, $|\{v\in V(T_{i'}):\phi_{i-1}(v)\notin e_{j(i')}\}|\le k(k\Delta)^{C_1+C_2};$\
       \item\label{map3} for~$j\in [|E(R)|]$, let $\mathcal{T}_{\le i-1}^j=\{T_{i'}:j(i')=j, i'\le i-1\}$ and we have 
       \begin{align*}
       j(i') = \min\{ j' : \omega'(e_{j'})-|V(\mathcal{T}^j_{\le i'-1})|/km\ge 2k\Delta\xi t\};
       \end{align*}
       \item\label{map2} $1=j(1)\le \dots \le j(i-1)$ and $\mathcal{T}_{\le i-1}^j=\emptyset$ for~$j>j(i-1)$;
       \item\label{map4} for all $e_j=x_1\dots x_k\in E(R)$ and distinct $p,p'\in [k]$, $$||\phi_{i-1}^{-1}(x_p)\cap V(\mathcal{T}^j_{\le i-1})|-|\phi_{i-1}^{-1}(x_{p'})\cap V(\mathcal{T}^j_{\le i-1})||\le 2k\Delta\xi tm+(i-1)k(k\Delta)^{C_1+C_2}.$$
   \end{enumerate}
We now give a rough interpretation of~\ref{map1} to~\ref{map4}. Property~\ref{map1} means most of~$T_{i'}$ is mapped to some edge~$e_{j(i')}$ in~$R$. Note that $\mathcal{T}_{\le i-1}^j$ is all trees $T_{i'}$ mapped to~$e_j$ with~$i'< i$. Property~\ref{map3} says that $T_{i'}$ maps to the first edge $e_j$ that is not full. Property~\ref{map2} says $j(1),\dots,{j(i)}$ is a non-decreasing sequence. Finally, Property~\ref{map4} says that we can ensure that~$T_{i'}$ maps to~$e_{j(i')}$ in a balance way. 

When $i=1$, we map $T_1$ to~$e_1$ and $r$ to~$v_1$. Since $|V(T_1)|\le 2k\Delta \xi t$, our claims~\ref{map1} to~\ref{map4} hold. We now define $\phi_i$ as follow. Let the root vertex of~$T_i$ be $r_i$. Then $V(T_{\le i-1})\cap V(T_i)=\{r_i\}$. Let $\pi \in S_k$ be such that $|C_{\pi (1)}(T_i)|\ge \dots \ge |C_{\pi (k)}(T_i)|$. Let
\begin{align*}
j(i) = \min\{ j' : \omega'(e_{j'})-|V(\mathcal{T}^j_{\le i-1})|/km\ge 2k\Delta\xi t\}.
\end{align*}
Note that $j(i)$ exists, otherwise, we have 
\begin{align*}
    \sum_{j\in [|E(R)|]}(\omega'(e_j)-|V(\mathcal{T}^{j}_{\le i-1})|/km) & < 2\Delta\xi t|E(R)|,\\
    |V(T_{\le i-1})|& > km \sum_{j\in [|E(R)|]} \omega'(e_j) - 2\Delta\xi t km |E(R)|\\
    & \overset{\mathclap{\text{\eqref{edge total weight}}}}{\ge} (1-\eta-\alpha/2)tm - 2\Delta\xi t^{k+1} km 
    \ge (1-\eta -\alpha)tm
    \ge|V(T)|,
\end{align*}
a contradiction. 

Let $e_{j(i)}=y_1\dots y_k$ and $\pi\in S_k$ be such that $$|\phi^{-1}_{i-1}(y_{\pi(1)})\cap V(\mathcal{T}^j_{\le i-1})|\le \dots \le |\phi^{-1}_{i-1}(y_{\pi(k)})\cap V(\mathcal{T}^j_{\le i-1})|.$$
Recall that~\ref{robust2} and~\ref{robust3} say that $e_j$ is $C_1$-reachable from~$v$ and $C_2$-rotatable.
By Proposition~\ref{prop:homo exist}, there exists a homomorphism~$\phi'$ from~$T_i$ to~$R$ such that $\phi'(r_i)= \phi(r_i)$ and, for all $j \in [k]$, 
\begin{align}
    |C_j(T_i) \backslash \phi'^{-1}(y_j)|\le (k\Delta)^{C_1+C_2} . \label{eqn:C_j(T_i)}
\end{align}
Define a homomorphism~$T_{\le i}$ to~$R$ such that 
\[        
\phi_{i}(u) =
\begin{cases}
\phi_{i-1}(u) &\text{ if }u\in V(T_{\le i-1}), \\
\phi'(u) & \text{ if }u\in V(T_i).\\
\end{cases}
\]
Note that from the choice of~$j(i)$, $\phi_i$ satisfies~\ref{map3} and~\ref{map2}.

For distinct $p,p'\in [k]$, we have 
\begin{align*}
    &||\phi_{i}^{-1}(x_p)\cap V(\mathcal{T}^{j(i)}_{\le i})|-|\phi_{i}^{-1}(x_{p'})\cap V(\mathcal{T}^{j(i)}_{\le i})||\\
    \le & \max \left\{ 
        \left||\phi_{i-1}^{-1}(x_p)\cap V(\mathcal{T}^{j(i)}_{\le i-1})|-|\phi_{i-1}^{-1}(x_{p'})\cap V(\mathcal{T}^{j(i)}_{\le i-1})| \right| +  \sum_{j \in [k]} |C_j(T_i) \backslash \phi'^{-1}(y_j)| 
        , |V(T_i)| \right\}\\
    \overset{\mathclap{\text{\ref{map4}, \eqref{eqn:C_j(T_i)}}}}{\le} & \max \{ 2k\Delta\xi tm+(i-1)k(k\Delta)^{C_1+C_2}+k(k\Delta)^{C_1+C_2} , 2k\Delta\xi tm\}
    \le 2k\Delta\xi tm+ik(k\Delta)^{C_1+C_2}.
\end{align*}
Hence,~\ref{map4} holds. 
Thus we obtain a homomorphism~$\phi_s$ from~$T$ to~$R$. 

Consider $j \in [|E(R)|]$ with~$\omega'(e_{j'})>0$, so $\omega'(e_{j'})\ge 1/t^{k+1}$. 
Let $i'$ be maximum such that $j(i') = i$, which exists as $\omega'(e_{j'})\ge 1/t^{k+1} \ge 2\Delta\xi t$. 
Note that 
\begin{align*}
    \omega'(e_{j})-\frac{|V(\mathcal{T}^j_{\le s})|}{km} 
    & = \omega'(e_{j})-\frac{|V(\mathcal{T}^j_{\le i'})|}{km}
     \ge \omega'(e_{j})-\frac{|V(\mathcal{T}^j_{\le i'-1})|}{km} - \frac{|V(T_{i'})|}{km}\\\
   &     \overset{\mathclap{\text{\ref{map3}}}}{\ge} 2k\Delta\xi t - 2\Delta\xi t = 2(k-1)\Delta \xi t .
\end{align*}
for~$x \in e_j$, we have 
\begin{align}
        |\phi_s^{-1}(x)\cap V(\mathcal{T}^j_{\le s})| 
            & \overset{\mathclap{\text{\ref{map4}}}}{\le} \frac{ | V(\mathcal{T}^j_{\le s})|}{k} + 2k\Delta\xi tm+ s k(k\Delta)^{C_1+C_2} \nonumber \\ 
            & \le \omega'(e_{j}) m  +  2\Delta\xi tm+ s k(k\Delta)^{C_1+C_2} \le \omega'(e_{j}) m +\theta^2 m. \label{eqn:phis(x)}
\end{align}
Therefore for each non-isolated vertex $v \in V(R)$, 
\begin{align*}
    |\phi_s^{-1}(v)| 
    & \le \sum_{e_j \in E(R)} |\phi^{-1}(v) \cap V(\mathcal{T}^j_{\le s}) |\\
    & \le  \sum_{v \in e_j \in E(R)} |\phi^{-1}(v) \cap V(\mathcal{T}^j_{\le s}) | + \sum_{i \in [s]} | V(T_i) \setminus \phi^{-1}(e_{j(i)})| \\
    & \overset{\mathclap{\text{\eqref{eqn:phis(x)},~\ref{map1}}}}{\le} 
       \sum_{v \in e_j \in E(R)} (\omega'(e_{j}) m +\theta m)+  s k(k\Delta)^{C_1+C_2}
    \\
    & \le m  \sum_{v \in e_j \in E(R)} \omega'(e_{j})  +  \zeta \theta m
    \overset{\mathclap{\text{\eqref{eqn:omega'(e)}}}}{\le} (1-\zeta)\omega(v)m.
\end{align*}
We are done by setting $\phi = \phi_s$. 
\end{proof}

Now, we are ready to turn this homomorphism on reduced graph to an embedding for almost spanning tree by Lemma~\ref{expand}. Recall that if $d(v,V_2;\dots;V_k) \ge d$, $v$ is $d$-dense into~$\{V_2,\dots , V_{k}\}$.

\begin{lemma}\label{lem:embed}
    Let $1/n\ll 1/t\ll 1/r\ll\varepsilon\ll \theta \ll \zeta \ll \alpha \ll \eta  \ll  d\ll 1/\Delta\ll 1/k$. Let $G$ be a $k$-graph on~$n$ vertices. 
		Let $\mathcal{Q}$ be a $\varepsilon$-regular partition of $G$ and $R$ be the corresponding $(\mathcal{Q},\varepsilon,d)$-reduced graph. Let the cluster of reduced graph $R$ be $V_1,\dots ,V_t$ and vertex set of~$R$ be $[t]$. Let an edge $e_1\in E(R)$ be $1\dots k$. Let $T$ be a rooted $k$-loose tree at~$r$ with~$\Delta_1(T)\le \Delta$ and $|V(T)|\le (1-\alpha - \eta)n$. Then there is a vertex $v^{\ast}\in V_1$ such that the following statement holds:
    
    Let $A\subseteq V(G)$ be such that $|A|\le \eta n$, $|V_i\backslash A|\ge \theta|V_i|$ for~$i\in [t]$ and $A\cap V_i=\emptyset$ for~$i\in [k]$. Suppose that $\phi$ is a $(T,R,1,e_1)$-homomorphism with~$|V_i\backslash A|-|\phi^{-1}(i)|\ge \zeta |V_i\backslash A|$ for~$i\in [t]$. Then there exists an embedding~$\psi$ from~$T$ to~$G$ such that $\psi(r)=v^{\ast}$ and  for any $v\in V(T)$, $\psi(v)\in V_{\phi(v)}\backslash A$. 
\end{lemma}

\begin{proof}
  By the definition of reduced graph $R$, we have that $\{V_1,\dots V_k\}$ is an $\varepsilon$-regular tuple with density at least $d$. Since $d(V_1;\dots; V_k)\ge d$, there exists a vertex~$z$ in~$V_1$ such that~$z$ is $d/2$-dense into~$\{V_2,\dots,V_{k-1}\}$. Let $v^*=z$. Next we show that $v^*$ is the vertex as required.

 Suppose that $\phi $ is a $(T,R,1,e_1)$-homomorphism with~$|V_i\backslash A|-|\phi^{-1}(i)|\ge \zeta |V_i\backslash A|$ for~$i\in [t]$. We now construct the embedding~$\psi$ from~$T$ to~$G$ as following. At first, embed the root vertex $r$ such that $\psi(r)=v^*$. Next we embed all edges in one time according to the breadth-first manner order. In this process, once we embed a vertex, then we remove this vertex from its cluster and update all~$V_i$. Recall that in homomorphism~$\phi$ we use at most $(1-\zeta)|V_i\backslash A|$ vertices for each cluster $V_i$ with~$i\in [s]$. So~$|V_i|$ for~$i\in[t]$ always larger than $\zeta \cdot \theta \cdot |V_i|> \varepsilon |V_i|$. 
    
    Throughout, we will keep the dense property: for every vertex $v\in V(T)$, every edge $e$ with~$v\in e$ and $\phi(e)=\phi(v){i_1}{i_2}\dots i_{k-1}$, $\psi(v)$ is $d/4$-dense into~$\{V_{i_1},V_{i_2},\dots ,V_{i_{k-1}}\}$. 
    
    Since $\phi$ is $(T,R,1,e_1)$-homomorphism and for all edges $e$ with~$r\in e$, $\phi(e)=e_1$. In particular, for the root vertex $r$, $\psi(r)=v^*\in V_1$ and $v^*$ is $d/2$-dense into~$\{V_2,\dots,V_k\}$. Hence $\psi(v^*)=r$ satisfies the dense property. 
    
    For other vertices except $r$, there are at most $\Delta$ $(k-1)$-tuples $\{V_{i_1},V_{i_2},\dots ,V_{i_{k-1}}\}$ such that there exists an edge $e$ with~$v\in e$ and $\phi(e)=\phi(v){i_1}{i_2}\dots i_{k-1}$. 
		Assume that we aim to embed an edge $a_1\dots a_{k}$, where $a_1$ has been embedded. Suppose that $V_{\phi(a_i)}=W_i$ for~$i\in [k]$. 
		Let $m = |V_i|$ for $i \in [k]$. 
		Since~$a_1$ is $d/4$-dense into~$\{W_2,\dots ,W_{k-1}\}$ and $|W_2|,|W_3|,\dots ,|W_k|\ge \zeta m\ge \sqrt{\varepsilon} m$, apply Lemma~\ref{expand} to obtain~$y^{2},y^{3},\dots ,y^{k}$ satisfying the dense property. Thus, we can embed all other vertices in~$T$.
\end{proof}

\subsection{ Complete embedding} \label{sec:complete embedding}
We are ready to prove Theorem~\ref{generaltheorem}.

\begin{proof}[Proof of Theorem~\ref{generaltheorem}]
   Let $$1/n\ll 1/t\ll 1/r\ll\varepsilon \ll \theta\ll \zeta \ll \alpha\ll \beta \ll \eta \ll d\ll \gamma \ll 1/\Delta.$$ 
   Let $G$ be a $k$-graph on~$n$ vertices with~$\overline{\delta_\ell}(G)\ge \delta_{k,\ell}^{R}+\gamma$. Let $\mathcal{Q}$ be a $\varepsilon$-regular partition of $G$ and $R$ be the $(\mathcal{Q},\varepsilon, d)$-regular reduced graph.
   Let $T$ be a rooted $k$-loose tree on~$n$ vertices. 
   \medskip
   
   \noindent \textbf{Step 0: setting up.} 
    Let $x$ be a vertex in~$T$ such that $\eta n/\Delta\le |V(T(x))|\le \eta n$, which exists by a breadth-first search. 
    Let $T_x = T(x)$ and $T'=T\backslash (T_x\backslash x)$ be subtrees rooted at~$x$.
    By adjusting~$\eta$ slightly, we may assume that 
    \begin{align*}
        |V(T_x)| =  \eta n.
    \end{align*}
    Let $T''$ be a $k$-loose tree rooted at~$x$ such that $T'' \subset T$ and 
    \begin{align*}
        |V(T')|-|V(T'')|=(k-1)\alpha n.
    \end{align*}
    (We can obtain~$T''$ from~$T'$ by deleting leaf edge $\alpha n $ times.)

   By Lemma~\ref{regular set up}, there exists a spanning subgraph $R'$ of the reduced graph $R$ with 
   $$\overline{\delta}_\ell^{\alpha}(R')\ge  \delta_{k,\ell}^{R}+\gamma/2.$$
   Let $R^*$ be a spanning $\eta$-robust subgraph of~$R'\subseteq R$. Let $e_1,\dots, e_{|E(R^*)|}$ be an enumeration of edges in~$R^*$ satisfying~\ref{robust2}. Without loss of generality, assume the first edge be $e_1=1\dots k$.

   By Lemma~\ref{lem:embed} with~$(\alpha,T) = ((k-1)\alpha, T'')$, we obtain a vertex $v^*\in V_1$ such that the following holds.
   \begin{enumerate}[label = {\rm($\ast$)}]
    \item \label{itm:embed}
    Let $A\subseteq V(G)$ with~$|A|\le \eta n$, $|V_i\backslash A|\ge \theta|V_i|$ for~$i\in [t]$ and $A\cap V_i=\emptyset$ for~$i\in [k]$. Suppose that~$\phi$ is a $(T'',R,1,e_1)$-homomorphism~$\phi$ with~$|V_i\backslash A|-|\phi^{-1}(i)|\ge \zeta |V_i\backslash A|$ for~$i\in [t]$. Then there exists an embedding~$\psi$ from~$T''$ to~$G \setminus A$ such that $\psi(r)=v^{\ast}$ and, for any $v\in V(T)$, $\psi(v)\in V_{\phi(v)}$.  
   \end{enumerate}

  \medskip \noindent  \textbf{Step 1: find absorbers immersed by~$T_x$.} 
    Let $F \subseteq V(G)$ be such that $V_1, \dots, V_k \subseteq F$ and, for~$i \in [k+1,t]$, $|F \cap V_i| = \theta|V_i|$.
    Note that $|F|\le \sum_{1\le i \le k}|V_i|+t\theta (n/t) \le \zeta n$. 
    By Proposition~\ref{j-degree} and~\eqref{eqn:robustthresholdlowerbound}, we have 
    \begin{align*}
       \overline{\delta}_{1}(G), \overline{\delta}_{1}(G \setminus (F\setminus x) ) \ge   1/2 + \gamma/2.
    \end{align*}
    By Lemma~\ref{lemma:intersect}, there is a set $\mathcal{A}\subseteq A_{\Delta}(F)$ of at most $\beta n$ pairwise vertex-disjoint $\Delta$-absorbing tuples such that for every ordered tuple $(w_1,\dots,w_k )$ in~$V(G)$, we have 
    \begin{align*}
    |A_{\Delta}(w_1,\dots,w_k,F)\cap \mathcal{A}|\ge \alpha n. 
    \end{align*}
    Note that $| V ( \mathcal{A}) | \le k^2 \Delta | \mathcal{A} | \le k^2 \Delta \beta n $.
 
    By Lemma~\ref{robust covering lemma} with~$(G,T) = ( G \setminus (F\setminus x) , T_x)$, we obtain an embedding~$\psi_1$ from~$T_x$ into~$G$ such that~$\mathcal{A}$ is immersed by~$\psi_1$, $\psi_1(x)=v^*\in V_1$ and, by setting $A =  \psi_1^{-1}(T_x \setminus x)$, we have  
    \begin{align*}
        | V_i \backslash A| \ge | (F \setminus x) \cap V_i | = 
        \begin{cases}
            |V_i|  & \text{if $i\in [k]$,}\\
            \theta|V_i| & \text{if $i\in [k+1,t]$.}\\
        \end{cases}
    \end{align*}

  \medskip \noindent \textbf{Step 2: extend $\psi_1$ to embedding of~$T'$.} Define a vertex weighted function~$\omega:[t]\rightarrow [\theta,1]$ such that, for~$i\in [t]$, if $i$ is non-isolated in~$R^*$, then set
  \begin{align*}
  \omega(i)=\frac{|V_i\backslash A|}{|V_i|} \ge  \theta 
  \end{align*}
  and set $\omega(i) = 0 $ otherwise. 
  Recall that $R^*$ is $\eta$-robust.
  By Lemma~\ref{lem:assign} with~$(v_1,T, \alpha) = (1,T'', \alpha)$, there exists a $(T'',R^*,1,e_1)$-homomorphism~$\phi$ such that, for all $i \in [t]$, $|\phi^{-1}(i)|\le (1-\zeta)\omega(i)|V_i|$. 
  Together with~\ref{itm:embed}, we obtain an embedding~$\psi_2$ of~$T''$ to~$G \setminus A$ such that $\psi_2(x)=v^{\ast}$.

  \medskip \noindent \textbf{Step 3: extend to an embedding of~$T$.} 
  Note that $\psi_1 \cup \psi_2$ is an embedding of~$T_x \cup T''$ to~$G$, which immerses~$\mathcal{A}$.
  Note that $|V(T) \setminus V(T_x \cup T'')| = |V(T')|-|V(T'')|=(k-1)\alpha n$.
  Then by Lemma~\ref{Absorption lemma} with~$(p,T,T^*)=(\alpha n,T_x \cup T'', T) $, we obtain an embedding of~$T$ to~$G$.
\end{proof}

\section{Robust spanning subgraph  construction}\label{defn}
In this section, we give the construction of the $\eta$-robust spanning subgraph in the $\alpha$-perturbed $k$-graph. Recall the definition of shadow graph~$\partial_j(G)$ and $\partial_j(A)$ in Section~\ref{Notations}.
\begin{defn}\label{defn:our graph}
    Let $G$ be a $k$-graph on~$V$ and $A\in E(\partial_{k-2}(G))$. Set $C_A(G)$ be the induced graph of largest component in~$L_G(A)$. In our setting, $L_G(A)$ is a $2$-graph and $C_A(G)$ will be uniquely defined. If $G$ is clear from the context, then we use $C_A$ instead of~$C_A(G)$. Let $E_A=\{A\cup uv~|~uv\in E(C_A) \text{ and }A\in E(\partial_{k-2}(G))\}$. We define a spanning subgraph $G^*$ of~$G$ with edge set $E(G^*)=\bigcup_{A\in \partial_{k-2}(G)}E_A$. 
    
    For~$0\le s\le (k-2)$ and $A\in E(\partial_s(G))$, let \begin{align*}
    E_A=\bigcup_{A\cup B\in E(\partial_{k-2}(G))}E_{A\cup B}=\bigcup_{A\subseteq A'\in E(\partial_{k-2}(G))}E_{A'}.
\end{align*} 
\end{defn}

Hence $E(G^*)=E_\emptyset$. By the above definition, we can imply a proposition about link graphs.
\begin{proposition}\label{propo:rotatat reduce}
    Let $G$ be a $k$-graph and $B\in E(\partial_{j}(G))$ with~$j\le k-4$. Then $(L_G(B))^*\subseteq L_{G^*}(B)$
\end{proposition}
\begin{proof}
    Let $e$ be an edge in~$(L_G(B))^*$, then $e\in E_A$ for some $A\in E(\partial_{k-2-j}(B))$. It implies that $e\backslash A\in E(C_A(L_G(B)))$ by Definition~\ref{defn:our graph} and $e\backslash A\in C_{A\cup B}(G)$. Since $|A\cup B|=k-2$, we have $e\cup B\in E(C_{A\cup B}(G))$, implying that $e\in L_{G^*}(B)$. Hence $(L_G(B))^*\subseteq L_{G^*}(B)$.
\end{proof}
The next lemma states that $G^*$ is $\eta$-robust. Recall the definition of~$\eta$-robust in Definition~\ref{robust}.
\begin{lemma}\label{lem:k-2thereshold}
    Let $1/n\ll \alpha \ll \eta \ll \gamma \ll1/k \le 1/4$. Let $G$ be a $k$-graph on~$n$ vertices with~$\overline{\delta}^\alpha_{k-2}(G)\ge 1/2+\gamma$. Then its spanning subgraph $G^*$ is $\eta$-robust.
\end{lemma}

This will become immediately by Theorems~\ref{Matching existence},~\ref{reach} and~\ref{main rotatable lemma}. Our main result, Theorem~\ref{Theorem Matching threshold} immediately follows from Theorem~\ref{generaltheorem} and Lemma~\ref{lem:k-2thereshold}.
\begin{proof}[Proof of Theorem~\ref{Theorem Matching threshold}]
    By Theorem~\ref{generaltheorem}, it is sufficient to show that $\delta_{k,k-2}^{R}\le 1/2$. Let $1/n\ll \eta\ll \gamma$. Let $G$ be a $k$-graph on~$n$ vertices with~$\overline{\delta}^\alpha_{k-2}(G)\ge 1/2+\gamma$.
    Let $G^*$ be the corresponding spanning graph of~$G$ as in Definition~\ref{defn:our graph}. By Lemma~\ref{lem:k-2thereshold}. $G^*$ is $\eta$-robust. Hence, $\delta_{k,k-2}^{R}\le 1/2$.
\end{proof}

\section{Robust fractional matching}\label{matching}
In this section, our aim is to prove Lemma~\ref{Matching existence}, the existence of robust fractional matching for~$G^*$. 
\begin{lemma}\label{Matching existence}
      Let $1/n\ll \alpha \ll \eta \ll\gamma\ll 1/k\le 1/4$. Let $G$ be a $k$-graph on~$n$ vertices with~$\overline{\delta}^\alpha_{k-2}(G)\ge 1/2+\gamma$. Suppose that $\omega:V(G)\rightarrow [0,1]$ with~$\sum_{v\in V(G)}\omega(v)\ge (1-\eta)n$ and $\omega(v)=0$ for all isolated vertices~$v\in V(G)$. Then there exists a perfect $\omega$-fractional matching $\omega^{\ast}$ on~$G^*$. 
\end{lemma}

We need the following well-known results.
\begin{theorem}[Erd{\H{o}}s and Gallai~\cite{erdos1961minimal}]\label{theo:EG}
    Any graph $G$ on~$N$ vertices with 
    \begin{align*}
        e(G)>\max \{\binom{2k-1}{2},\binom{k-1}{2}+(k-1)(N-k+1)\}
    \end{align*}
    admits a matching of size~$k$.
\end{theorem}

\begin{theorem}[Karamata~\cite{karamata1932inegalite}]\label{Inequal}
    Let $\mathbf{a}=(a_i)_{i=1}^n$ and $\mathbf{b}=(b_i)_{i=1}^n$ be (finite) sequences of real numbers from an interval $(\alpha,\beta)$. Let $f:(\alpha,\beta)\rightarrow \mathbf{R}$ be a convex function. If the sequence $\mathbf{a}$ majorizes $\mathbf{b}$, that is $\sum_{i=1}^ja_i\ge \sum_{i=1}^jb_i$ for~$j\in [n]$, then the following holds $$\sum_{i=1}^nf(a_i)\ge \sum_{i=1}^nf(b_i).$$
\end{theorem}
Now we state a key structural lemma about $C_A$ for all $A\in E(\partial_{k-2}(G))$. 

\begin{lemma}\label{lem:sturc1}
    Let $1/n\ll \gamma\ll 1$. Let $G_1$, $G_2$ and $G_3$ be $2$-graphs on vertex set $V$ with~$|V|=n$ and $e(G_1),e(G_2),e(G_3)>(1/2+\gamma)\binom{n}{2}$. Let $C_1$, $C_2$ and $C_3$ be the induced subgraphs of the largest connected components in~$G_1$, $G_2$ and $G_3$, respectively. Then the following statements hold for all $i\in [3]$
    \begin{enumerate}[label=\rm{(M\arabic*)}]
        \item\label{Matching 1}$|V(C_i)|> (1/2+\gamma) n$ and $e(C_i)> (1/2+\gamma)\binom{n}{2}-\binom{n-|V(C_i)|}{2}\ge (1/4+\gamma)\binom{n}{2}$;
        \item\label{Matching 2}$C_i$ contains a matching of size~$\l(1/4+\gamma/3\r)n$;
        \item\label{Matching 3}$C_i$ contains a triangle;
        \item\label{Matching 4} there exist distinct $i_1,i_2\in [3]$ such that $E(C_{i_1})\cap E(C_{i_2})\neq \emptyset$.
    \end{enumerate}

\end{lemma}
\begin{proof}
     Suppose for a contradiction that $G_i$ has no component with at least $(\frac{1}{2}+\gamma) n$ vertices. Then we could write $V=V_1\cup V_2\cup \dots \cup V_s$, such that $V_j$ are disjoint connected components of~$G_i$ with~$|V_j|\le (1/2+\gamma) n$. Let $f(x)=x(x-1)/2$ for~$x\ge 0$. So $f$ is a convex function. Note that $((1/2+\gamma)n,(1/2-\gamma)n,0,\dots ,0)$ majorizes the sequence of components size~$(|V_1|,\dots, |V_s|)$ . By Theorem~\ref{Inequal}, $$e(G_i)\le \sum_{j\in [s]}\binom{|V_j|}{2}\le \binom{(\frac{1}{2}+\gamma)n}{2}+\binom{(\frac{1}{2}-\gamma)n}{2}\le \l(\frac{1}{2}+\gamma\r)\binom{n}{2},$$
     a contradiction. Hence, $|V(C_i)|> (1/2+\gamma)\binom{n}{2}$. Moreover,  
    $$e(C_i)=e(G_i)-e(G_i\backslash C_i)> \l(\frac{1}{2}+\gamma\r)\binom{n}{2}-\binom{\l(\frac{1}{2}-\gamma\r)n}{2}\ge \l( \frac{1}{4}+\gamma \r) \binom{n}{2}.$$
    Thus,~\ref{Matching 1} holds.

    Let $|V(C_i)|=(1-x_i)n$. Note that $x_i\le\l( 1/2-\gamma \r)$ and there are at most $\binom{x_in}{2}$ edges not in~$C_i$. Hence we have
    \begin{align*}
         e(C_i)&\ge \l(\frac{1}{2}+\gamma\r)\binom{n}{2}-\binom{x_in}{2}
         \ge \l(\frac{1}{2}+\gamma-x_i^2\r)\frac{n^2}{2}\\
         = &\l(\frac{1}{4}+\gamma\r)^2\frac{n^2}{2}+\l(\frac{3}{16}-x_i^2+\frac{\gamma}{2}-\frac{\gamma^2}{4}\r)\frac{n^2}{2}
         \ge \binom{(\frac{1}{4}+\gamma)n}{2}+\l(\frac{3}{16}-\frac{x_i}{2}+\frac{\gamma}{4}\r)\frac{n^2}{2}\\
         >&\binom{(\frac{1}{4}+\frac{\gamma}{3})n}{2}+\l((\frac{1}{4}+\frac{\gamma}{3})n-1\r)\l((1-x_i)n-(\frac{1}{4}+\frac{\gamma}{3})n+1\r),
    \end{align*}
    where the second inequality holds as $x_i\le 1/2$. Also, $e(C_i)>(1/4+\gamma)\binom{n}{2}\ge \binom{(2(1/4 + \gamma/3)n}{2}$. By Theorem~\ref{theo:EG}, there exists a matching in~$G_i$ of size~$(1/4 + \gamma/3)n$ implying~\ref{Matching 2}.
    Note that $$e(C_i)\ge \l(\frac{1}{2}+\gamma\r)\binom{n}{2}-\binom{x_i n}{2}> \frac{(1-x_i)^2n^2}{4}=\frac{|V(C_i)|^2}{4}.$$
    Thus,~\ref{Matching 3} holds by Mantel's~Theorem.

    Suppose that~\ref{Matching 4} is false, so $E(C_1),E(C_2),E(C_3)$ are pairwise disjoint. Without loss of generality, assume that $x_1\le x_2\le x_3$. By~\ref{Matching 1},  
    \begin{align}
        e(C_1)+e(C_2)+e(C_3)&> \sum_{i\in [3]}\l( \left(\frac{1}{2}+\gamma \right) \binom{n}{2}-\binom{x_in}{2}\r)
        \ge
        \l(\frac{3}{2}+3\gamma-x_1^2-x_2^2-x_3^2\r)\binom{n}{2}.\label{inequal1}
    \end{align}
    On the other hand, since there is no common edge, the inclusion-exclusion principle implies that,
    \begin{align*}
        e(C_1)+e(C_2)+e(C_3)&\le \sum_{i\in [3]}\binom{|V(C_i)|}{2}-\sum_{i,j\in [3],i\neq j}\binom{|V(C_i\cap C_j)|}{2}+\binom{|V(C_1\cap C_2\cap C_3)|}{2}\\
        &\le \sum_{i\in [3]} \binom{|V(C_i)|}{2}-\sum_{i,j\in [3],i\neq j}\binom{|V(C_i\cap C_j)|}{2}+\binom{|V(C_1\cap C_2)|}{2}\\
        &=\sum_{i\in [3]} \binom{|V(C_i)|}{2}-\binom{|V(C_1\cap C_3)|}{2}-\binom{|V(C_2\cap C_3)|}{2}\\
        &\le\binom{n}{2}\l( \sum_{i\in [3]}(1-x_i)^2-(1-x_1-x_2)^2-(1-x_1-x_3)^2+\gamma\r)\\
        &=\binom{n}{2}\l (1+2x_1-x_1^2-2x_1x_2-2x_1x_3+\gamma \r).
    \end{align*}
Together with (\ref{inequal1}) and the fact that $x_1\le x_2\le x_3<1/2$, we have
\begin{align*}
    0&<\l (1+2x_1-x_1^2-2x_1x_2-2x_1x_3+\gamma \r)-\l(\frac{3}{2}+3\gamma-x_1^2-x_2^2-x_3^2\r)\\
    &=(x_2-x_1)^2+(x_3-x_1)^2-\frac{(1-2x_1)^2}{2}
    \le 2(x_3-x_1)^2-\frac{(1-2x_1)^2}{2}\\
    &=\frac{1}{2}(1-2x_3)(4x_1-2x_3-1)\le -\frac{1}{2}(1-2x_3)^2<0,
    \end{align*}
a contradiction. Hence~\ref{Matching 4} holds. 
\end{proof}
By this structural lemma, we immediately deduce that there exists a fractional matching with size~$n/4$ in~$C_{A}$ for all $A\in E(\partial_{k-2}(G))$.
\begin{proposition}\label{fractional matching existence}
    Let $1/n\ll \alpha \ll \eta \ll\gamma\ll 1/k\le 1/4$. Let $G$ be a $k$-graph on~$n$ vertices with~$\overline{\delta}^\alpha_{k-2}(G)\ge 1/2+\gamma$ and $\omega:V(G)\rightarrow [0,1]$ with~$\sum_{v\in V(G)}\omega(v)\ge (1-\eta)n$ and $\omega(v)=0$ for all isolated vertices~$v\in V(G)$. Let $A\in E(\partial_{k-2}(G))$. Then there is an $\omega$-fractional matching with size at least $n/4$ in~$C_{A}$.
\end{proposition}
\begin{proof}
    By Lemma~\ref{lem:sturc1}~\ref{Matching 2}, there exists a matching $M$ of size~$(1/4+\gamma/3)n$ in~$C_A$. Define a fractional matching $\omega^*:E(C_A)\rightarrow [0,1]$ such that $$\omega^*(e)=\min_{v\in e} \{\omega(v)\} \text{ for all $e\in E(M)$}.$$
    So $\omega^*$ is a $\omega$-fractional matching with size~$$\sum_{e\in E(C_A)}\omega^*(e)\ge \l(\frac{1}{4}+\frac{\gamma}{3}\r)n-2\eta n\ge \frac{1}{4}n.$$
\end{proof}
To prove the existence of perfect $\omega$-fractional matching in~$G^*$, we will use the following result, which is a consequence of linear programming.

\begin{proposition}[{\cite[Proposition~2.11]{lang2022minimum}}]\label{prop:link matching}
    Let $G$ be a $k$-graph and a weight function~$\omega:V(G)\rightarrow [0,1]$ with~$\omega(v)=0$ for isolated vertices. Suppose that there exists an $m\le (\sum_{v\in V(G)}\omega(v))/k$ such that for every non-isolated vertex $v\in V(G)$, the link graph $L_G(\{v\})$ has an $\omega$-fractional matching with size~$m$. Then $G$ has an $\omega$-fractional matching with size~$m$.
\end{proposition}  

We now prove Lemma~\ref{Matching existence}. 

\begin{proof}[Proof of Lemma~\ref{Matching existence}]
    Let $m=\sum_{v\in V(G)}\omega(v)$. Let $0\le j\le k-2$ and $S\in E(\partial_j(G))$. We claim that in~$L_{G^*}(S)$, there exists an $\omega$-fractional matching with size~$m/k$. We prove this claim by induction on~$|S|$ from~$k-2$ to~$0$.

    First if $|S|=k-2$, then our claim follows from Proposition~\ref{fractional matching existence}. For~$j=|S|<k-2$, let $x$ be a neighbour of~$S$ in~$\partial_{j+1}(G)$. So in~$L_{G^*}(S\cup \{x\})$, by induction hypothesis, there exists a fractional matching with size~$m/k$. Let $I_S$ be the set of isolated vertices in~$L_{G^*}(S)$. If $S\neq \emptyset$, then $j>0$. By  Definition~\ref{defn:perturbed degree}, the degree of~$S$ in~$\overline{\partial_{j+1}(G)}$ is at most $\alpha n$. Then $|I_S|\le \alpha n$ and $$\frac{\sum_{v\in V(G^*)\backslash I_S}\omega(v)}{k-j}=\frac{m-\sum_{v\in I_S}w(v)}{k-j}\ge  \frac{m-\alpha n}{k-j}\ge \frac{m}{k}.$$
    The above inequality follows $m\ge (1-\eta)n$. If $S=\emptyset$, then $\sum_{v\in V(G)\backslash I_S}\omega(v)/k=m/k$.
    By Proposition~\ref{prop:link matching} with~$(H,k)=(L_{G^*}(S)\backslash I_S,k-j)$, there exists an $\omega$-fractional matching with size~$m/k$ in~$L_{G^*}(S)$. Thus, our claim holds and the lemma follows $S=\emptyset$ case.
\end{proof}

\section{Reachable Lemma}\label{reachable }
Let $C\in \mathbb{N}$, let $G$ be a $k$-graph and $e=u_1\dots u_k\in E(G)$. Recall that $e$ is $C$-reachable from~$u$ if for any rooted $k$-loose tree~$T$ at~$r$, there exists a homomorphism~$\phi$ from~$T$ to~$G$ such that $\phi(r)=u$ and $\phi(v)\in e$ for~$v$ with~$\dist(r,v)>C$. Moreover, for an edge $e'\in E(G)$, if $e$ is $C$-reachable from all $u\in e'$, then we say $e$ is $C$-reachable from~$e'$. 

Recall the definition of~$G^*$ in Definition~\ref{defn:our graph}. 
Our aim in this section is to prove  the existence of the edge enumeration of~$E_{\emptyset}$ with tree reachable proposition. 
 \begin{lemma}[Reachable lemma]\label{reach}
    Let $1/n\ll \alpha\ll \gamma \ll 1/k\le 1/4$. Let $G$ be a $k$-graph on~$n$ vertices with~$\overline{\delta}^\alpha_{k-2}(G)\ge 1/2+\gamma$. Then there is an enumeration~$e_1,\dots ,e_m$ of all edges in~$G^*$ such that $e_j$ is $n^{4k}$-reachable from~$e_i$ for~$i\le j\le m$.
\end{lemma}
In the next subsection, we give an overview of a key step in our proof and the definition of~$(\mathcal{A},\mathcal{K})$-reachable. In Section~\ref{subsec:pre}, we list some properties about reachability. In Section~\ref{reach general}, we give an auxiliary colouring method of shadow graphs to partition edges. In Section~\ref{(A,k)-reach}, we show a key lemma for proving Lemma~\ref{reach}. In Section~\ref{reach odd}, we prove Lemma~\ref{reach}.

\subsection{Sketch of Proof of Lemma~\ref{reach} and $(\mathcal{A},\mathcal{K})$-reachable}\label{Pre Reach}
We give an overview of one step in the proof of Lemma~\ref{reach} when $k$ is even.

Consider $A\in E(\partial_{k-2}(G))$. 
Note that all edges in~$E_A$ are reachable from themselves since $E_A$ is tight connected. 
This leads to an equivalent relationship on~$E(\partial_{k-2}(G))$ such that $A\sim A'$ if and only if $E_A\cup E_{A'}$ are reachable within themselves. This can be represented by a colouring of~$E(\partial_{k-2}(G))$.

When $k=4$, this is just an edge-colouring of~$\delta_{2}(G)$. It turns out that this is Gallai locally $2$-edge colouring. By analysing monochromatic components, we obtain the desired enumeration of~$E_{\emptyset}$.
In particular, we can partition~$E(\partial_{k-2}(G))$ into those $A$ such that each edge in~$E_A$ is reachable from all edges in~$E_{\emptyset}$. Such edges can be placed at the end of the enumeration of~$E_{\emptyset}$. We call edges in these $E_A$ ``good'' otherwise ``bad''. It turns out that the set of bad edges are reachable themselves.

We now discuss the case when $k=6$.  Consider $A\in E(\partial_{k-4}(G))=E(\partial_2(G))$. Then $E_A$ can be viewed as a $4$-graph. By the discussion above, we can partite $E_A$ into~$A$-good or $A$-bad edges. This gives us an ordering within~$E_A$. We would like to extend these into a total ordering of~$E_{\emptyset}$.  This can be done via defining a similar colouring on~$E(\partial_2(G))$ as in the case when $k=4$. 

Similarly, we extend this argument for larger $k$. Namely, we first look at~$E_A$ with~$|A|=k-2$, then $|A|=k-6$ and so on. In order to track the edges as in the overview, we need the following notations. 

Intuitively, the non-$A$-label edges are the $A$-good edges, while the $A$-label edges are the $A$-bad. For technical reason, we also define $K_A$
which is a $2$-graph, more importantly, $$K_A\cap E(\partial_2(e))\neq \emptyset~~~~\text{for all $A$-label edge } e .$$
which is needed to help us define $A$-label edges formally.

 Let $G$ be a $k$-graph on vertex set~$V$. Let $\partial_{\mod 2}(G)=\bigcup_{s\equiv k \mod 2, s< k}E(\partial_{s}(G))$. Let $\mathcal{A}$ be an upward closed subset of~$\partial_{\mod 2}(G)$. For~$A\in \mathcal{A}$, recall that $\partial_2(A)=\partial_2(L_G(A))$ and let $K_A$ be an induced subgraph of~$\partial_2(A)$ without isolated vertices. We write $\mathcal{K}=\{K_A~|~A\in \mathcal{A}\}$. Throughout this section, we further assume that $E(K_A)=E(\partial_2(A))$ if $A\in E(\partial_{k-2}(G))$. 

We say that $e\in E_A$ is \emph{$A$-label} with respect to~$(\mathcal{A},\mathcal{K})$ if $e$ can be written as $v_1v_2\dots v_k$ such that $A=v_1\dots v_{|A|}$ and $v_{s+1}v_{s+2}\in E(K_{v_1\dots v_s})$ for all $s$ with~$|A|\le s< k$ and $s\equiv k\mod 2$. Let $E^+_A$ be the set of all $A$-label edges.

We say that $e\in E_A$ is \emph{non-$A$-label} with respect to~$(\mathcal{A},\mathcal{K})$ if $e$ can be written as $v_1v_2\dots v_k$ such that $A=v_1\dots v_{|A|}$ and $v_{s+1}v_{s+2}\notin E(K_{v_1\dots v_s})$ for some $s$ with~$|A|\le s< k$ and $s\equiv k\mod 2$. Let~$E_A^{-}$ be the set of all non-$A$-label edges.

We remark that one edge can be both non-$A$-label and $A$-label which depends on the order of vertices in the edge. We will show that non-$A$-label edges are reachable from any edges in~$E_{\emptyset}$, while $A$-label edges are only reachable within themselves. It is captured in the following definition.
\begin{defn}\label{defn:clique}
    Let $1/n\ll \alpha \ll 1/k\le 1/4$. Let $G$ be an $\alpha$-perturbed $k$-graph on~$n$ vertices. We say $G$ is \emph{$(\mathcal{A},\mathcal{K})$-reachable} if
\begin{enumerate}[label={(A\arabic*)}]
    \item \label{Reach1} $\mathcal{A}$ is an upward closed subset of~$\partial_{\mod 2}(G)$;
    \item \label{Reach2} $\mathcal{K}=\{\text{$K_A$ is an induced subgraph of~$\partial_2(A)$ without isolated vertices}~|~A\in \mathcal{A}\} $ and $E(K_A)=E(\partial_2(A))$ for all $A\in E(\partial_{k-2}(G))$;
    \item \label{Reach3} $|V(K_A)|\ge (1/2+3\alpha)n$ for all $A\in \mathcal{A}$;
    \item \label{Reach4}if $E(K_A)= E(\partial_2(A))$, then every  $e\in E_A$ is $n^{3k+1-2|A|}$-reachable from all $e'\in E_A$;
    \item \label{Reach5}if $E(K_A)\neq E(\partial_2(A))$, then 
    \begin{enumerate}[label=(A5.\arabic*)]
        \item \label{Reach42}there is an edge $e\in E_A$ such that $e$ is not $n^{3k+1-2|A|}$-reachable for some $e'\in E_{A}\subseteq E_{\emptyset}$;
        \item \label{Reach43}every $e\in E_A^+$  is $n^{3k-2|A|}$-reachable from any $e'\in E_A^+$;
        \item \label{Reach44}every $e\in E_A^-$ is $n^{3k-2|A|}$-reachable from all edges $e'\in E_{\emptyset}$.
    \end{enumerate}
\end{enumerate}
\end{defn}
One key ingredient of proving Lemma~\ref{reach} is that $G$ is $(\partial_{\mod 2}(G),\mathcal{K})$-reachable. 
\begin{lemma}\label{Lem:reach even case}
    Let $1/n\ll \alpha \ll \gamma\ll 1/k\le 1/4$. Let $G$ be a $k$-graph on~$n$ vertices with~$\overline{\delta}^\alpha_{k-2}(G)\ge 1/2+\gamma$. Then $G$ is $(\partial_{\mod 2}(G),\mathcal{K})$-reachable for some $\mathcal{K}$.
\end{lemma}
Note that this implies Lemma~\ref{reach} when $k$ is even.
\begin{proof}[Proof of Lemma~\ref{reach} for even $k$]
    By Lemma~\ref{Lem:reach even case}, $G$ is $(\mathcal{A},\mathcal{K})$-reachable with~$\mathcal{A}=\partial_{\mod 2}(G)$. We enumerate edges in~$E(G^*)=E_{\emptyset}$ such that if $e_i\in E_{\emptyset}^+$ and $e_j\in E_{\emptyset}\backslash E_{\emptyset}^+$, then $i<j$. Consider $e_i,e_j$ with~$i\le j$. If $e_i\in E_{\emptyset}^+$, by~\ref{Reach4}, ~\ref{Reach43} and~\ref{Reach44}, $e_j$ is $n^{4k}$-reachable from~$e_i$. If~$e_i\in E_{\emptyset}\backslash E_{\emptyset}^+\subseteq E_{\emptyset}^-$, then $e_j\in E_{\emptyset}^-$ as well. So $e_i$ is $n^{4k}$-reachable from~$e_j$ by~\ref{Reach4} or~\ref{Reach44}. Hence our enumeration is as required. 
\end{proof}

We now sketch the proof of Lemma~\ref{Lem:reach even case}. We will see that $G$ is $E(\partial_{k-2}(G), \mathcal{K})$-reachable (see Proposition~\ref{k-2 reachable}). Let $\mathcal{A}\subseteq \partial_{\mod 2}(G)$ be a maximal set such that $G$ is $(\mathcal{A}, \mathcal{K})$-reachable for some~$\mathcal{K}$. If $\mathcal{A} =  \partial_{\mod 2}(G)$, then we are done. 
Let $A$ be the maximal element of~$\partial_{\mod 2}(G)\backslash \mathcal{A}$. We aim to find a suitable $K_A$ such that $G$ is $(\mathcal{A}\cup A,\mathcal{K}\cup K_A)$-reachable (and repeat this until $\mathcal{A}=\partial_{\mod 2}(G)$). 

We know that for each $u_1v_1\in E(\partial_2(A))$, all edges in~$E_{A\cup u_1v_1}^+$ are reachable within themselves. Next, we characterize edges between different $E_{A\cup u_1v_1}^+$ and $E^+_{A\cup u_2v_2}$ by an edge-colouring~$\phi_A$ of~$\partial_2(A)$ such that $\phi_A(u_1v_1)=\phi_A(u_2v_2)$  if $E^+_{A\cup u_1v_1}\cup E^+_{A\cup u_2v_2}$ are reachable within themselves (see Section~\ref{reach general}). Clearly, if $\partial_2(A)$ is monochromatic, we can set $K_A=\partial_2(A)$. 
Next we identify the structure of this edge colouring of~$\partial_2(A)$. We will show that this edge colouring is a Gallai locally $2$-edge-colouring (see Corollary~\ref{Gallai coloured prop}) and then analysis the reachable properties of each monochromatic component depending on its order (see Lemmas~\ref{large component},~\ref{small component},~\ref{middle component}). Finally, we define the $K_A$ according to the edge colouring.

Note that to prove Lemma~\ref{reach} when $k$ is odd, we need some additional (but similar) arguments (see Section~\ref{reach odd}).

\subsection{Preliminaries and common sets}~\label{subsec:pre}
First, we list some basic properties about reachable.
\begin{proposition}
    Let $C\in \mathbb{N}$. Let $G$ be a $k$-graph and $e=u_1\dots u_k$, $e'=u_1'\dots u_k'$ in~$E(G)$. Then the following holds.
    \begin{enumerate}[label={\rm (\roman*)}]\label{prop:reachablefact1}
        \item $e$ is $1$-reachable from all $u\in e$ and so from~$e$; \label{itm:reachablefact1}
        \item if $e''\in E(G)$ is $C_1$-reachable from~$e'$ and $e'$ is $C_2$-reachable from~$e$, then $e''$ is $(C_1+C_2)$-reachable from~$e$;\label{itm:reachablefact2}
        \item if $e'$ is $C$-reachable from~$u_i$ for all $i\in [k]\backslash \{i_0\}$ and some $i_0$, then $e'$ is $(C+1)$-reachable from~$e$;\label{itm:reachablefact3}
        \item if $|e\cap e'|=k-1$, then $e'$ is $1$-reachable from~$e$;\label{itm:reachablefact4}
        \item if $e'$ and $e$ is tight connected, then $e'$ is $n^k$-reachable from~$e$;\label{itm:reachablefact5}

    \end{enumerate}
\end{proposition}
\begin{proof}
    Note that~\ref{itm:reachablefact1} and~\ref{itm:reachablefact2} follow from the definition of reachable. 
 
 Let $T$ be a rooted $k$-loose tree at~$r$. For~\ref{itm:reachablefact3}, without loss of generality, we may assume that $i_0=1$. For all $j\in [2,k]$, let $\mathcal{T}_j$ be all rooted subtrees $T(v)$ with~$ v \in C_{j}(T)$ and $\dist(r,v)=1$. 
By Fact~\ref{partition tree}, $r, V(\mathcal{T}_2), \dots, V(\mathcal{T}_{k})$ partition~$V(T)$. 
Consider $ j\in [2,k]$ and $T'\in \mathcal{T}_j$ rooted at~$r'$. By our assumption, there exists a homomorphism~$\phi_{T'}$ from~$T'$ to~$G$ such that, for all $ v \in V(T')$,
\begin{align*}
	\phi_{T'}(v) 
	\begin{cases}
		=u_{j}~~ &\text{if $v=r'$},\\
		  \in e~~& \text{if $\dist(r',v)\ge C$}.
	\end{cases}
\end{align*}
 Define $\phi : V(T) \rightarrow V(G)$ to be such that 
\begin{align*}
	\phi(v) = 
	\begin{cases}
		u_{i_0} & \text{if $v =r$,}\\
		\phi_{T'}(v)  & \text{if $ v \in V(T')$ for some $T' \in \bigcup_{j \in [2,k]}\mathcal{T}_j$}.
	\end{cases}
\end{align*}
So $\phi$ is the desire homomorphism.

For~\ref{itm:reachablefact4}, without loss of generality, suppose that $(u_2,\dots,u_k)=(u_2',\dots ,u_k')$. By~\ref{itm:reachablefact1}, it is enough to show that $e'$ is $1$-reachable from~$u_1$. Set homomorphism~$\phi$ such that \begin{align*}
    	\phi(v) =
	\begin{cases}
		u_1~~ &\text{if $v=r$},\\
		  u_i'~~& \text{if $v\in C_i(T)\backslash\{r\}$ for all $i\in [k]$}.
    \end{cases}
\end{align*}
Since the length of tight path in~$G$ is at most $n^k$,~\ref{itm:reachablefact5} follows from~\ref{itm:reachablefact2} and~\ref{itm:reachablefact4}.
\end{proof}

Next, we show that the definition of~$(\mathcal{A},\mathcal{K})$-reachable is well-defined when $\overline{\delta}^{\alpha}_{k-2}(G)>1/2$.
\begin{proposition}~\label{k-2 reachable}
    Let $1/n\ll\alpha\ll \gamma \ll 1/k\le 1/4$. Let $G$ be a $k$-graph on~$n$ vertices with~$\overline{\delta}^{\alpha}_{k-2}(G)>1/2+\gamma$. Then $G$ is $(E(\partial_{k-2}(G)), \mathcal{K})$-reachable.
\end{proposition}
\begin{proof}
For all $A\in E(\partial_{k-2}(G))$, we set $K_A$ to be $\partial_2(A)$ after removing all isolated vertices. So~\ref{Reach2} holds. Note that~\ref{Reach1} holds trivially. By Definition~\ref{defn:perturbed degree}, $|K_A|\ge (1/2-\alpha)n,$ so~\ref{Reach3} holds. Since $|A|=k-2$ and $\overline{\delta}^{\alpha}_{k-2}(G)>1/2+\gamma$, then by Definition~\ref{defn:our graph}, $E_A$ is tight connected. Thus,~\ref{Reach4} follows Proposition~\ref{prop:reachablefact1}~\ref{itm:reachablefact5}. Hence,  $G$ is $(E(\partial_{k-2}(G), \mathcal{K})$-reachable.
\end{proof}

Then from Definition~\ref{defn:clique}, we can imply following properties about $(\mathcal{A},\mathcal{K})$-reachable.

\begin{proposition}\label{A,Kreachable fact}
    Let $1/n\ll \alpha\ll 1/k\le 1/4$. Let $G$ be an $\alpha$-perturbed $k$-graph on~$n$ vertices. Suppose that $G$ is $(\mathcal{A},\mathcal{K})$-reachable. Then the following hold for~$A,A_1,A_2\in \mathcal{A}$
    \begin{enumerate}[label={\rm (\roman*)}]
        \item \label{itmA,Kreachable fact1}  $E_A^+\cup E_A^-=E_A$;
        \item \label{itmA,Kreachable fact2}  if $|A|<k-2$, then $E_A^+=\bigcup_{uv\in E(K_A)}E_{A\cup uv}^+$ and $E_A^-=\bigcup_{uv\in E(K_A)}E_{A\cup uv}^-\cup \bigcup_{uv\notin E(K_A)}E_{A\cup uv}$;
        \item \label{itmA,Kreachable fact3} if $|A|<k-2$ and $u\in V(K_A)$, then there exists an edge $e\in E_A^+$ with~$u\in e$;
        \item \label{itmA,Kreachable fact7} every $e\in E_A^-$ is $n^{3k-2|A|}$-reachable from all edges in~$E_{\emptyset}$; 
        \item \label{itmA,Kreachable fact4} every $e\in E_A$ is $n^{3k-2|A|+1}$-reachable from any $e'\in E_A^+$;
        \item \label{itmA,Kreachable fact5} if $E(K_A)\neq E(\partial_2(A))$, then every $e\in E_A^+$ is not $4n^{3k-2|A|}$-reachable from any $e'\in E_A^-$;
        \item \label{itmA,Kreachable fact6} if some edge $e_1\in E_{A_1}^+$ is $C$-reachable from some edge $e_2\in E_{A_2}^+$, then every $f_1\in E_{A_1}^+$ is $C+n^{3k-2|A_1|+1}+n^{2k-2|A_2|+1}$-reachable from every $f_2\in E_{A_2}^+$.
    \end{enumerate}
\end{proposition}
\begin{proof}
   \ref{itmA,Kreachable fact1} and~\ref{itmA,Kreachable fact2} follows from the definition. For~$\ref{itmA,Kreachable fact3}$, if $|A|<k-2$, then by~\ref{Reach3}, $|V(K_{A})|\ge (1/2+3\alpha) n$. By Definition~\ref{defn:perturbed degree}, the degree of~$u$ in~$\overline{\partial_2(A})$ is at most $\alpha n$. Then there exists $uv\in E(\partial_2(A))\cap E(K_A)$. Since $\mathcal{A}$ is upward closed, greedily choose $\{x_1,y_1,\dots, x_m,y_m\}$ with~$2m+|A|+2=k$ such that $$x_{i}y_i\in E(K_{A\cup uv\cup \bigcup_{j<i}x_jy_j}).$$ Let $e=A\cup uv\cup \bigcup_{i\le m}x_iy_i$ as desired.
    
    For~$\ref{itmA,Kreachable fact7}$, let $e\in E_A^-$. Since $e$ is non-$A$-label, $e=v_1\dots v_{k}$ such that $A=v_1\dots v_{|A|}$ and $v_{s+1}v_{s+2}\notin E(K_{v_{1}\dots v_{s}})$ for some $s$ with~$|A|\le s<k$ and $s\equiv k\mod 2$. On the other hand, $v_{s+1}v_{s+2}\in E(\partial _2(v_1\dots v_s))$ and so $E(K_{v_1\dots v_s})\neq E(\partial_2(v_1\dots v_s))$. Then by~\ref{Reach44} on~$K_{v_1\dots v_s}$, $e$ is $n^{3k-2|A|}$-reachable from all edges in~$E_{\emptyset}$. 

   To see~\ref{itmA,Kreachable fact4}, by~\ref{Reach4}, assume that $E(K_A)\neq E(\partial_2(A))$. By~\ref{itmA,Kreachable fact1},~\ref{Reach43} and~\ref{Reach44},~\ref{itmA,Kreachable fact4} holds.

    For~\ref{itmA,Kreachable fact5}, let $e\in E_{A}^+$. Suppose that some edge $f\in E_A^+$ is $2n^{3k-2|A|}$-reachable from some edge $f'\in E_A^-$. Then by~\ref{Reach43} and~\ref{Reach44},~$e$ is $n^{3k-2|A|}$-reachable from~$f$. By Proposition~\ref{prop:reachablefact1}~\ref{itm:reachablefact2}, $e$ is $3n^{3k-2|A|}$-reachable from~$f'$. Note that $f'\in E_A^-$ is $n^{3k-2|A|}$-reachable from any edge in~$E_{\emptyset}$. Hence again by Proposition~\ref{prop:reachablefact1}~\ref{itm:reachablefact2}, $e\in E_A^+$ is $4n^{3k-2|A|}$-reachable from all edges in~$E_\emptyset$. Note that $E_{A}\subseteq E_\emptyset$. This contradicts with~\ref{Reach42}.

    Finally,~\ref{itmA,Kreachable fact6} follows from Proposition~\ref{prop:reachablefact1}~\ref{itm:reachablefact2} with~$(e,e',e'')=$ $(f_2,e_2,e_1)$ and $(f_2,e_1,f_1)$.
    \end{proof}

Next we give a notation to describe the common edge set between different $K_A$. Let $r\in \mathbb{N}$. Suppose that $G$ is $(\mathcal{A},\mathcal{K})$-reachable. Let $A_1,\dots,A_r$ be $r$ sets in~$\mathcal{A}$ and each of size~$m\le k-2$. Let $C=\{x_1,y_1,\dots, x_i,y_i\}$ be a set of size~$2i\le k-m$. We say that $C$ is \emph{$(A_1,\dots, A_r,i)$-common} if for all $j\in [i]$,
$$x_{j}y_{j}\in \bigcap_{s\in [r]}E(K_{A_s\cup x_1y_1,\dots x_{j-1}y_{j-1}}).$$
Then we imply a property about the existence of such common set for~$r=2$.

\begin{proposition}\label{prop:common}
      Let $1/n\ll \alpha\ll \gamma \ll 1/k\le 1/4$. Let $G$ be a $k$-graph on~$n$ vertices with~$\overline{\delta}^\alpha_{k-2}(G)\ge 1/2+\gamma$. Suppose that $G$ is $(\mathcal{A},\mathcal{K})$-reachable. Then for any $A_1,A_2\in \mathcal{A}$ with~$|A_1|=|A_2|\le k-4$ and $i\le (k-|A|-2)/2$, there exists an $(A_1,A_2,i)$-common set.
\end{proposition}

\begin{proof}
    It suffices to show that there exists a vertex sequence $\{x_1,y_1,\dots x_i,y_i\}$ such that for~$i^*\le i$, $$x_{i^*}y_{i^*}\in E(K_{A_1\cup \bigcup_{j<i^*} x_jy_j})\cap E(K_{A_2\cup \bigcup_{j<i^*} x_jy_j}).$$ 
    Suppose that we have found $x_1,y_1,\dots,x_{i^*-1},y_{i^*-1}$ for some~$i^*\in [i]$. Let $B_i=A_i\cup \bigcup_{j<i*} x_jy_j$ and $B_2=A_2\cup \bigcup_{j<i*} x_jy_j$. By Definition~\ref {defn:clique}~\ref{Reach3}, $|V(K_{B_i})|\ge (1/2+3\alpha)n$ for~$i\in [2]$ and so $$|V(K_{B_1})\cap V(K_{B_2})|\ge |V(K_{B_1})|+|V(K_{B_2})|-n>6\alpha n.$$
     By Proposition~\ref{perturb prop}~\ref{itm:perturb prop 2}, there exists an edge $x_{i^*}y_{i^*}\in E(K_{B_1})\cap E(K_{B_2})$ as we desire.
\end{proof}

In addition, the following lemma states that if there is no two label edges for two sets which are reachable, then the corresponding $K_A$ of these two sets in~$\mathcal{K}$ will intersect in a large size.

\begin{lemma}\label{label intersect}
    Let $1/n\ll \alpha\ll \gamma \ll 1/k\le 1/4$. Let $G$ be a $k$-graph on~$n$ vertices with~$\overline{\delta}^\alpha_{k-2}(G)\ge 1/2+\gamma$. Suppose that $G$ is $(\mathcal{A},\mathcal{K})$-reachable. Let $A\subseteq V$ with~$|A|<(k-4)$. Let $z_1,z_2\in V\backslash A$ be such that $A\cup z_1,A\cup z_2\in \mathcal{A}$. Suppose that for every $e\in E_{A\cup z_1}^+$ and every  $e'\in E^+_{A\cup z_2}$, $e$ is not $n^{3k-2|A|-2}$-reachable from~$e'$ or $e'$ is not $n^{3k-2|A|-2}$-reachable from~$e$. Then for any $(A\cup z_1,A\cup z_2,i)$-common set $C$ with~$2i\le k-|A|-5$, we have
    $$|V(K_{A\cup z_1 \cup C})\cap V(K_{A\cup z_2 \cup C})|> n/2.$$
\end{lemma}

 \begin{proof}
            Let $A_i=A\cup z_i \cup C$ and $K_i=K_{A_i}$ for~$i\in [2]$.
            Note that $|A_i|=|A|+1+2i\le k-4.$ Suppose to the contrary that $|V(K_1)\cap V(K_2)|\le n/2$. We claim the following hold
            \begin{enumerate}[label={(k$_{\arabic*})$}]
                
                \item  $|V(K_1)\cap V(K_2)|> 3\alpha n$ and $E(K_1)\cap E(K_2)\neq \emptyset$;\label{intersect1}
                \item  $|V(K_i)|< (1-3\alpha)n$ and $E(K_i)\neq E(\partial_2(A))$;\label{intersect2}
                \item  for~$\{i,j\}= [2]$, $|V(K_i)\backslash V(K_j)|>3\alpha n.$\label{intersect3}
            \end{enumerate}
        By~\ref{Reach3}, we have $$|V(K_1)\cap V(K_2)|\ge |V(K_1)|+|V(K_2)|-n\ge 3\alpha n.$$
        By Proposition~\ref{perturb prop}~\ref{itm:perturb prop 2}, $E(K_1)\cap E(K_2)\neq \emptyset$. Thus~\ref{intersect1} holds.  If $|K_1|\ge (1-3\alpha) n$, then  
         $$|V(K_1)\cap V(K_2)|\ge |V(K_1)|+|V(K_2)|-n>n/2,$$
         a contradiction and similarly we have $|V(K_2)|<(1-3\alpha) n$. Hence~\ref{intersect2} holds. 
         for~$\{i,j\}=[2]$,
         \begin{align*}
             |V(K_i)\backslash V(K_j)|\ge |V(K_i)|-|V(K_i)\cap V(K_j)|\overset{\ref{Reach3}}{\ge}(1/2+3\alpha) n-n/2 >3\alpha n.
         \end{align*}
        This implies~\ref{intersect3}. 
         
            Let $v_1v_2\in E(K_1)\cap E(K_2)$ which exists by~\ref{intersect1}. Since $K_1$ is an induced subgraph of~$\partial_2(A_1)$ without isolated vertices and by Definition~\ref{defn:perturbed degree}, the relative degree of~$v_1$ in~$\overline{\partial_{2}(A_1)}$ is at most $\alpha$. Then by~\ref{intersect3}, there exists $u_2\in V(K_2)\backslash V(K_1)$ such that $v_1u_2\in E(\partial_{2}(A_1))$. By Proposition~\ref{prop:common}, there exists an $(A_1\cup v_1v_2,A_2 \cup v_1v_2,\frac{k-(|A^*|+|C|+1)-2-2}{2})$-common set $D$. Then consider the following three $(k-2)$-subsets
          $$B_1=A_1\cup v_1v_2\cup D=A\cup C\cup D\cup z_1v_1v_2,$$
          $$B_2=A_2\cup v_1v_2\cup D=A\cup C\cup D\cup z_2v_1v_2,$$
          $$B_3=A_1\cup v_1u_2\cup D=A\cup C\cup D\cup z_1v_1u_2.$$

          We have 
          \begin{enumerate}[label={(R$_{\arabic*})$}]
              \item\label{in1} $E_{B_1}\subseteq E^+_{A_1}= E^+_{A\cup z_1}$ (by Proposition~\ref{A,Kreachable fact}~\ref{itm:reachablefact2} and $v_1v_2\in E(K_1)\cap E(K_2)$);
              \item\label{in2} $E_{B_2}\subseteq E^+_{A_2}= E^+_{A\cup z_2}$ (by Proposition~\ref{A,Kreachable fact}~\ref{itm:reachablefact2} and $v_1v_2\in E(K_1)\cap E(K_2)$);
              \item\label{in3} $E_{B_3}\subseteq E^-_{A_1}$ (as $u_2\notin V(K_1)$).
          \end{enumerate}
 By Lemma~\ref{lem:sturc1}~\ref{Matching 4}, there exist $i,j\in [3]$ such that $E(C_{B_i})\cap E(C_{B_j})\neq \emptyset$. Next we discuss the following cases:

    \medskip \noindent    \textbf{Case 1:} $E(C_{B_1})\cap E(C_{B_2})\neq \emptyset$. Let $y_1y_2\in E(C_{B_1})\cap E(C_{B_3})$, $e_1=B_1\cup y_1y_2$, $e_2=B_2\cup y_1y_2$. Then $$|e_1\cap e_2|=k-|B_1\backslash B_2|=k-1.$$
    Hence $e_1\in E_{A_1}^+$ is $1$-reachable from~$e_2\in E_{A_2}^+$ by Proposition~\ref{prop:reachablefact1}~\ref{itm:reachablefact4}, and vice versa. This is a contradiction. 
        
    \medskip \noindent      \textbf{Case 2:} $E(C_{B_1})\cap E(C_{B_3})\neq \emptyset$. Let $y_1y_2\in E(C_{B_1})\cap E(C_{B_3})$, $e_1=B_1\cup y_1y_2$ and $e_3=B_3\cup y_1y_2$. Note that $$|e_1\cap e_3|=k-|B_1\backslash B_3|=k-1.$$
    Hence $e_1\in E_{A_1}^+$ is $1$-reachable from~$e_3$ by Proposition~\ref{prop:reachablefact1}~\ref{itm:reachablefact4} and vice versa. By Proposition~\ref{A,Kreachable fact}~\ref{itmA,Kreachable fact5}, $E(K_1)=E(K_{A_1})=E(\partial_2(A))$. This contradicts with~\ref{intersect2}.

    \medskip \noindent      \textbf{Case 3:} $E(C_{B_2})\cap E(C_{B_3})\neq \emptyset$. Let $y_1y_2\in E(C_{B_2})\cap E(C_{B_3})$. Let $e_2=B_2\cup y_1y_2$ and $e_3=B_3\cup y_1y_2$. Next, we claim that $e_2$ is $2n^{3k-2|A_2|}$-reachable from~$e_3$.
        Note that $$e_3\backslash e_2=B_3\backslash B_2= \{z_1,u_2\}.$$ 
        By Proposition~\ref{prop:reachablefact1}~\ref{itm:reachablefact3} and~\ref{itm:reachablefact1}, it suffices to show that $e_2$ is $(2n^{3k-2|A_2|}-1)$-reachable from~$u_2$. 

        By~$u_2\in V(K_2)$ and Proposition~\ref{A,Kreachable fact}~\ref{itmA,Kreachable fact2}, there exists an edge $f\in E_{A_2}^+$ containing $u_2$. Since $e_2\in E_{A_2}^+$, by Definition~\ref{defn:clique}~\ref{Reach43} and Proposition~\ref{prop:reachablefact1}~\ref{itm:reachablefact1}, $e_2$ is $n^{3k-2|A_2|}$-reachable from~$f$ and so from~$u_2$. Hence, $e_2$ is $2n^{3k-2|A_2|}$-reachable from~$e_3$. Meanwhile, $e_3\in E^{-}_{A_1}$ by~\ref{in3}. By~Proposition~\ref{A,Kreachable fact}~\ref{itmA,Kreachable fact7}, $e_3$ is $n^{3k-2|A_1|}$-reachable from all edges $e'\in E_{\emptyset}$. By Proposition~\ref{prop:reachablefact1}~\ref{itm:reachablefact2}, $e_2$ is $3n^{3k-2|A_2|}$-reachable from all edges in~$E_{\emptyset}$. Since $E(K_2)\neq E(\partial_2(A_2))$ by~\ref{intersect2}, there is $e''\in E_{A_2}^-$. So $e_2\in E_{A_2}^+$ is $3n^{3k-2|A_2|}$-reachable from~$e''\in E_{A_2}^-$. This contradicts with Proposition~\ref{A,Kreachable fact}~\ref{itmA,Kreachable fact5} and~\ref{intersect2}.
    \end{proof}

\subsection{Auxiliary edge-colouring}\label{reach general}
Recall that $\partial_{\mod 2}(G)=\bigcup_{s\equiv k \mod 2, s< k}\partial_{s}(G)$. Let $A\in \partial_{\mod 2}(G)$. In this subsection, we will define an edge-colouring of~$\partial_2(A)$ to help us partition label edges.

\begin{defn}\label{2-edge colour}
   Let $1/n\ll \alpha\ll \gamma \ll 1/k\le 1/4$. Let $G$ be a $k$-graph on~$n$ vertices with~$\overline{\delta}^\alpha_{k-2}(G)\ge 1/2+\gamma$. Suppose that $G$ is $(\mathcal{A}, \mathcal{K})$-reachable with~$E(\partial_{k-2}(G))\subseteq \mathcal{A}$. Let $A$ be a maximal element of~$\partial_{\mod 2}(G)\backslash \mathcal{A}$. We define an auxiliary edge-colouring $\phi_A$ on~$\partial_2(A)$ to be such that for all distinct $u_1v_1,u_2v_2\in E(\partial_2(A))$,  $\phi_A(u_1v_1)=\phi_A(u_2v_2)$ if and only if there exists an edge~$e_1\in E^+_{A\cup u_1v_1}$ and an edge $e_2\in E^+_{A\cup u_2v_2}$ such that $e_i$ is $n^{3k-2|A|-2}$-reachable from~$e_j$ for~$\{i,j\}=[2]$.
\end{defn}

Then we have the following property by Proposition~\ref{A,Kreachable fact}~\ref{itmA,Kreachable fact5}.
\begin{proposition}\label{any label edge}
     Let $1/n\ll \alpha\ll \gamma \ll 1/k\le 1/4$. Let $G$ be a $k$-graph on~$n$ vertices with~$\overline{\delta}^\alpha_{k-2}(G)\ge 1/2+\gamma$. Suppose that $G$ is $(\mathcal{A}, \mathcal{K})$-reachable with~$E(\partial_{k-2}(G))\subseteq \mathcal{A}$. Let $A$ be a maximal element of~$\partial_{\mod 2}(G)\backslash \mathcal{A}$. Let $\phi_A$ be the edge-colouring as in Definition~\ref{2-edge colour}. If $u_1v_1,u_2v_2\in E(\partial_2(A))$ with~$\phi_A(u_1v_1)=\phi_A(u_2v_2)$, then any $e_1\in E^+_{A\cup u_1v_1}$ is $3n^{3k-2|A|-2}$-reachable from all edges $e_2\in E^+_{A\cup u_2v_2}$.
\end{proposition}

 We say a  $2$-graph $H$ on~$n$ vertices with an edge-colouring $\phi$ is $\emph{Gallai colouring}$ if and only if there is no rainbow triangle in~$H$.
 We say $\phi$ is \emph{locally $2$-edge-colouring}, if $\{\phi(e):v\in e\}|\le 2$ for all~$v\in V(H)$. Next, we show a key lemma that $\phi_A$ is Gallai locally $2$-edge-colouring.
\begin{lemma}\label{Gallai coloured}
    Let $1/n\ll \alpha\ll \gamma \ll 1/k\le 1/4$. Let $G$ be a $k$-graph on~$n$ vertices with~$\overline{\delta}^\alpha_{k-2}(G)\ge 1/2+\gamma$. Suppose that $G$ is $(\mathcal{A}, \mathcal{K})$-reachable with~$E(\partial_{k-2}(G))\subseteq \mathcal{A}$. Let $A$ be a maximal element of~$\partial_{\mod 2}(G)\backslash \mathcal{A}$. Let $\phi_A$ be the edge-colouring as in Definition~\ref{2-edge colour}. Then $\phi_A$ is Gallai locally $2$-edge-coloured. 
\end{lemma}
\begin{proof}
        Suppose to the contrary that $\phi_A$ is not locally $2$-edge-coloured. Hence there exist distinct $\phi_A(wz_1)$, $\phi_A(wz_2)$ and $\phi_A(wz_3)$ for some $w\in V(G)$. Let $A\cup wz_i=A_i$ for~$i\in [3]$.
        \begin{claim}\label{Cliam: edge three common}
            There is an $(A_1,A_2,A_3, (k-|A_1|-2)/2)$-common set $C$.
        \end{claim}
       
        \begin{proofclaim}
            Let $m=(k-|A_1|-2)/2$. Suppose that we have chosen $x_1,y_1,\dots ,x_{j^*},y_{j^*}$ for some $j^*\in [m-1]\cup \{0\}$ such that $X_{j^*}=\{x_1,y_1,\dots ,x_{j^*},y_{j^*}\}$ and is $(A_1, A_2, A_3, j^*)$-common. By the definition of~$\phi_A$,  for every $e_1\in E^+_{A\cup wz_1}=E^+_{A_1}$ and every $e_2\in E^+_{A\cup wz_2}=E^+_{A_2}$, $e_i$ is not $n^{3k-2|A|-2}$-reachable from~$e_{i'}$ for some $\{i,i'\}= [2]$. So by Lemma~\ref{label intersect} with~$(C,A)=(X_{j^*},A\cup w)$, $$|V(K_{A_1\cup X_{j^*}})\cap  V(K_{A_2\cup X_{j^*}})|=|V(K_{A\cup wz_1\cup X_{j^*}})\cap  V(K_{A\cup wz_2\cup X_{j^*}})|>n/2.$$
              Since $A$ is the maximal set of~$\partial_{\mod 2}(G)\backslash \mathcal{A}$, we have $A_1,A_2,A_3\in \mathcal{A}$. By~\ref{Reach3}, $|V(K_{A_3\cup X_{j^*}})|\ge  (1/2+3\alpha)n$. Thus we have 
        \begin{align*}
            |\bigcap_{i\in [3]}V(K_{A_i\cup X_{j^*}})|\ge|V(K_{A_1\cup X_{j^*}})\cap  V(K_{A_2\cup X_{j^*}})|+|V(K_{A_3\cup X_{j^*}})|-n> 3\alpha n.
        \end{align*} 
    By Proposition~\ref{perturb prop}~\ref{itm:perturb prop 2}, there exists $x_{j^*+1}y_{j^*+1}\in \bigcap_{i\in [3]}E(K_{A_i\cup X_{j^*}}) $. Let $C=\{x_1,y_1\dots , x_m,y_m\}$ as required.
        \end{proofclaim}
 For~$i \in [3]$, let $B_i=A_i\cup C$. 
 Note that $B_i\in E(\partial_{k-2}(G))$ for~$i\in [3]$. By Lemma~\ref{lem:sturc1}~\ref{Matching 4}, there exist $B_{i_1}$ and $B_{i_2}$ such that $E(C_{B_{i_1}})\cap E(C_{B_{i_2}})\neq \emptyset$ for distinct $i_1,i_2\in [3]$. Let $u_1u_2\in E(C_{B_{i_1}})\cap E(C_{B_{i_2}})$. Let $f_j=B_{i_j}\cup u_1u_2$ for~$j\in [2]$. Then $f_j\in E_{B_{i_j}}\subseteq E^+_{A_{i_j}}=E^+_{A\cup wz_{i_j}}$. Note that 
          \begin{align*}
              |f_1\cap f_2|=|(B_{i_1}\cup u_1u_2)\cap (B_{i_2}\cup u_1u_2)|=|A\cup (wz_{i_1}\cap wz_{i_2})\cup C\cup u_1u_2|=k-1.
          \end{align*}
          Then by Proposition~\ref{prop:reachablefact1}~\ref{itm:reachablefact4}, $f_i$ is $1$-reachable from~$f_j$ for~$\{i,j\}=[2]$ . However, $f_j\in E^+_{A\cup wz_{i_j}}$ for~$j\in [2]$. This contradicts the fact that $\phi_A(wz_{i_1})\neq \phi_A(wz_{i_2})$. Hence, $\phi_A$ is locally $2$-edge-coloured.
          
         Suppose that $\phi_A$ is not Gallai-coloured. Then there exists a rainbow triangle $u_1u_2u_3$. We replace ($wz_1$, $wz_2$, $wz_3$) by~$(u_1u_2,u_1u_3,u_2u_3)$ in the above argument. We can deduce that there is no rainbow triangle in~$\phi_A$.  
    \end{proof}
Then the following lemma shows that if $H$ is almost complete with a Gallai locally $2$-edge-colouring, we can describe its structure.
\begin{lemma}\label{near Gallai}
    Let $1/n\ll \alpha\ll 1$ and $H$ be a $2$-graph on~$n$ vertices with~$\delta(H)\ge (1-\alpha)n$. Suppose that $H$ is Gallai locally $2$-edge-coloured by~$\phi$ with monochromatic components $H_1,\dots,H_m$ with~$|V(H_1)|\ge \dots \ge |V(H_m)|$. Then the following hold
    \begin{enumerate}[label=\rm{(g$_{\arabic*}$)} ]
        \item \label{gaedgecolour1} if $H_1$ spans $H$, then $H_2,\dots, H_m$ are vertex-disjoint and $|V(H_1)|=n$;
    \item\label{gaedgecolour2}if $H_1$ does not span $H$, then $|V(H_2)|\ge (1-2\alpha)n$ and $H_1\cup H_2$ spans $H$;
    \item\label{gaedgecolour3} $|V(H_1)|\ge (1-2\alpha)n$ and $|V(H_i)|\le n/2$ for~$i\in [3,m]$.
    \end{enumerate}

\end{lemma}

\begin{proof}
    For each vertex $v\in V(H)$, let $N_i(v)$ be the set of~$u$ such that $\phi(uv)=i$ and $H_v=H[v\cup N(v)]$, then
    \begin{claim}\label{size}
         There exists a colour $j_v$ in~$H_v$ such that $N_{j_v}(v)\neq \emptyset$ and  $$|\{z\in V(H_v)~|N_j(z)\neq \emptyset\}|\ge (1-2\alpha)n.$$
    \end{claim}
    \begin{proofclaim}
          Suppose the claim is false for some $v\in V(H)$. Without loss of generality, suppose that $\{\phi(vu)~|~u\in N(v)\}=[2]$. As $\delta(H)\ge (1-\alpha)n$, if for every $u\in N_1(v)$ with~$N_2(u)\neq \emptyset$, then 
          $$|\{z\in V(H_v)~|N_2(z)\neq \emptyset\}|\ge \deg(v)\ge (1-\alpha)n,$$
        a contradiction. Furthermore, let $u\in N_1(v)$ with~$N_2(u)=\emptyset$. Since there is no rainbow triangle for all $w\in N_2(v)\cap N(v)$, $\phi(uw)=1$. Hence
          \begin{align*}
              |\{u\in N(v)|N_1(u)\neq \emptyset\}|&\ge |N_1(v)|+|N_2(v)\cap N(u)|\\
              &\ge |N_1(v)|+(\deg(v)-|N_1(v)|)+\deg(u)-n
              \ge 2\delta(H)-n\ge (1-2\alpha)n,
          \end{align*}
          a contradiction.
    \end{proofclaim}
Next, we claim that there exist $i_1,i_2$ such that $j_v\in \{i_1,i_2\}$ for all $v\in V(H)$. Otherwise, assume that $j_u,j_v,j_w$ are pairwise distinct. Then
    $$|\{z~|N_{j_v}(z)\neq \emptyset\}\cap \{z~|N_{j_u}(z)\neq \emptyset\}\cap \{z~|N_{j_w}(z)\neq \emptyset\}|\ge 3(1-2\alpha)n-2n>0.$$
So there is a vertex $z'$ with~$N_{j_v}(z'),N_{j_w}(z'),N_{j_u}(z')\neq \emptyset$. This contradicts with that $\phi$ is locally $2$-edge-coloured.

To see~\ref{gaedgecolour1}, if $H_1$ spans $H$, then $|V(H_1)|=n$. Moreover, after removing edges of~$H_1$, every vertex is incident to edges of at most one colour. So the remaining graph are separated monochromatic components.

If $H_1$ does not span $H$, then $i_1\neq i_2$. By Claim~\ref{size}, for all $v\in V(H)$, $j_v\in \{i_1,i_2\}$ and so $N_{i_1}(v)\neq \emptyset$ or $N_{i_2}(v)\neq \emptyset$. This implies that $V(H)= V(H_{i_1})\cup V(H_{i_2})$ and $|V(H_{i_j})|\ge (1-2\alpha)n$ for~$j\in [2]$. After removing edges of~$H_1 \cup H_2$, every vertex is incident to edges of at most one colour. So the remaining graph are separated monochromatic components. If $|V(H_{i_3})|\ge 4\alpha n$ with~$i_3\neq i_1,i_2$, then $$|V(H_{i_3})\cap V(H_{i_1})\cap V(H_{i_2})|> 4\alpha n+(1-2\alpha)n+(1-2\alpha)n-2n\ge 0.$$
So there is a vertex $u\in V(H_{i_3})\cap V(H_{i_1})\cap V(H_{i_2})$. It implies that $N_{i_j}(u)\neq \emptyset$ for~$j\in [3]$. This contradicts with locally $2$-edge-coloured. Meanwhile, by Claim~\ref{size}, we have that $|V(H_{i_1})|$ and~$|V(H_{i_2})|$ are larger than $(1-2\alpha)n$. It indicates that $\{i_1,i_2\}=[2]$. Thus $|V(H_2)|\ge (1-2\alpha)n$ and $H_1\cup H_2$ spans $H$. Therefore,~\ref{gaedgecolour2} holds.

For~\ref{gaedgecolour3}, if $H_1$ spans $H$, then $|V(H_1)|=n$ and $H_2,H_3,\dots, H_m$ are vertex-disjoint. Hence for~$i\in [3,m]$, $$2|V(H_i)|\le |V(H_2)|+\dots +|V(H_m)|\le n$$ implying $V(H_i)\le n/2$. If $H_1$ does not span $H$, then $|V(H_1)|\ge (1-2\alpha)n$ and $|V(H_i)|\le 4\alpha n\le n/2$ for~$i\in [3,m]$ by the argument of~\ref{gaedgecolour2}.
\end{proof}

    Let $|A|\le k-4$ and $I_A$ be the set of isolated vertices in~$\partial_2(A)$. By Definition~\ref{defn:perturbed degree}, $|V(\partial_2(A))\backslash I_A|\ge (1-\alpha) n$ and the minimum degree of~$V(\partial_2(A))\backslash I_A$ is at least $(1-\alpha)n$. Next by Lemma~\ref{near Gallai} with~$H=\partial_2(A)\backslash I_A$, we obtain following structural statement for~$\phi_A$.

\begin{coro}\label{Gallai coloured prop}
    Let $1/n\ll \alpha\ll \gamma \ll 1/k\le 1/4$. Let $G$ be a $k$-graph on~$n$ vertices with~$\overline{\delta}^\alpha_{k-2}(G)\ge 1/2+\gamma$. Suppose that $G$ is $(\mathcal{A}, \mathcal{K})$-reachable with~$E(\partial_{k-2}(G))\subseteq \mathcal{A}$. Let $A$ be a maximal element of~$\partial_{\mod 2}(G)\backslash \mathcal{A}$. Let $\phi_A$ be the edge-colouring of~$\partial_2(A)$ as defined in Definition~\ref{2-edge colour} with colours~$[m]$. Let $H_1,\dots H_m$ be monochromatic components of~$\phi_A$ such that $|V(H_1)|\ge \dots \ge |V(H_m)|$. Let~$I_A$ be the set of isolated vertices in~$\partial_2(A)$. Then the following hold

\begin{enumerate}[label=\rm{(c$_{\arabic*})$}]
    \item\label{edgecolour1} if $H_1$ spans $\partial_2(A)\backslash I_A$, then $H_2,\dots, H_m$ are vertex-disjoint and $|V(H_1)|\ge (1-\alpha)n$;
    \item\label{edgecolour2}if $H_1$ does not span $\partial_2(A)\backslash I_A$, then $|V(H_2)|\ge (1-3\alpha)n$ and $H_1\cup H_2$ spans $\partial_2(A)\backslash I_A$;
    \item\label{edgecolour3} $|V(H_1)|\ge (1-3\alpha)n$ and $|V(H_i)|\le n/2$ for~$i\in [3,m]$.
\end{enumerate}
\end{coro}

Then according to the size of each component, we discuss reachable properties in different way. First, we show that if $|V(H_i)|$ is near $n$, then all edges in~$E(H_i)$ are reachable from all edges in~$E_{\emptyset}$. 

\begin{lemma}\label{large component}
     Let $1/n\ll \alpha\ll \gamma \ll 1/k\le 1/4$. Let $G$ be a $k$-graph on~$n$ vertices with~$\overline{\delta}^\alpha_{k-2}(G)\ge 1/2+\gamma$. Suppose that $G$ is $(\mathcal{A}, \mathcal{K})$-reachable with~$E(\partial_{k-2}(G))\subseteq \mathcal{A}$. Let $A$ be a maximal element of~$\partial_{\mod 2}(G)\backslash \mathcal{A}$. Let $\phi_A$ be the edge-colouring of~$\partial_2(A)$ as defined in Definition~\ref{2-edge colour} with colours~$[m]$. Suppose that $H_i$ is a monochromatic component of~$\phi_A$ with~$|V(H_i)|> (1-3\alpha)n$. Then any $e\in E_{A\cup u_1v_1}$ with~$u_1v_1\in E(H_i)$ is $n^{3k-2|A|-1}$-reachable from every edge in~$E_\emptyset$.
\end{lemma}
\begin{proof}
    It suffices to show that $e\in E^+_{A\cup u_1v_1}$ is $n^{3k-2|A|-1}$-reachable from all vertices~$v\in V(G)$. First, if~$v\in V(H_i)$, then there exists a vertex $v'$ such that $vv'\in E(H_i)$. Let $e'\in E^+_{A\cup vv'}$. Then since $\phi_A(vv')=\phi_A(u_1v_1)$, by Proposition~\ref{any label edge}, $e'$ is $3n^{3k-2|A|-2}$-reachable from~$e$ and so from~$v$. Hence we may assume that~$v\notin V(H_i)$.
    \begin{claim}
         There exists an edge $e'\in E_{\emptyset}$ such that $v\in e'$ and $e'\backslash v\subseteq V(H_i)$. 
    \end{claim}
    \begin{proofclaim}
        Since $|V(H_i)|\ge (1-3\alpha)n$ and any $S\in E(\partial_j(G))$ has relative degree at most $\alpha$ in~$\overline{\partial_{j+1}(G)}$ for~$j\in [k-2]$, we can iteratively choose $u_2,\dots ,u_{k-2}\in V(H_i)$ such that $vu_2\dots u_{k-2}\in E(\partial_{k-2}(G))$. Let $B=vu_2\dots u_{k-2}$. Since $\overline{\delta}^{\alpha}_{k-2}(G)\ge 1/2+\gamma$, by Lemma~\ref{lem:sturc1}~\ref{Matching 1}, there exist at least $\frac{1}{4}\binom{n}{2}$ edges in~$C_B$. Note that $$|E(C_{B})\cap \binom{V(H_i)}{2}|\ge \frac{1}{4}\binom{n}{2}+\binom{(1-3\alpha)n}{2}-\binom{n}{2}>0.$$
        Let $u_{k-1}u_k\in E(C_B)\cap \binom{V(H_i)}{2}$. So $vu_2\dots u_k\in E_\emptyset$ as desired.
    \end{proofclaim}
 Let $e'$ be the edge such that $v\in e'$ and $e'\backslash v\subseteq V(H_i)$, so $e$ is $3n^{3k-2|A|-2}$-reachable from all vertices of~$e'$ except~$v$. By Proposition~\ref{prop:reachablefact1}~\ref{itm:reachablefact3}, $e$ is $n^{3k-2|A|-1}$-reachable from~$e'$ and so from~$v$.
\end{proof}
Then we show that edges in small components are also reachable from every edge in~$E_\emptyset$. 
\begin{lemma}\label{small component}
      Let $1/n\ll \alpha\ll \gamma \ll 1/k\le 1/4$. Let $G$ be a $k$-graph on~$n$ vertices with~$\overline{\delta}^\alpha_{k-2}(G)\ge 1/2+\gamma$. Suppose that $G$ is $(\mathcal{A}, \mathcal{K})$-reachable with~$E(\partial_{k-2}(G))\subseteq \mathcal{A}$. Let $A$ be a maximal element of~$\partial_{\mod 2}(G)\backslash \mathcal{A}$. Let $\phi_A$ be the edge-colouring of~$\partial_2(A)$ as defined in Definition~\ref{2-edge colour} with colours~$[m]$. Suppose that $H_i$ is a monochromatic component of~$\phi_A$ with~$|V(H_i)|<(1/2+3\alpha)n$ for some $i\in [m]$. Then any $e\in E_{A\cup u_1v_1}$ with~$u_1v_1\in E(H_i)$ is $4n^{3k-2|A|-1}$-reachable from every edge in~$E_\emptyset$.
\end{lemma}
\begin{proof}
Let $e\in E_{A\cup u_1v_1}$. By Lemma~\ref{large component} and Proposition~\ref{prop:reachablefact1}~\ref{itm:reachablefact2}, it suffices to show that $e$ is $3n^{3k-2|A|-1}$-reachable from some edge $e'\in E_{A\cup u_2v_2}$ such that $u_2v_2\in E(H_j)$ with~$|V(H_j)|\ge (1-3\alpha)n$ for some $j\in [m]$. Consider following cases.

\medskip \noindent  \textbf{Case 1:} $|A\cup u_1v_1|< k-2$. Then by~\ref{Reach3}, $$|V(K_{A\cup u_1v_1} )\backslash V(H_i)|\ge |V(K_{A\cup u_1v_1})|-|V(H_i)|>0.$$
Let $x\in V(K_{A\cup u_1v_1} )\backslash V(H_i)$. Note that $x$ is not isolated in~$\partial_2(A\cup u_1v_1)$ by~\ref{Reach2}. By Definition~\ref{defn:perturbed degree},~\ref{Reach2} and $|V(K_{A\cup u_1v_1})|\ge (1/2+3\alpha)n$, there exists  $xy\in E(K_{A\cup u_1v_1})$. Choose $e'\in E^+_{A\cup u_1v_1xy}$, so $e'\in E^+_{A\cup u_1v_1}$. Then by Proposition~\ref{A,Kreachable fact}~\ref{itmA,Kreachable fact4} and $e\in E_{A\cup u_1v_1}$, $e$ is $n^{3k-2|A|-1}$-reachable from~$e'$. 
Note that $x\notin V(H_i)$, $u_1\in V(H_i)$ and so $xu_1\notin E(H_i)$. By Corollary~\ref{Gallai coloured prop}~\ref{edgecolour1} and~\ref{edgecolour2}, $u_1\in V(H_j)$ with~$|V(H_j)|\ge (1-3\alpha)n$. Since $\phi_A$ is locally $2$-edge colouring, $xu_1\in E(H_j)$. By Definition~\ref{(A,k)-reach}, $e'\in E_{A\cup u_1v_1xy}\subseteq E_{A\cup u_1x}$ with~$u_1x\in E(H_j)$ as required.

\medskip \noindent  \textbf{Case 2: $|A\cup u_1v_1|=k-2$}. By Lemma~\ref{lem:sturc1}~\ref{Matching 1}, $|V(C_{A\cup u_1v_1})|\ge (1/2+\gamma)n$ and then $$|V(C_{A\cup u_1v_1})\backslash V(H_i)|\ge |V(C_{A\cup u_1v_1})|-|V(H_i)|>\alpha n.$$
By Definition~\ref{defn:perturbed degree}, there exists a vertex $x\notin V(H_i)$ such that $u_1x\in E(\partial_2(A))$. Note that $u_1\in V(H_i)$ and so $xu_1\notin E(H_i)$. By Corollary~\ref{Gallai coloured prop}~\ref{edgecolour1} and~\ref{edgecolour2}, $u_1\in V(H_j)$ with~$|V(H_j)|\ge (1-3\alpha)n$ for some~$j\in [2]$. Since $\phi_A$ is locally $2$-edge-colouring, this implies that $u_1x\in E(H_j)$. By Lemma~\ref{lem:sturc1}~\ref{Matching 1}, $|V(C_{A\cup u_1x})|\ge (1/2+\gamma)n$. Then $$ | V (C_{A\cup u_1x})\cap V(C_{A\cup u_1v_1})|\ge |V(C_{A\cup u_1v_1})|+|V(C_{A\cup u_1x})|-n\ge 0.$$
Let $w\in V (C_{A\cup u_1x})\cap V(C_{A\cup u_1v_1})$, $wz_1\in E(C_{A\cup u_1x})$ and $wz_2\in E(C_{A\cup u_1u_2})$. Let
$$f=A\cup u_1v_1wz_1\text{ and }f'=A\cup u_1xwz_2.$$
So $f\in E_{A\cup u_1v_1}^+$. Note that $f'\in E_{A\cup u_1x}$ with~$u_1x\in E(H_j)$ and $|V(H_j)|\ge (1-3\alpha)n$. Since $f,e\in E^+_{A\cup u_1v_1}$, $e$ is $n^{3k-2|A|-2}$-reachable from~$f$ by~\ref{Reach4}. Therefore to complete the proof, by Proposition~\ref{prop:reachablefact1}~\ref{itm:reachablefact2}, it suffices to show that $f$ is $2n^{3k-2|A|-1}$-reachable from~$f'$. Since $f'\backslash f=\{x,z_2\}$, by Proposition~\ref{prop:reachablefact1}~\ref{itm:reachablefact3} and~\ref{itm:reachablefact1}, it suffices to show that $f$ is $n^{3k-2|A|-1}$-reachable from~$x$. Since $x\in V(C_{A\cup u_1v_1})$, let $xy\in E(C_{A\cup u_1v_1})$ and $f''=A\cup u_1v_1xy$. Then $f'',f\in E_{A\cup u_1v_1}$. Thus $f$ and $f''$ is tight connect by Definition~\ref{defn:our graph}. By Proposition~\ref{prop:reachablefact1}~\ref{itm:reachablefact5}, $f$ is $n^k$-reachable from~$f''$. Since $x\in f''$ and $|A|<k$, $f$ is $n^{3k-2|A|-1}$-reachable from~$x$ and so $2n^{3k-2|A|-1}$-reachable from~$f'$.
\end{proof}

Finally, we deal with the reachable property about edges in a component with middle size. Note that by Corollary~\ref{Gallai coloured prop}, this component is $H_2$. We will show that these edges are reachable within themselves.
\begin{lemma}\label{middle component}
   Let $1/n\ll \alpha\ll \gamma \ll 1/k\le 1/4$. Let $G$ be a $k$-graph on~$n$ vertices with~$\overline{\delta}^\alpha_{k-2}(G)\ge 1/2+\gamma$. Suppose that $G$ is $(\mathcal{A}, \mathcal{K})$-reachable with~$E(\partial_{k-2}(G))\subseteq \mathcal{A}$. Let $A$ be a maximal element of~$\partial_{\mod 2}(G)\backslash \mathcal{A}$. Let $\phi_A$ be the edge-colouring of~$\partial_2(A)$ as defined in Definition~\ref{2-edge colour} with colours $[m]$. Let $H_1,\dots H_m$ be monochromatic components of~$\phi_A$ such that $|V(H_1)|\ge \dots \ge |V(H_m)|$. Suppose that $(1/2+3\alpha)n<|V(H_2)|<(1-3\alpha)n$. Then every edge $e\in E^+_{A\cup u_1v_1}$ with~$u_1v_1\in E(\partial_2(A)[V(H_2)])$ is $n^{3k-2|A|}$-reachable from every edge $e'\in E^+_{A\cup u_2v_2}$ with~$u_2v_2\in E(\partial_2(A)[V(H_2)])$.
\end{lemma}
\begin{proof}
    Let $I_A$ be the set of isolated vertices in~$\partial_2(A)$. By Corollary~\ref{Gallai coloured prop}, $H_1$ spans $V(\partial_2(A))\backslash I_A$ and $|V(H_1)|\ge (1-\alpha)n$.  By Corollary~\ref{Gallai coloured prop}~\ref{edgecolour1}, we have that $u_1,u_2\in V(H_1)\cap V(H_2)$. So $u_1v_1, u_2v_2\in \{E(H_1),E(H_2)\}$ as~$\phi_A$ is locally $2$-edge-coloured. 
    
    If $\phi_A(u_1v_1)=\phi_A(u_2v_2)$, then lemma follows Proposition~\ref{any label edge}. Assume that $\phi_A(u_1v_1)\neq \phi_A(u_2v_2)$. If $u_1v_1\in E(H_1)$, then by Lemma~\ref{large component} on~$H_1$, $e$ is $n^{3k-2|A|}$-reachable from all edges in~$E_{\emptyset}$ and so from~$e'$. Hence, suppose that $u_1v_1\in E(H_2)$ and $u_2v_2\in E(H_1)$. 

    By Proposition~\ref{prop:common}, there exists a $(A\cup u_1v_1,A\cup u_2v_2, (k-|A|-4)/2)$-common set $C$. Let $B_i=A\cup u_iv_i\cup C$ for~$i\in [2]$ and so $|B_1|=|B_2|=k-2$. Then by Lemma~\ref{lem:sturc1}, we have $$|V(C_{B_1})\cap V(C_{B_2})|\ge |V(C_{B_1})|+|V(C_{B_2})|-n> 0.$$
    Let $x\in V(C_{B_1})\cap V(C_{B_2})$, $xy\in E(C_{B_1})$ and $xz\in E(C_{B_2})$. Let
    $$f=B_1\cup xy=A\cup u_1v_1\cup C\cup xy\text{ and }f'=B_2\cup xz=A\cup u_2v_2\cup C\cup xz.$$
    Note that $f'\backslash f=\{z,u_2,v_2\}$. Next we claim that $f$ is $2n^{3k-2|A|-1}$-reachable from~$f'$. By Proposition~\ref{prop:reachablefact1}~\ref{itm:reachablefact1} and~\ref{itm:reachablefact3}, it suffices to show that $f$ is $n^{3k-2|A|-1}$-reachable from~$u_2,v_2$. Note that $u_2\in V(H_2)$ implying that there exists an edge $u_2u'\in E(H_2)$ such that $\phi_A(u_2u')=\phi_A(u_1v_1)$. Let $f''\in E^+_{A\cup u_2u'}$. Since $f\in E^+_{A\cup u_1v_1}$ and~$u_2\in f''$, by the definition of~$\phi_A$ and Proposition~\ref{any label edge}, $f$ is $n^{3k-2|A|-1}$-reachable from~$f''$ and so from~$u_2$. Note that $v_2\in V(H_2)$ as well. We replace $u_2$ by~$v_2$ in the above argument. We can deduce that $f$ is also  $n^{3k-2|A|-1}$-reachable from~$v_2$. Thus $f$ is $2n^{3k-2|A|-1}$-reachable from~$f'$.

Meanwhile, $e,f\in E^+_{A\cup u_1v_1}$ and $e',f'\in E^+_{A\cup u_2v_2}$. Then by Proposition~\ref{prop:reachablefact1}~\ref{itm:reachablefact2} and~\ref{Reach43},
we obtain that $e$ is $n^{3k-2|A|}$-reachable from~$e'$.
\end{proof}
 
\subsection{$(\mathcal{A},\mathcal{K})$-reachable}\label{(A,k)-reach}
In this subsection, we aim to prove that our host graph $G$ is $(\mathcal{A},\mathcal{K})$-reachable for~$\mathcal{A}=\partial_{\mod 2}(G)$. Recall the definition of~$G^*$ and $E_{\emptyset}$ in Definition~\ref{defn:our graph}. 

\begin{proof}[Proof of Lemma~\ref{Lem:reach even case}]
Suppose that $G$ is $(\mathcal{A}, \mathcal{K})$-reachable for some upward closed subset $\mathcal{A}\subseteq \partial_{\mod 2}(G)$. If $\mathcal{A}=\partial_{\mod 2}(G)$, then we are done. Otherwise, let $A$ be a maximal subset of~$\partial_{\mod 2}(G)\backslash \mathcal{A}$. By Proposition~\ref{k-2 reachable}, we can assume that $|A|<k-2$. Note that $\mathcal{A}\cup \{A\}$ is still upward closed. To prove the lemma, it suffices to show that $G$ is $(\mathcal{A}\cup \{A\},\mathcal{K}\cup \{K_A\})$-reachable for some $K_A$.

 Let  $I_A$ be the set of isolated vertices in~$\partial_2(A)$. By Definition~\ref{defn:perturbed degree}, we have $|I_A|\le \alpha n$. We now define $K_A$ as follows.

\medskip \noindent   \textbf{Case 1:} every edge $e\in E_A$ is $n^{3k+1-2|A|}$-reachable from every edge $e'\in E_A$. Then set $K_A=\partial_2(A)\backslash I_A$. Note that~\ref{Reach1}, ~\ref{Reach2} and~\ref{Reach4} hold by construction. Since $|V(\partial_2(A))\backslash I_A|\ge (1-\alpha)n$,~\ref{Reach3} holds. 

\medskip \noindent   \textbf{Case 2:}
there exists an edge $e\in E_A$ is not $n^{3k+1-2|A|}$-reachable to some $e'\in E_A\subseteq E_{\emptyset}$. Define~$\phi_A$ as in Definition~\ref{2-edge colour}. Then $\phi_A$ has at least two monochromatic components. Let $H_1,\dots ,H_m$ be monochromatic components of~$\phi_A$ with~$|V(H_1)|\ge \dots \ge |V(H_m)|$. Set $K_A=\partial_2(A)[V(H_2)]$. 

Note that by our construction and the definition of~$\phi_A$,~\ref{Reach1},~\ref{Reach2} and~\ref{Reach42} hold for~$K_A$. 

 \begin{claim}~\label{size claim}
       $(1/2+3\alpha)n\le |V(H_2)|\le (1-3\alpha)n$.
    \end{claim}
\begin{proofclaim}
   Suppose to the contrary. By Corollary~\ref{Gallai coloured prop}~\ref{edgecolour3}, $|V(H_i)|\ge (1-2\alpha)n$ or $|V(H_i)|\le (1/2+3\alpha)n$ for all $i\in [m]$.
   By Proposition~\ref{A,Kreachable fact}~\ref{itmA,Kreachable fact2}, $E_A=\bigcup_{i\in [m]}\bigcup_{uv\in E(H_i)}E_{A\cup uv}$. Hence together with Lemmas~\ref{small component} and~\ref{large component}, all edges in~$E_A$ are $n^{3k-2|A|+1}$-reachable from all edges in~$E_{\emptyset}$ contradicting with our assumption.
\end{proofclaim}

By Claim~\ref{size claim}, $|V(K_A)|=|V(H_2)|\ge (1/2+3\alpha)n$. So~\ref{Reach3} holds for~$K_A$.

Suppose that $e_1,e_2\in E_A^+$. By Proposition~\ref{A,Kreachable fact}~\ref{itmA,Kreachable fact2}, there exist $u_1v_1, u_2v_2\in E(K_A)$ such that $e_j\in E_{A\cup u_jv_j}^+$ for~$j\in [2]$. By Lemma~\ref{middle component},~\ref{Reach43} holds.

Suppose that $e_1\in E_{A}^-$. By Proposition~\ref{A,Kreachable fact}~\ref{itmA,Kreachable fact2}, we have
$$E_{A}^-=\bigcup_{uv\in E(K_A)}E_{A\cup uv}^-\cup \bigcup_{uv\notin E(K_A)}E_{A\cup uv}.$$
If $e_1\in E^-_{A\cup uv}$ for some $uv\in E(K_A)$, then by~Proposition~\ref{A,Kreachable fact}~\ref{itmA,Kreachable fact7} on~$A\cup uv$, $e_1$ is $n^{3k-2|A|-3}$-reachable from all edges in~$E_{\emptyset}$. If $e_1\in E_{A\cup uv}$ with~$uv\notin E(K_A)$, then $uv\in E(H_i)$ with~$i\neq 2$.  By Corollary~\ref{Gallai coloured prop}~\ref{edgecolour3}, $|V(H_i)|\ge (1-\alpha)n$ or $|V(H_i)|\le n/2$.
Hence, by Lemmas~\ref{small component} and~\ref{large component}, $e_1$ is $n^{3k-2|A|}$-reachable from all edges in~$E_{\emptyset}$. Thus~\ref{Reach44} holds for~$K_A$. 

Therefore, the definition of~$K_A$ is as required.
\end{proof}

\subsection{Proof of Lemma~\ref{reach} for odd $k$}\label{reach odd}
For odd $k$, $\emptyset\notin \mathcal{A}=\partial_{\mod 2}(G)$ and all non-isolated vertices in~$G$ are contained in~$\partial_{\mod 2}(G)$. Next we will define an auxiliary vertex-colouring on~$V(G)$ which enables us to partition label edges in different $E_v$. 
\begin{defn}\label{vertex colour}
   Let $1/n\ll \alpha\ll \gamma \ll 1/k\le 1/5$ with~$k$ odd. Let $G$ be a $k$-graph on~$n$ vertices with~$\overline{\delta}^\alpha_{k-2}(G)\ge 1/2+\gamma$ and $I_G$ be the set of isolated vertex in~$G$. Suppose that $G$ is $(\mathcal{A}, \mathcal{K})$-reachable for~$\mathcal{A}=\partial_{\mod 2}(G)$. We define a vertex-colouring $\psi$ on~$V(G)\backslash I_G$ such that for all distinct $u,v\in V(G)\backslash I_G$, $\psi(u)=\psi(v)$ if and only if there exists $e\in E^+_u$ and $e'\in E^+_v$ such that $e$ is $n^{3k-2}$-reachable from~$e'$ and $e'$ is $n^{3k-2}$-reachable from~$e$.
\end{defn}

Then we have the following property by Proposition~\ref{A,Kreachable fact}~\ref{itmA,Kreachable fact4}.
\begin{proposition}\label{any label vertex}
     Let $1/n\ll \alpha\ll \gamma \ll 1/k\le 1/5$ with~$k$ odd. Let $G$ be a $k$-graph on~$n$ vertices with~$\overline{\delta}^\alpha_{k-2}(G)\ge 1/2+\gamma$ and $I_G$ be the set of isolated vertex in~$G$. Suppose that $G$ is $(\mathcal{A}, \mathcal{K})$-reachable for~$\mathcal{A}=\partial_{\mod 2}(G)$. Let $\psi$ be the vertex-colouring as in Definition~\ref{vertex colour}. If $x,y\in V(G)\backslash I_G$ with~$\psi(x)=\psi(y)$, then for any $e_1\in E^+_{x}$ and any $e_2\in E^+_{y}$, $e_1$ is $3n^{3k-2}$-reachable from~$e_2$.
\end{proposition}

Next we can show that $\psi$ is $2$-vertex-colouring.
\begin{lemma}\label{vertex 2 colour}
    Let $1/n\ll \alpha\ll  \gamma\ll 1/k \ll 1/5$ with~$k$ odd. Let $G=(V,E)$ be a $k$-graph on~$n$ vertices with~$\overline{\delta}^\alpha_{k-2}(G)\ge 1/2+\gamma$. Suppose that $G$ is $(\mathcal{A}, \mathcal{K})$-reachable for~$\mathcal{A}=\partial_{\mod 2}(G)$. Then $\psi$ is $2$-coloured on~$V(G)\backslash I_G$.
\end{lemma}

\begin{proof}
Suppose to the contrary that there are $z_1,z_2,z_3$ with distinct $\psi(z_1),\psi(z_2)$ and $\psi(z_3)$. 
\begin{claim}~\label{three common tuple}
There is a $(z_1, z_2, z_3, (k-3)/2)$-common set $C$. 
\end{claim}

   \begin{proofclaim}
       Let $m=(k-3)/2$. Suppose that we have chosen $x_1,y_1,\dots ,x_{j^*},y_{j^*}$ for some $j^*\in [m-1]\cup \{0\}$ such that $\bigcup_{j\le j^*} x_jy_j=X_{j^*}$ and $X_{j^*}$ is a $(A\cup z_1, A\cup z_2,A\cup z_3, j^*)$-common set. By the definition of~$\psi$, for any $e\in E^+_{z_1}$ and $e'\in E^+_{z_2}$, $e$ is not $n^{3k-2}$-reachable from~$e'$ or $e'$ is not $n^{3k-2}$-reachable from~$e$. By Lemma~\ref{label intersect} with~$(C,A)=(X_{j^*},\emptyset)$, it implies that $$|V(K_{z_1\cup X_{j^*}})\cap  V(K_{z_2\cup X_{j^*}})|> n/2.$$
     Since  $|V(K_{z_3\cup X_{j^*}})|\ge  (1/2+3\alpha)n$ by~\ref{Reach3}, we have 
        
        \begin{align*}
            |\bigcap_{i\in [3]}V(K_{z_i\cup X_{j^*}})|\ge|V(K_{z_1\cup X_{j^*}})\cap  V(K_{z_2\cup X_{j^*}})|+|V(K_{z_3\cup X_{j^*}})|-n\ge 3\alpha n.
        \end{align*} 
Then there exists $x_{j^*+1}y_{j^*+1}\in \bigcap_{i\in [3]}E(K_{z_i\cup X_{j^*}}) $ by Proposition~\ref{perturb prop}~\ref{itm:perturb prop 2}. Let $C=\{x_1,y_1\dots , x_m,y_m\}$ as required.
   \end{proofclaim}
Let $B_i=z_i\cup C \text{ for } i\in [3].$
Note that $|B_i|=k-2$ for~$i\in [3]$. By Lemma~\ref{lem:sturc1}, there exist $B_{i_1}$ and~$B_{i_2}$ such that $E(C_{B_{i_1}})\cap E(C_{B_{i_2}})\neq \emptyset$ for distinct $i_1,i_2\in [3]$. Let $u_1u_2\in E(C_{B_{i_1}})\cap E(C_{B_{i_2}})$ and $e_j=B_{i_j}\cup u_1u_2\in E_{z_{i_j}}^+$ for~$j\in [2]$. Then
          \begin{align*}
              |e_1\cap e_2|=|(B_{i_1}\cup u_1u_2)\cap (B_{i_2}\cup u_1u_2)|=k-|B_{i_1}\backslash B_{i_2}|=k-1
          \end{align*}
          and so $e_j$ is $1$-reachable from~$e_{j'}$ for~$\{j,j'\}=[2]$ by Proposition~\ref{prop:reachablefact1}~\ref{itm:reachablefact4}. 
          However, $e_{i_j}\in E^+_{z_{i_j}}$ for~$j\in [2]$. This contradicts with the fact that $\psi(z_{i_1})\neq \psi(z_{i_2})$.
\end{proof}

Finally, we show the reachable lemma for odd case.

\begin{proof}[Proof of Lemma~\ref{reach} for odd $k$]

 By Lemma~\ref{Lem:reach even case}, $G$ is $(\mathcal{A},\mathcal{K})$-reachable for~$\mathcal{A}=\partial_{\mod 2}(G)$. Let $I_G$ be the set of isolated vertices in~$G$. Define a vertex-colouring $\psi$ on~$V(G)\backslash I_G$ as in Definition~\ref{vertex colour}. By Lemma~\ref{vertex 2 colour}, $\psi$ is 2-vertex coloured. Let $H_1$ and $H_2$ be the colour classes of~$\psi$ with~$|H_1|\ge |H_2|$. Note that $|I_G|\le \alpha n$ by Definition~\ref{defn:perturbed degree}. Then $|H_1|\ge (1-\alpha)n/2$.

Let $B_i=\{e~|~e\in E_x^+ \text{ with }x\in H_i\}$. Now, we enumerate all edges in~$E(G^*)=E_{\emptyset}$ such that if $e_{i_1}\in B_1, e_{i_2}\in B_2\backslash B_1$ and $e_{i_3}\in E_{\emptyset}\backslash (B_1\cup B_2)$, then $i_1\le i_2\le i_3$. We claim this is the enumeration as required.

 If $e\in E_{\emptyset}\backslash(B_1\cup B_2)$, then by Proposition~\ref{A,Kreachable fact}~\ref{itmA,Kreachable fact1}, $e\in E_x^-$ for some $x\in V(G)$. By~Proposition~\ref{A,Kreachable fact}~\ref{itmA,Kreachable fact7}, $e$ is $n^{3k}$-reachable from all edges in~$E_{\emptyset}$. 

\begin{claim}
    If $e\in B_2$, then $e$ is $n^{3k+1}$-reachable from all $e'\in B_1\cup B_2$. 
\end{claim}
\begin{proofclaim}
Suppose that $e\in E_x^+$ for some $x\in H_2$ and $e'\in E_y^+$ for some $y\in V(G)$. If $y\in H_2$, then our claim follows $\psi(x)=\psi(y)$ and Proposition~\ref{any label vertex}. If $y\in H_1$, then $e'\in B_1$. Since $|V(K_{x})|> (1/2+3\alpha)n$ by~\ref{Reach3}, $$|V(K_x)\cap H_1|\ge |V(K_x)|+|H_1|-n>0.$$ Let $z\in V(K_x)\cap H_1$. By Proposition~\ref{prop:common}, there exists a $(x,z,(k-3)/2)$-common set $C$. Then $|x\cup C|=|z\cup C|=k-2$. By Lemma~\ref{lem:sturc1}~\ref{Matching 1}, we have that 
    $$|V(C_{x\cup C})\cap V(C_{z\cup C})|\ge |V(C_{x\cup C})|+|V(C_{z\cup C})|-n>0.$$
    Let $w_1\in V(C_{x\cup C})\cap V(C_{z\cup C})$, $w_1w_2\in E(C_{x\cup C})$ and $w_1w_3\in E(C_{x\cup C})$. Let
    $$f=x\cup C\cup w_1w_2\text{ and }f'=z\cup C\cup w_1w_3.$$
    Then $f\in E_x^+$ and $f'\in E_z^+$. We claim that $f$ is $2n^{3k}$-reachable from~$f'$. Note that $f'\backslash f=\{z,w_3\}$. By Proposition~\ref{prop:reachablefact1}~\ref{itm:reachablefact3} and~\ref{itm:reachablefact1}, it suffices to show that $f$ is $n^{3k}$-reachable from~$z$. Since $z\in V(K_x)$, by Proposition~\ref{A,Kreachable fact}~\ref{itmA,Kreachable fact2}, there exists an $f''\in E_x^+$ with~$z\in f''$. Then $f$ is $n^{3k}$-reachable from~$f''$ by~\ref{Reach43}. Hence, $f$ is $n^{3k}$-reachable from~$z$ and so $2n^{3k}$-reachable from~$f'$.
    
Note that $f$ and $e$ are both in~$E_x^+$. So $e$ is $n^{3k}$-reachable from~$f$ by~\ref{Reach43}. Since $f'\in E_z^+$ and $e'\in E_y^+$ with~$y,z\in H_1$, by~$\psi(y)=\psi(z)$ and Proposition~\ref{any label vertex}, $f'$ is $n^{3k}$-reachable from~$e'$. Therefore, by Proposition~\ref{prop:reachablefact1}~\ref{itm:reachablefact2} with~$(e,e',e'')=(f',f,e)$ and $(e',f',e)$, $e$ is $n^{3k+1}$-reachable from~$e'$.
\end{proofclaim}
Finally, if $e,e'\in B_1$, suppose that $e\in E_x^+$ with~$x\in H_1$ and $e'\in E_y^+$ with~$y\in H_1$. Then by~$\psi(x)=\psi(y)$ and Proposition~\ref{any label vertex}, $e$ is $n^{3k+1}$-reachable from~$e'$. 

Therefore our enumeration is as required.
\end{proof}
\section{Rotatable lemma}\label{rotatable  }
Let $C\in \mathbb{N}$, let $G$ be a $k$-graph, and $e=u_1\dots u_k\in E(G)$. Recall that $e$ is $C$-rotatable if for any $u\in e$, $\sigma\in S_k$ and any rooted $k$-loose tree~$T$ at~$r$, there exists a homomorphism~$\phi$ from~$T$ to~$G$ such that $\phi(r)=u$ and $\phi(v)=u_{\sigma(s)}$ for all $v\in C_s(T)$ with~$\dist(r,v)\ge C$.

Our aim of this section is to prove the rotatable proposition for~$G^*$.

\begin{lemma}\label{main rotatable lemma}
    Let $1/n\ll\alpha \ll  \gamma\le  {1}/{k}\le {1}/{4}$. Let $G$ be a $k$-graph on~$n$ vertices with~$\overline{\delta}^\alpha_{k-2}(G)\ge 1/2+\gamma$. Then all edges in~$G^*$ is $10^7(k!)n^4$-rotatable.
\end{lemma}
When $k=4$, we prove the following (slightly stronger) statement.
\begin{lemma}\label{4-rotatable lemma}
    Let $1/n\ll\alpha \ll  \gamma\le 1/4$. Let $G$ be a $4$-graph on~$n$ vertices with~$\overline{\delta}^\alpha_{2}(G)\ge 1/2+\gamma$. Then all edges in~$G^*$ is $10^7n^4$-rotatable.
\end{lemma}

Then together with Lemmas~\ref{Matching existence} and~\ref{reach}, these imply Lemma~\ref{lem:k-2thereshold}.
\begin{proof}[Proof of Lemma~\ref{lem:k-2thereshold}]
    By Lemmas~\ref{Matching existence},~\ref{reach} and~\ref{main rotatable lemma},~\ref{robust1},~\ref{robust2} and~\ref{robust3} in Definition~\ref{robust} hold for~$G^*$. Therefore $G^*$ is $\eta$-robust.
\end{proof}

In the next subsection, we refine the definition of rotatable. In Section~\ref{subsetcion:2} we show that Lemma~\ref{main rotatable lemma} can be deduced from Lemma~\ref{4-rotatable lemma}. We show some auxiliary structures in Section~\ref{subsetcion:3} when $k=4$. At last, We prove Lemma~\ref{4-rotatable lemma} in Section~\ref{subsetcion:4}.
\subsection{Definitions refinement and basic properties}\label{subsetcion:1}

We refine the definition of rotatable as follows. 
Let $C \in \mathbb{N}$ and $s \in[k]$. 
Let $G$ be a $k$-graph and $ e = u_1 \dots u_k \in E(G)$. 
Let $T$ be a rooted $k$-loose tree at~$r$.
Recall that $C_j(T)$ in Section~\ref{sturc of tree}.
 For~$j\ge 0$ and $s\in [k]$, let $C_s^j(T)=\{v~|~v\in C_{s}(T): \dist(r,v)=j\}$,  $C_s^{\le j}(T)=\bigcup_{i \in [j]}C_s^i(T)$ and $C_s^{\ge j}(T)=\bigcup_{j\ge i}C_s^i(T)$.

Let $\sigma, \pi \in S_k$. We say that $e$ is \emph{$(C,\pi,\sigma)$-rotatable} if, for any rooted $k$-loose tree~$T$ at~$r$, there exists a homomorphism~$\phi$ from~$T$ to~$G$ such that
\begin{align*}
	\phi(v) = 
	\begin{cases}
		u_{\pi(1)} & \text{if $v=r$},\\
		u_{\sigma(s)} & \text{if $v\in C^{\ge C}_{s}(T)$}.
		\end{cases}
\end{align*}

We say that $e$ is \emph{$[C,\pi,\sigma]$-rotatable} if $e$ is $(C,\pi,\sigma \pi)$-rotatable.
We say that $e$ is \emph{$[C,i,\sigma]$-rotatable} if $e$ is $[C,\pi,\sigma ]$-rotatable for all $\pi\in S_k$ with~$\pi(1)=i$.
We say that $e$ is \emph{$[C,*,\sigma]$-rotatable} if $e$ is $[C,\pi,\sigma]$-rotatable for all $\pi \in S_k$. Note that the order of vertices of~$e$ matters, but it will be known from context.

We illustrate one key idea used in the proof of Lemma~\ref{4-rotatable lemma}. Suppose that $\sigma\in S_k$. Let $e=u_1\dots u_k$ be an $k$-edge. Consider a tree~$T$ rooted at~$r$. Let $\mathcal{T}_i=\{T(v):v\in C_i^1(T)\}$ for~$i\in [k-1]$. By Fact~\ref{partition tree}, $r, V(\mathcal{T}_1), \dots, V(\mathcal{T}_{k-1})$ partition~$V(T)$. If $e$ is $[C,i,\sigma]$-rotatable for~$i\neq 1$. Then we attempt to initially embed~$r$ to~$u_1$ and $C_i^1(T)$ to~$u_i$. Next we extend our embedding for each $T'\in \mathcal{T}_i$ such that the root vertex in~$T'$ mapping to~$u_i$ and most vertices of~$C_i(T)\cap V(T')$ mapping to~$u_{\sigma(i)}$.

 We now list some basic properties of rotatable. 

\begin{proposition} \label{prop:rotationfact}
Let $C \in \mathbb{N}$.
Let $G$ be a $k$-graph and $ e = u_1 \dots u_k \in E(G)$.
Let $\pi, \sigma, \tau \in S_k$ with~$\tau=(123\dots k)\in S_k$. 
Then the following holds
\begin{enumerate}[label={\rm (\roman*)}]
	\item if $\sigma(i) = i$, then $e$ is $[1, i, \sigma]$-rotatable and so $e$ is $[1,*,\emph{\textbf{id}}]$-rotatable;
	\label{itm:rotationfact2}
	\item if $e$ is $(C, \pi\tau^j ,  \sigma \tau^j)$-rotatable for all $j \in [k-1]$, then $e$ is $(C+1, \pi, \sigma)$-rotatable; \label{itm:rotationfact3}
    \item if $i_0\in [k]$ and $e$ is $[C,i,\sigma]$-rotatable for all $i\in [k]\backslash \{i_0\}$, then $e$ is $[C+1,*,\sigma]$-rotatable; 
    \label{itm:rotationfact1}
	\item if $e$ is $C$-rotatable, then $e$ is $[C, * ,\sigma]$-rotatable;  \label{itm:rotationfact4}
	\item if $e$ is $[C,*,\sigma]$-rotatable for all $\sigma \in S_k$, then $e$ is $C$-rotatable;
	\label{itm:rotationfact5}
    \item if $e$ is $[C,*,\sigma]$-rotatable, then $e'=u_{\pi(1)}\dots u_{\pi(k)}$ is $[C,*,\pi\sigma\pi^{-1}]$-rotatable.\label{itm:rotationfact8};
	\item if $\sigma_1, \sigma_2 \in S_k$ and $e$ is $[C,*,\sigma_i]$-rotatable for~$i \in [2]$, then $e$ is $[2C,*,\sigma_1\sigma_2]$-rotatable;
	\label{itm:rotationfact6}
	\item if $A$ is a generator of~$S_k$ and $e$ is $[C,\ast,\sigma]$-rotatable for all $\sigma \in A$, then $e$ is $k!C$-rotatable. 
	\label{itm:rotationfact7}

\end{enumerate}
\end{proposition}

\begin{proof}
For~\ref{itm:rotationfact2}, note that if $\pi\in S_k$ and $\pi(1)=i$, then $\sigma\pi(1)=i$. Set $\phi: V(T) \rightarrow V(G)$ to be such that, for all $ v \in C_{s}(T)$, $\phi(v) = u_{\sigma\pi(s)}$. Then $e$ is $[1,i,\sigma]$-rotatable and so $e$ is $[1,*,\textbf{id}]$-rotatable.

We now prove~\ref{itm:rotationfact3}.
Consider any rooted $k$-loose tree~$T$ at~$r$. 
For~$j \in [k-1]$, let $\mathcal{T}_j$ be all rooted subtrees~$T(v)$ with~$ v \in C_{\tau^j(1)}^1(T)$. 
Note that $r, V(\mathcal{T}_1), \dots, V(\mathcal{T}_{k-1})$ partition~$V(T)$. Consider $ j\in [k-1]$, let $T'\in \mathcal{T}_j$ rooted at~$r'$.
Let $C_{s}(T')=C_{\tau^j(s)}(T)\cap V(T')$ for~$s\in [k]$ and $C^{C}_s(T')=C^{C+1}_{\tau^j (s)}\cap V(T')$.
By our assumption, there exists a homomorphism~$\phi_{T'}$ from~$T'$ to~$G$ such that
\begin{align*}
	\phi_{T'}(v) = 
	\begin{cases}
		u_{ \pi\tau^j(1) } & \text{if $v=r'\in C^1_{\tau^j(1)}(T)$},\\
		u_{ \sigma\tau^j(s) } & \text{if $v\in  C^{\ge C}_{s}(T')$}= C^{\ge (C+1)}_{\tau^j(s)}(T)\cap V(T').
	\end{cases}
\end{align*}
Define $\phi : V(T) \rightarrow V(G)$ to be such that 
\begin{align*}
	\phi(v) = 
	\begin{cases}
		u_{\pi(1)} & \text{if $v=r$,}\\
		\phi_{T'}(v)  & \text{if $ v \in V(T')$ for some $T' \in \bigcup_{j \in [k-1]}\mathcal{T}_j$}.
	\end{cases}
\end{align*}
It is easy to check that $\phi$ is a homomorphism from~$T$ to~$G$ with~$\phi(r)=u_{\pi(1)}$ and for all $v\in C_s^{\ge C+1}(T)$, $\phi(v)=u_{\sigma(s)}$.

For~\ref{itm:rotationfact1}, it suffices to show that $e$ is $[C+1,i_0,\sigma]$-rotatable. Let $\pi\in S_k$ with~$\pi(i)=i_0$. Note that for~$j\in [k-1]$, $\pi\tau^j(1)\neq i_0$. Then $e$ is $(C,\pi\tau^j,\sigma\pi \tau^j)$-rotatable by our assumption. So $e$ is $(C+1,\pi,\sigma\pi)$-rotatable by~\ref{itm:rotationfact3}. Hence $e$ is $[C+1,i_0,\sigma]$-rotatable. 

It is easy to see that~\ref{itm:rotationfact4},~\ref{itm:rotationfact5} and~\ref{itm:rotationfact8} hold.

For~\ref{itm:rotationfact6}, 
consider any rooted $k$-loose tree~$T$ at~$r$. 
By our assumption, there exists a homomorphism~$\phi_1$ from~$T$ to~$G$ such that
\begin{align*}
	\phi_1(v) = 
	\begin{cases}
		u_{ \pi(1) } & \text{if $v=r$,}\\
		u_{ \sigma_1\pi(s) } & \text{if $v\in  C^{\ge C}_{s}(T)$.} 
	\end{cases}
\end{align*}
For~$j\in [k]$, let $\mathcal{T}_j$ be all rooted subtrees $T(v)$ with~$ v \in C_{j}^C(T)$. Consider $ j\in [k]$ and let $T'\in \mathcal{T}_j$ rooted at~$r'$.
Let $\pi_j=\sigma_1\pi\tau^{j-1}$ and $C_{s}(T')=C_{\tau^{j-1}(s)}(T)\cap V(T')$ for~$s\in [k]$. So $C^{C'}_s(T')=C^{C+C'}_{\tau^{j-1} (s)}(T)\cap V(T')$. Since $e$ is $[C,*,\sigma_2]$-rotatable, there exists a homomorphism~$\phi'_{T'}$ from~$T'$ to~$G$ to be such that 
\begin{align*}
	\phi'_{T'}(v) = 
	\begin{cases}
		u_{ \pi_s(1) }= u_{ \sigma_1\pi(j) }& \text{if $v=r'$,}\\
		u_{ \sigma_2\pi_s(s) }=u_{\sigma_2\sigma_1\pi\tau^{j-1}(s)} & \text{if $v\in  C^{\ge C}_{s}(T')=C^{\ge 2C}_{\tau^{j-1} (s)}(T)\cap V(T')$.}
	\end{cases}
\end{align*}
Next, define $\phi : V(T) \rightarrow V(G)$ to be such that 
\begin{align*}
	\phi(v) = 
	\begin{cases}
		\phi_1(v) & \text{if $\dist(r,v)\le C$,}\\
		\phi_{T'}(v)  & \text{if $ v \in V(T')$ for some $T' \in \bigcup_{j \in [k-1]}\mathcal{T}_j$}.
	\end{cases}
\end{align*}
It is easy to check that $\phi$ is a homomorphism from~$T$ to~$G$ with~$\phi(r)=u_{\pi(1)}$ and for all $v\in C_s^{\ge 2C}(T)$, $\phi(v)=u_{\sigma_1\sigma_2\pi(s)}$. Therefore~\ref{itm:rotationfact6} holds and~\ref{itm:rotationfact7} follows. 
\end{proof}
\subsection{Reduction to the case $k=4$}\label{subsetcion:2}

We now prove Lemma~\ref{main rotatable lemma} assuming Lemma~\ref{4-rotatable lemma}. Recall the Definition~\ref{defn:our graph} of~$G^*$.
  \begin{proof}[Proof of Lemma~\ref{main rotatable lemma}]
      Let $e=u_1\dots u_k\in E_{\emptyset}=E(G^*)$ with~$u_{1}u_2\in C_{u_3\dots u_{k}}$. Note that $\{(1i):i\in [2,k]\}$ is a generator set of~$S_k$.
      By Proposition~\ref{prop:rotationfact}~\ref{itm:rotationfact7}, it suffices to show that for~$i\in [2,k]$, $e$ is $[10^7n^4,\ast,(1i)]$-rotatable.
      
      Let $i\in [4,k]$, $X=\{1,2,3,i\}$ and $\sigma\in S_X$. We now show that $e$ is $[10^7n^4,*,\sigma]$-rotatable (which implies the lemma). Let $L$ be the link graph of~$e\backslash\{u_1,u_2,u_3,u_i\}$ in~$G$ and $L'$ be the link graph of~$e\backslash\{u_1,u_2,u_3,u_i\}$ in~$G^*$. So $\alpha$-perturbed relative $2$-degree in~$L$ is at least $1/2+\gamma$ by Proposition~\ref{perturb prop}~\ref{itm:perturb prop 1}. Let $L^*$ be the corresponding graph of~$L$ as in Definition~\ref{defn:our graph}. Then by Proposition~\ref{propo:rotatat reduce}, we have that $L^*\subseteq L'$. 

     Let $T$ be a rooted $k$-loose tree at~$r$ and $\pi\in S_k$ with~$\pi(1)=y$. If $y\notin  X$, then $\sigma(y)=y$. Thus by Proposition~\ref{prop:rotationfact}~\ref{itm:rotationfact2}, $e$ is $[1,y,\sigma]$-rotatable for~$y\notin X$.

     Suppose that $\pi(1)=y$ for some $y\in X$. Remove vertices in~$\bigcup_{j\in [k]\backslash X}C_{\pi^{-1}(j)}(T)$ from each edge of~$T$. Then we obtain a set $\mathcal{T}$ of vertex-disjoint $4$-trees. Let $T'\in \mathcal{T}$ be the unique tree with root vertex $r$. Let $C_{\pi^{-1}(j)}(T')=C_{\pi^{-1}(j)}(T)\cap V(T')$ for~$j\in X$. By our assumption, each edge in~$L^*$ is $10^7n^4$-rotatable. Then there is a homomorphism $\phi_{T'}$ from~$T'$ to~$L^*$ such that for~$u\in V(T')$ and $s\in \pi^{-1}(X)$,
     \begin{align*}
	\phi_{T'}(u) = 
	\begin{cases}
		u_{ \pi(1) } & \text{if $u=r$,}\\
		u_{ \sigma \pi(s) } & \text{if $u\in  C^{\ge 10^7n^4}_{s}(T')$}.
	\end{cases}
\end{align*}
 Define $\phi : V(T) \rightarrow V(G)$ to be such that 
     \[
\phi(u) =
\begin{cases}
u_{\pi(s)}=u_{\sigma \pi(s)} &~\text{ if }u\in C_{s}(T)\text{ with }\pi(s)\notin X, \\
u_{\sigma\pi(s)}&~\text{ if }u\in C_{s}(T)\backslash V(T')\text{ with }\pi(s)\in X,\\
\phi_{T'}(u) &~ \text{ if }u\in V(T').

\end{cases}
\]   
Note that $\phi$ is the desired homomorphism from~$T$ to~$G$.
  \end{proof}

    \subsection{Rotatable Structures for~$4$-graphs}\label{subsetcion:3}

We now establish some rotatable structures for proving Lemma~\ref{4-rotatable lemma}. Thus throughout this subsection, we will take $k=4$. We will also adapt the following notation throughout. Let $G$ be a $4$-graph and $e$ an edge in~$G^*$. We write $e=\overline{u_1u_2}u_3u_4$ to mean that $e\in E_{u_1u_2}$.

We sketch the proof for Lemma~\ref{4-rotatable lemma}. Let $A=\{(13),(23),(43)\}\subseteq S_4$. Note that $A$ is a generator set of~$S_4$. By Proposition~\ref{prop:rotationfact}~\ref{itm:rotationfact7}, if any edge $e$ is $[C,*,\sigma]$-rotatable for all $\sigma\in A$, then $e$ is rotatable and we are done. In this subsection, we will focus on the cases that $\sigma=(34)$ and $(23)$ (see more in Lemmas~\ref{prop:triangle rotate} and~\ref{lem:13rotatable}). Furthermore, we will show a structure in Lemma~\ref{prop:special structure} with a transitive property which will help us prove $[C,*,(13)]$-rotatable for edge $e$ in Section~\ref{subsetcion:4}.

The next proposition states that the relationship between rotatablility of two edges that are tight connected.

\begin{proposition}\label{prop:rotationtight}
    Let $C\in \mathbb{N}$. Let $G$ be a $4$-graph. Let $e=u_1\dots u_4$ and $f=u_1'\dots u_4'$ be two edges in~$G$. Suppose that there is a tight walk $W=w_1\dots w_{4\ell}$ from~$e$ to~$f$ such that 
    \begin{align*}
        w_i=u_{\rho_e(i)} \text{ and } w_{4(\ell-1)+i}=u_{\rho_f(i)}' \text{ for } i\in [4]. 
    \end{align*}
Let $T$ be a rooted $k$-loose tree at~$r$, $j\in [4]$ and $\sigma\in S_4$. Suppose that $\phi$ is a homomorphism from~$T$ to~$G$ such that
    \begin{align*}
        \phi(v)=
        \begin{cases}
            u_j &~\text{if $v=r$,}\\
            u_{\sigma(s)}&~\text{if $v\in C^{\ge C}_s(T)$}.
        \end{cases}
    \end{align*}
    Then there exists a homomorphism~$\phi'$ from~$T$ to~$G$ such that 
    \begin{align*}
        \phi'(v)=
        \begin{cases}
            u_j &~\text{if $v=r$,}\\
            u'_{\rho_f \rho_e^{-1} \sigma(s)}&~\text{if $v\in C^{\ge 4(C+{\ell})}_s(T)$}.
        \end{cases}
    \end{align*}
    Moreover, if $e$ is $[C,*,\sigma]$-rotatable, then $f$ is $[4C+8\ell,*,\rho^{-1}\sigma]$-rotatable, where $\rho=\rho_e\rho_f^{-1}$.
    Furthermore, if $e$ is $C$-rotatable, then $f$ is $(4C+8\ell)$-rotatable.
\end{proposition}

\begin{proof}
Recall that the definition of layers $L_j$ of~$T$ for~$j\ge 1$ in Section~\ref{sturc of tree}. Note that if $u\in L_j$, then $j/4\le \dist(r,u)\le j$. For all $t\ge 0$ and $j\in [4]$, let $w_{4(\ell-1+t)+j}=w_{4(\ell-1)+j}$. To define $\phi$, we first embed the first $4C$ layers of~$T$ as $\phi'$ so that it has the correct ``orientation''. We then use $W$ to transfer the embedding onto~$f$. Formally, we define $\phi':V(T)\rightarrow V(G)$ to be such that
   \begin{align*}
       \phi'(v)=
       \begin{cases}
           \phi(v) &~\text{ if $v\in L_i$ with~$i\le 4C$,}\\
           w_{\rho_e^{-1}\sigma(s)+4(i-1)}&~\text{ if $v\in L_{4C+4i+s}$ with~$i\ge 1$ and $s\in [4]$}.
       \end{cases}
   \end{align*}
Then for each $i\ge \ell$ and $s\in [4]$, if $v\in L_{4C+4i+s}$, we have $$\phi'(v)=w_{4(l-1)+\rho_e^{-1}\sigma(s)}=u'_{\rho_f \rho_e^{-1} \sigma(s)}.$$ Therefore $\phi'$ is the desired homomorphism.

For the moreover statement, suppose that $e$ is $[C,*,\sigma]$-rotatable. Let $\pi\in S_k$. Let $w\in T$ with~$w\in \bigcup_{j\in [4]}L_{4\ell +j}$. Consider subtree $T(w)$. Then for each $i\in [4]$, there exists a homomorphism~$\phi_{w,i}$ from~$T(w)$ to~$G^*$ such that  
\begin{align*}
        \phi_{w,i}(v)=
        \begin{cases}
            u_i&~\text{if $v=w$,}\\
            u'_{\rho_f \rho_e^{-1}\sigma(s)}&~\text{if $v\in C^{\ge 4(C+{\ell})}_s(T(w))$.}
        \end{cases}
    \end{align*}
Thus, define $\phi_0:V(T)\rightarrow V(G)$ to be such that for~$j\in [4]$, 
\begin{align*}
        \phi_0(v)=
        \begin{cases}
            \phi_1(v)=w_{4(\ell-i)+\rho_f^{-1}\pi(j)} &~\text{if $v\in L_{4(i-1)+j}$ and $1\le i\le \ell$,}\\
            \phi_{w,i}(v)&~\text{if $v\in V(T(w))$ with~$w\in \bigcup_{j\in [4]}L_{4\ell +j}, \phi_1(w)=u_i$.}
        \end{cases}
    \end{align*}
Note that we have that $$\phi_0(r)=w_{4(\ell-1)+\rho_f^{-1}\pi(1)}=u'_{\pi(1)}$$ and for vertex $v\in C_s^{\ge 8\ell +4C}(T)$, $\phi_0(v)= u'_{\rho^{-1}\sigma (s)}$.
Hence, $f$ is $[4C+8\ell,*, \rho^{-1}\sigma]$-rotatable and so the furthermore statement follows.
\end{proof}

The next proposition states that there is an even tight walk between any edges in~$E_{uv}$. Recall that $C_{uv}$ is the induce graph of the largest component in the link graph of~$uv$ in~$G$.
\begin{proposition}\label{prop:rotationfact2}
    Let $1/n\ll\alpha \ll  \gamma$. Let $G$ be a $4$-graph on~$n$ vertices with~$\overline{\delta}^\alpha_{2}(G)\ge 1/2+\gamma$. Let $\overline{u_1u_2}u_3u_4, \overline{u_1u_2}y_3y_4\in E(G)$. Then there exists a tight $4$-walk $W=w_1w_2\dots w_{4\ell}$ such that, 
    \begin{align*}
        (w_1,w_2,w_3,w_4,w_{4\ell-3},w_{4\ell-2},w_{4\ell-1},w_{4\ell})=(u_1,u_2,y_3,y_4,u_1,u_2,u_3,u_4).
    \end{align*}
\end{proposition}

\begin{proof}
   Note that there is a $2$-walk $W'=w'_1\dots w'_{2\ell}$ with~$\ell\le 2n^2$ and $(w'_1,w'_2,w'_{2\ell-1},w'_{2\ell})=(y_3,y_4,u_3,u_4)$ in~$C_{u_1u_2}$. Indeed, $C_{u_1u_2}$ is a connect component and there exists a triangle in~$C_{u_1u_2}$ by Lemma~\ref{lem:sturc1}~\ref{Matching 3}. 
    Using this triangle, we can adjust the length of walk between $u_3u_4$ and $y_3y_4$ to be a even number. 
    
    Then define $W$ to be $w_{4(i-1)+j}=u_j$ and $w_{4(i-1)+2+j}=w'_{2(i-1)+j}$ for~$i\in [\ell]$ and $j\in [2]$. Note that $W$ is the desired tight walk.
\end{proof}

Let $e=\overline{u_1u_2}u_3u_4\text{ and }f= \overline{u_1u_2}y_3y_4$.
Together with Lemma~\ref{prop:rotationtight} (with~$\rho_e =\rho_f= \mathbf{id}$), we immediately imply following corollary for edges in~$E_{u_1u_2}$.
\begin{coro}\label{coro:tight}
   Let $1/n\ll\alpha \ll  \gamma$. Let $G$ be a $4$-graph on~$n$ vertices with~$\overline{\delta}^\alpha_{2}(G)\ge 1/2+\gamma$. Suppose that $\overline{u_1u_2}u_3u_4\in E(G)$ is $[C, *, \sigma]$-rotatable. Then all edges $e=\overline{u_1u_2}u'u''\in E_{u_1u_2}$ are $[4C+16n^2,*,\sigma]$-rotatable.
\end{coro}

Following lemma states that every edge in~$G^*$ is $[30n^2,*,(34)]$-rotatable.

\begin{lemma}\label{prop:triangle rotate}
    Let $1/n\ll\alpha \ll  \gamma$. Let $G$ be a $4$-graph on~$n$ vertices with~$\overline{\delta}^\alpha_{2}(G)\ge 1/2+\gamma$. Then each edge $\overline{u_1u_2}u_3u_4\in E(G^*)$ is $[30n^2,*,(34)]$-rotatable.
\end{lemma}
\begin{proof}
 By Proposition~\ref{prop:rotationfact2} with~$(\overline{u_1u_2}u_3u_4, \overline{u_1u_2}u_4u_3)$ playing the role of~$(\overline{u_1u_2}u_3u_4, \overline{u_1u_2}y_3y_4)$, there exists a tight walk 
 $W=w_1w_2\dots w_{4l}$ with~$\ell\le 2n^2$ such that 
 \begin{align*}     (w_1,w_2,w_3,w_4,w_{4\ell-3},w_{4\ell-2},w_{4\ell-1},w_{4\ell})=(u_1,u_2,u_3,u_4,u_1,u_2,u_4,u_3).
 \end{align*}
By Proposition~\ref{prop:rotationfact}~\ref{itm:rotationfact2}, $\overline{u_1u_2}u_3u_4$ is $[1,*,\mathbf{id}]$-rotatable. 
Apply Proposition~\ref{prop:rotationtight} with~$(\overline{u_1u_2}u_3u_4,\overline{u_1u_2}u_3u_4,\mathbf{id},(34))$ playing the role of~$(e,f,\rho_e,\rho_f)$. It implies that $\overline{u_1u_2}u_3u_4$ is $[30n^2,*,(34)]$-rotatable.
\end{proof}

Next, we show that under some circumstances, $\overline{u_1u_2}u_3u_4$ is $(23)$-rotatable.
\begin{lemma} \label{lem:13rotatable}
Let $1/n\ll\alpha \ll  \gamma$. Let $G$ be a $4$-graph on~$n$ vertices with~$\overline{\delta}^\alpha_{2}(G)\ge 1/2+\gamma$.
Let $\overline{u_1 u_2} u_3 u_4$, $\overline{u_1 u_2} y_3 y_4 $, $u_1 u_3 y_3 y_4\in E(G)$. 
Then $\overline{u_1 u_2} u_3 u_4$ is $[ 10n^2, *, (23) ]$-rotatable.
\end{lemma}

\begin{proof}
Let $\sigma=(23)$ and $T$ be a rooted $k$-loose tree at~$r$. Since $\sigma(i)=i$ for~$i\in \{1,4\}$, by Proposition~\ref{prop:rotationfact}~\ref{itm:reachablefact1}, $\overline{u_1u_2}y_3y_4$ is $[1,i,\sigma]$-rotatable for  $i\in \{1,4\}$. By Proposition~\ref{prop:rotationfact}~\ref{itm:reachablefact3}, it suffices to show that $\overline{u_1u_2}u_3u_4$ is $[9n^2,3,\sigma]$-rotatable. Let $\pi\in S_4$ with~$\pi(1)=3$.

By Proposition~\ref{prop:rotationfact2} with~$\overline{u_2u_1}y_3y_4$ and $\overline{u_2u_1}u_3u_4$, there exists a tight walk 
 $W=w_1w_2\dots w_{4l}$ with~$\ell\le 2n^2$ such that $$(w_1,w_2,w_3,w_4,w_{4\ell-3},w_{4\ell-2},w_{4\ell-1},w_{4\ell})=(u_2,u_1,y_3,y_4,u_2,u_1,u_3,u_4).$$
Let $W'$ be the tight walk obtained by concatenating $u_3u_1y_3y_4$ and $W$. Rewrite $W'=w'_1\dots w'_{4(\ell+1)}$ such that $$(w'_1,w'_2,w'_3,w'_4,w'_{4\ell+1},w'_{4\ell+2},w'_{4\ell+3},w'_{4\ell+4})=(u_3,u_1,y_3,y_4,u_2,u_1,u_3,u_4).$$
Define the homomorphism~$\phi'$ from~$T$ to~$G$ such that
\begin{align}
		\phi'(v) = 	
		\begin{cases}
		u_1 & \text{if $v \in C_s(T)$ and $ \sigma\pi(s) =1$,}\\
		u_3 & \text{if $v \in C_s(T)$ and $ \sigma\pi(s)=2 $,}\\
		y_3 & \text{if $v \in C_s(T)$ and $ \sigma\pi(s)=3 $,}\\
		y_4 & \text{if $v \in C_s(T)$ and $ \sigma\pi(s)=4 $.}
		\end{cases}
\end{align}
In particular, $\phi'(r)=u_3$ as $\sigma\pi(1)=2$. Apply Proposition~\ref{prop:rotationtight} with~$(u_3u_1y_3y_4,u_2u_1u_3u_4,W',\mathbf{id},\mathbf{id})=(e,f,W,\sigma\pi,\rho_e,\rho_f).$ Then there exists a homomorphism~$\phi$ from~$T$ to~$G$ to be such that 
        \begin{align*}
        \phi(v)=
        \begin{cases}
            u_3 & \text{if $v=r$,}\\
		u_1 & \text{if $v \in C^{\ge 4(\ell+2)}_s(T)$ and $ \sigma\pi(s) =1$,}\\
		u_2 & \text{if $v \in C^{\ge4(\ell+2)}_s(T)$ and $ \sigma\pi(s)=2 $,}\\
		u_3 & \text{if $v \in C^{\ge4(\ell+2)}_s(T)$ and $ \sigma\pi(s)=3 $,}\\
		u_4 & \text{if $v \in C^{\ge4(\ell+2)}_s(T)$ and $ \sigma\pi(s)=4 $}
        \end{cases}
        =
            \begin{cases}
            u_3 &~\text{if $v=r$,}\\
            u_{\sigma\pi(s)}&~\text{if $v\in C^{\ge 4(\ell+2)}_s(T)$}.
        \end{cases}
    \end{align*}
    Note that $\ell\le 2n^2$. Hence, $\overline{u_1u_2}u_3u_4$ is $[9n^2,3,\sigma]$-rotatable as required.
\end{proof}
By the lemma, we have the following corollary.

\begin{coro} \label{cor:23-rotation}
Let $1/n\ll\alpha \ll  \gamma$. Let $G$ be a $4$-graph on~$n$ vertices with~$\overline{\delta}^\alpha_{2}(G)\ge 1/2+\gamma$.
Let $\overline{u_1 u_2} u_3 u_4 \in G^*$. 
Suppose that $E(C_{u_1 u_2}) \cap E(C_{u_1 u_3}) \ne \emptyset$.
Then $\overline{u_1 u_2} u_3 u_4$ is $[ 30n^2, *, (23) ]$-rotatable.
\end{coro}
\begin{proof}
    Let $y_3y_4\in E(C_{u_1 u_2}) \cap E(C_{u_1 u_3})$. Then $\overline{u_1u_2}y_3y_4$ and $\overline{u_1u_3}y_3y_4\in E(G)$. By Lemma~\ref{lem:13rotatable}, we are done.
\end{proof}
Next, we use a special structure to imply a rotation transitivity lemma.

\begin{lemma}\label{prop:special structure}
        Let $1/n\ll\alpha \ll  \gamma$. Let $G$ be a $4$-graph on~$n$ vertices with~$\overline{\delta}^\alpha_{2}(G)\ge 1/2+\gamma$.
        Let $\sigma\in S_3$.
        Let $\overline{u_1u_2}u_3u_4$, $\overline{u_1u_2}y_3y_4\in E(G)$. 
        Suppose that $xu_2u_3y_4\in E(G)$ which is $[C,*,\sigma]$-rotatable. 
        Then $\overline{u_1u_2}u_3u_4$ is $[C+30n^2,*,\sigma]$-rotatable.
        Moreover, if $xu_2u_3y_4$ is $C$-rotatable, then $\overline{u_1u_2}u_3u_4$ is $4!(C+30n^2)$-rotatable.
\end{lemma}

\begin{proof}

By Proposition~\ref{prop:rotationfact}~\ref{itm:rotationfact1}, it suffices to show that $\overline{u_1u_2}u_3u_4$ is $[C+10n^2,i^*,\sigma]$-rotatable for~$i^*\in \{2,3,4\}$. 
Consider $i^*\in\{2,3,4\}$. Let $\pi\in S_4$ be with~$\pi(1)=i^*$.
Let $T$ be a rooted $4$-tree at~$r$.

If $i^*=4$, then $\sigma(4)=4$. By Proposition~\ref{prop:rotationfact}~\ref{itm:rotationfact2}, $\overline{u_1u_2}u_3u_4$ is $[1,4,\sigma]$-rotatable.

We may assume that $i^*=2$ or $3$. Let 
\begin{align*}
(v_1,v_2,v_3,v_4)=(x,u_2,u_3,y_4).    
\end{align*} 
Since $v_1v_2v_3v_4$ is $[C,*,\sigma]$-rotatable, there is a homomorphism~$\phi'$ from~$T$ to~$G$ such that
\begin{align*}
   \phi'(u)=
\begin{cases}
    v_{i^{*}}~&\text{ if }u=r,\\
    v_{\sigma \pi(j)}~&\text{ if }u\in C^{\ge C}_{j}(T).
\end{cases}
\end{align*}

 Next consider each rooted subtree $T(v)$ with~$\dist(r,v)=C$. 
\begin{claim}
    Let $v\in \bigcup_{i\in [4]}C_i^{C'}(T)$ with~$C'\ge C$ and $\phi'(v)=v_j$. Then there exists a homomorphism~$\phi_v$ from~$T(v)$ to~$G$ such that 
\begin{align*}
   \phi_v(u)=
\begin{cases}
    \phi'(v)=v_j~&\text{ if }u=v,\\
    u_{\sigma\pi (s)}~&\text{ if }u\in C^{\ge C'+10n^2}_{s}(T)\cap V(T(v)).
\end{cases}
\end{align*}
    
\end{claim}
\begin{proofclaim}
    If $j=2 \text{ or }3$, then $\phi'(v)=v_j=u_j$ and $v\in C_{j_0}(T)$ with~$\sigma\pi(j_0)=j$. Define $\phi_v$ be such that
$$\phi_v(w)=u_{\sigma\pi(s)}~\text{ if }w\in C_{s}(T)\cap V(T(v)).$$

If $j=4$, then $\phi'(v)=v_4=y_4$ and $v\in C_{j_0}(T)$ with~$\sigma\pi(j_0)=4$. Let $\tau=(1234)$. for~$s\in [k]$, let
$$    C_s(T(v))=C_{\tau^{j_0-1}(s)}(T)\cap V(T(v)).$$
Note that $C^{i} _s(T(v))=C^{(i+C')}_{\tau^{j_0-1}(s)}(T)\cap V(T(v))$. Let $u_1'u_2'u_3'u_4'=u_1u_2y_3y_4$. Since $\sigma(4)=4$, by Proposition~\ref{prop:rotationfact}~\ref{itm:rotationfact2}, $u_1u_2y_3y_4$ is $[1,4, \sigma]$-rotatable and so $(1,\pi\tau^{j_0-1},\sigma)$-rotatable. Then define a homomorphism~$\phi'_v$ from~$T(v)$ to~$G$ to be such that 
\begin{align*}
		\phi'_v(u) = 	
			\begin{cases}
				u_4' & \text{if $u = v$, }\\
				u'_{\sigma \pi \tau^{j_0-1} ( s)} & \text{if $u\in C_s^{\ge 1}(T(v))$}
			\end{cases}
   =
   \begin{cases}
		u_1'=u_1 & \text{if $u \in C_s(T(v))$ and $ \sigma\pi\tau^{j_0-1}(s) =1$,}\\
		u_2'=u_2 & \text{if $u \in C_s(T(v))$ and $ \sigma\pi\tau^{j_0-1}(s)=2 $,}\\
		u_3'=y_3 & \text{if $u \in C_s(T(v))$ and $ \sigma\pi\tau^{j_0-1}(s)=3 $,}\\
		u_4'=y_4 & \text{if $u \in C_s(T(v))$ and $ \sigma\pi\tau^{j_0-1}(s)=4 $.}
		\end{cases}
\end{align*}
 By Proposition~\ref{prop:rotationfact}, there exists a tight walk $W=w_1\dots w_{4\ell}$ with~$\ell\le 2n^2$ such that 
    \begin{align*}
        (w_1,w_2,w_3,w_4,w_{4\ell-3},w_{4\ell-2},w_{4\ell-1},w_{4\ell})=(u_1,u_2,y_3,y_4,u_1,u_2,u_3,u_4).
    \end{align*}
Apply Proposition~\ref{prop:rotationtight} and obtain a homomorphism~$\phi_v$ from~$T(v)$ to~$G$ such that 
\begin{align}
		\phi_v(u)
  = 
			\begin{cases}
				y_4=v_4 & \text{if $u = r$,}\\
				u_{\sigma \pi \tau^{j_0-1}(s)} & \text{if $u\in C_s^{\ge 4(\ell+1)}$}(T(v))=C_{\tau^{j_0-1}(s)}^{\ge 4(\ell+1)+C'}(T)\cap V(T(v)).
			\end{cases}
\end{align}
For vertex $u\in C_{s}^{\ge C'+9n^2}(T)\cap V(T(v))$ with~$s\in [4]$, we have $\phi_v(u)=u_{\sigma\pi(s)}$.

 Therefore we may assume $j=1$. Consider $w\in V(T(v))$ with~$\dist(v,w)=1$. Then $w\in \bigcup_{i\in [4]}C_i^{(C'+1)}(T)$ and $\phi'(w)\neq \phi'(v)=v_1$. Since $\phi'(w)\in \{v_2,v_3,v_4\}$, by our previous argument, there exists a homomorphism~$\phi_w$ from~$T(w)$ to~$G$ such that 
\begin{align*}
   \phi_w(u)=
\begin{cases}
    \phi'(w)~&\text{ if }u=w,\\
    u_{\sigma\pi (s)}~&\text{ if }u\in C^{\ge (C'+1)+9n^2}_{s}(T)\cap V(T(w)).
\end{cases}
\end{align*}
Define the homomorphism~$\phi_v$ from~$T(v)$ to~$G$ to be such that  
\[\phi_v(u)=
    \begin{cases}
        v_1 ~&\text{ if $u=v$},\\
        \phi_{w}(u) ~&\text{ if $ v \in V(T(w))$ for some $w$ with~$\dist(v,w)=1$}.
    \end{cases}
    \]
  in~$\phi_v$, we have that $\phi_v(u)=u_{\sigma\pi(s)}\text{ if }u\in C^{\ge C'+10n^2}_{s}(T)\cap V(T(v))$ for~$s\in [4]$. Therefore our claim holds.
\end{proofclaim}
Now we can give the final homomorphism~$\phi$ from~$T$ to~$G^*$. Define $\phi$ as follow
\[\phi(u)=
    \begin{cases}
        \phi'(u) ~~~&\text{ if } u\in \bigcup_{j\in [4]} C_{j}^{\le C}(T),\\
        \phi_{v}(u) ~~~&\text{ if }u\in V(T(v)) \text{ for some } v \in \bigcup_{j\in [4]}C^{C}_{j}(T).
    \end{cases}
    \]
    Therefore, $\overline{u_1u_2}u_3u_4$ is $[C+10n^2,i^*,\sigma]$-rotatable for~$i^*\in \{2,3\}$.

    Note that $S_3\cup \{(34)\}$ is a generator set of~$S_4$. The moreover statement follows from Lemma~\ref{prop:triangle rotate} and Proposition~\ref{prop:rotationfact}~\ref{itm:rotationfact7}. 
\end{proof}

\subsection{Proof of Lemma~\ref{4-rotatable lemma}}\label{subsetcion:4}
In this subsection, we will prove the final ingredient that $e$ is $[C,*,(13)]$ by digraph. First we show that there is a special vertex from~$V(G)$.
\begin{proposition}\label{Proposition:many mon}
     Let $1/n\ll\alpha \ll  \gamma$. Let $G$ be a $4$-graph on~$n$ vertices with~$\overline{\delta}^\alpha_{2}(G)\ge 1/2+\gamma$. Let $u$ be a non-isolated vertex in~$V(G)$. Then there exists a vertex $w$ such that the number of~$v\in V(G)$ with~$E(C_{uw})\cap E(C_{uv})\neq \emptyset$ is at least $(1/2-2\alpha)n$.
\end{proposition}

\begin{proof}
     Let $N_1(u)$ be the neighbour set of~$u$ in~$\partial_2(G)$. Then $|N_1(u)|\ge (1-\alpha)n$ by Definition~\ref{defn:perturbed degree}. Let $x\in N_1(u)$ and $X=\{x'\in N_1(u):E(C_{ux})\cap E(C_{ux'})\neq \emptyset\}$. If $|X|\ge (1/2-2\alpha)n$, then we are done by setting $w=x$. We may assume that $|X|< (1/2-2\alpha)n$. By Lemma~\ref{lem:sturc1}~\ref{Matching 4}, for all $z,z'\in N_1(u)\backslash (X\cup \{u,x\})$, $E(C_{uz})\cap E(C_{uz'})\neq \emptyset$. Note that $$|N_1(u)\backslash (X\cup \{u,x\})|>(1-\alpha)n-({1}/{2}-2\alpha)n-2\ge {n}/2.$$ We are done by setting $w=z\in N_1(u)\backslash (X\cup \{u,x\})$.
\end{proof}

Now, we are ready to prove Lemma~\ref{4-rotatable lemma}.

\begin{proof}[Proof of Lemma~\ref{4-rotatable lemma}]
    Define a digraph $D$ on~$V(G)$ to be $\overrightarrow{u_1u_2}\in E(D)$ if each edge $\overline{u_1u_2}u'u''$ in~$E_{u_1u_2}$ is $[n^4,*,(23)]$-rotatable.
    
    \begin{claim}\label{outdegree}
        For any non-isolated vertex $u_1\in V(G)$, $d_D^+(u_1)\ge (1/2-2\alpha)n$.
    \end{claim}
\begin{proofclaim}
    By Proposition~\ref{Proposition:many mon}, there exists a vertex $u_2$ such that the number of~$v\in V(G^*)$ with~$E(C_{u_1u_2})\cap E(C_{u_1v})\neq \emptyset$ is at least $(1/2-2\alpha)n$. Let $A=\{v\in V(G)~|~E(C_{u_1u_2})\cap E(C_{u_1v})\neq \emptyset\}$. Then $|A|\ge (1/2-2\alpha)n$. By Lemma~\ref{lem:sturc1}~\ref{Matching 1}, we have that
    \begin{align*}
        |V(C_{u_1u_2})\cap A|\ge |V(C_{u_1u_2})|+|A|-n > 0.
    \end{align*}
    Let $u_3\in V(C_{u_1u_2})\cap A$ and $\overline{u_1u_2}u_3u_4\in E(G^*)$. Since $E(C_{u_1u_2})\cap E(C_{u_2u_3})\neq \emptyset$, by Corollary~\ref{cor:23-rotation}, $\overline{u_1u_2}u_3u_4$ is $[10n^2,*,(23)]$-rotatable.

For~$x_2\in A$ and $\overline{u_1x_2}x_3x_4\in E(G^*)$, let $y_3y_4\in E(C_{u_1u_2})\cap E(C_{u_1x_2})$. By Proposition~\ref{prop:rotationfact2}, there exist two tight walks $W=w_1\dots w_{4\ell}$ and $W'=w'_1\dots w'_{4\ell'}$ with~$\ell,\ell'\le 2n^2$ such that
\begin{align*}
    (w_1,w_2,w_3,w_4,w_{4\ell-3},w_{4\ell-2},w_{4\ell-1},w_{4\ell})=(u_1,x_2,x_3,x_4,u_1,x_2,y_3,y_4)
\end{align*}
and
\begin{align*}
(w'_1,w'_2,w'_3,w'_4,w'_{4\ell'-3},w'_{4\ell'-2},w'_{4\ell'-1},w'_{4\ell'})=(u_1,u_2,y_3,y_4,u_1,u_2,u_3,u_4).
\end{align*}
Note that $WW'=w_1\dots w_{4\ell}w_1'\dots u_{4\ell'}$ is a tight $4$-walk. Hence by Proposition~\ref{prop:rotationfact2} with~$\overline{u_1u_2}u_3u_4, \overline{u_1x_2}x_3x_4$, $\textbf{id}, \textbf{id})$ playing the role of~$(e,f,\rho_e,\rho_f)$,  we have $\overline{u_1x_2}x_3x_4$ is $[n^4,*,(23)]$-rotatable. Thus, each edge $e$ in~$E_{u_1v}$ with~$v\in A$ is $[n^4,*,(23)]$-rotatable. It implies that $A\subseteq N_D^+(u_1)$ and $d_D^+(u_1)\ge (1/2-2\alpha)n$.
\end{proofclaim}
   \begin{claim}\label{Claim final 1}
       Let $\overrightarrow{u_1u_2}\in E(D)$. Then every edge $e\in E_{u_1u_2}$ is $10^4n^4$-rotatable.
   \end{claim}
   \begin{proofclaim}
       By Lemma~\ref{lem:sturc1}~\ref{Matching 1} and Claim~\ref{outdegree}, we have that 
       \begin{align*}
           |V(C_{u_1u_2})\cap N_D^+(u_2)|\ge |V(C_{u_1u_2})|+|N_D^+(u_2)|-n>0. 
       \end{align*}
       Let $u_3\in V(C_{u_1u_2})\cap N_D^+(u_2)$ and $\overline{u_1u_2}u_3u_4\in E(G)$. Meanwhile, by Lemma~\ref{lem:sturc1}~\ref{Matching 1}, we obtain that
              \begin{align*}
           |V(C_{u_1u_2})\cap V(C_{u_2u_3})|\ge |V(C_{u_1u_2})|+|V(C_{u_2u_3})|-n>0. 
       \end{align*}
       Let $y_4\in V(C_{u_1u_2})\cap V(C_{u_2u_3})$ and $\overline{u_1u_2}y_3y_4$, $\overline{u_2u_3}xy_4\in E(G^*)$. Note that $u_3\in N^+_D(u_2)$, so $\overline{u_2u_3}xy_4$ is $[n^4,*,(23)]$-rotatable. Hence $xu_2u_3y_4$ is $[n^4,*,(13)]$-rotatable by Proposition~\ref{itm:rotationfact8} with~$\pi=(132)$. By Lemma~\ref{prop:special structure}, it implies that $\overline{u_1u_2}u_3u_4$ is $[2n^4,*,(13)]$-rotatable.

        By Lemma~\ref{prop:triangle rotate}, $\overline{u_1u_2}u_3u_4$ is $[n^4,*,(34)]$-rotatable. Meanwhile, since $u_2\in N_D^+(u_1)$, $\overline{u_1u_2}u_3u_4$ is $[n^4,*,(23)]$-rotatable. Note that $\{(13),(34),(23)\}$ is a generator of~$S_4$, so $\overline{u_1u_2}u_3u_4$ is $10^3n^4$-rotatable by Proposition~\ref{prop:rotationfact}~\ref{itm:rotationfact7}. By Corollary~\ref{coro:tight}, each edge $e\in E_{u_1u_2}$ is $10^4n^4$-rotatable.
   \end{proofclaim}
Suppose that $e\in E_{u_1u_2}$ is not $10^7n^4$-rotatable for some distinct $u_1u_2\in E(\partial_2(G))$. By Claim~\ref{Claim final 1}, $\overrightarrow{u_1u_2}\notin E(D)$. Then by Lemma~\ref{lem:sturc1}~\ref{Matching 1}, we have 
       \begin{align*}
           |V(C_{u_1u_2})\cap N_D^+(u_2)|\ge |V(C_{u_1u_2})|+d^+(u_2)-n>0. 
       \end{align*}
       Let $u_3\in V(C_{u_1u_2})\cap N^+(u_2)$ and $\overline{u_1u_2}u_3u_4\in E(G)$. Meanwhile, by Lemma~\ref{lem:sturc1}~\ref{Matching 1}, we obtain that
              \begin{align*}
           |V(C_{u_1u_2})\cap V(C_{u_2u_3})|\ge |V(C_{u_1u_2})|+|V(C_{u_2u_3})|-n>0. 
       \end{align*}
       Let $y_4\in V(C_{u_1u_2})\cap V(C_{u_2u_3})$ and $\overline{u_1u_2}y_3y_4$, $\overline{u_2u_3}xy_4\in E(G^*)$. Note that $u_3\in N^+_D(u_2)$. $\overline{u_2u_3}xy_4$ is $10^4n^4$-rotatable. So $xu_2u_3y_4$ is $10^4n^4$-rotatable. By Lemma~\ref{prop:special structure}, it implies that $\overline{u_1u_2}u_3u_4$ is $10^6n^4$-rotatable. By Corollary~\ref{coro:tight}, each edge $e\in E_{u_1u_2}$ is $10^7n^4$-rotatable, a contradiction.
\end{proof}

\section{Concluding Remark}\label{remark}
In this paper, we showed that $\delta_{k,k-2}^T=1/2$ for~$k\ge 4$. However, our proof is unable to map the root vertex to an arbitrary vertex due to~\ref{robust2}. Then we may ask whether this is possible.
\begin{ques}
    Let $1/n \ll \gamma \ll 1/\Delta, 1/k \le 1/4$. Let $G$ be a $k$-graph on~$n$ vertices with~$\overline{\delta}_{k-2}(G)\ge 1/2+\gamma$ and $v\in V(G)$. Let $T$ be a rooted $n$-vertex $k$-loose tree at~$r$ with~$\Delta_1(T)\le \Delta$. Is there an embedding~$\psi$ from~$T$ to~$G$ such that $\psi(r)=v$?
\end{ques}

Another direction is to investigate how large $\Delta_1(T)$ in Theorem~\ref{generaltheorem} can be. 
For~$2$-graph case, result in~\cite{komlos2001spanning} shows that tree~$T$ with~$\Delta(T)\le cn/\log n$ can be embedded to a graph $G$ with~$\overline{\delta}_1(G) >1/2$ and this is tight up to a constant $c$.

Furthermore, for almost spanning loose trees, there are two directions to consider.

\subsection{Almost Spanning loose trees in general hypergraphs}
 R\"odl, Ruci\'nski and Szemer\'edi~\cite{Rodlalmost} showed that in~$k$-graph~$G$, if $\overline{\delta}_{k-1}(G) >1/k$, then there exists an almost perfect matching in~$k$-graphs. However, the perfect matching threshold $\delta_{k,k-1}^{PM} = 1/2$. Note that in Theorem~\ref{generaltheorem}, the~$1/2$ for embedding spanning tree is needed due to the absorption step. So we wonder whether the relative minimum degree threshold for existence of almost spanning loose trees is smaller than the threshold for spanning loose trees.
\begin{problem}
    Let $1/n\ll \varepsilon \ll \gamma \ll 1/\Delta, 1/k$ and $\ell \in [k-1]$.
    Determine infimum $\delta$ such that if $G$ is a $k$-graph on~$n$ vertices with~$\overline{\delta}_{\ell}(G)\ge \delta+\gamma$ and $T$ be a $k$-loose tree on~$(1-\varepsilon)n$ vertices with~$\Delta_1(T)\le \Delta$, then there is an embedding from~$T$ to~$G$. In particular, when $\ell = k-1$, is $\delta<\delta^{T}_{k,k-1}=1/2$? 
\end{problem}

 Recall that a binary $k$-loose tree of even depth containing a perfect matching. 
 On the other hand, there is no perfect matching in loose Hamiltonian path.
 Keevash, K\"uhn, Mycroft and Osthus~\cite{looseHamilton} showed that if a $k$-graph $G$ with sufficient large $n$ vertices and $\overline{\delta}_{k-1}(v)> \frac{1}{2(k-1)}$, then there exists a loose Hamiltonian cycle and so a loose Hamiltonian path in~$G$. 
 Hence, we ask whether the minimum degree threshold for the  existence of embedding for a loose tree~$T$ depends on the size of the largest matching in~$T$.
 
\begin{problem}
    Let $1/n\ll \gamma \ll 1/\Delta, 1/k$, $\ell \in [k-1]$ and $\theta\in [0,1]$. Determine the infimum $\delta$ such that if $G$ is a $k$-graph on the $n$ vertices with~$\overline{\delta}_{\ell}(G)\ge \delta+\gamma$ and $T$ is a $n$-vertex $k$-loose tree with the largest matching size~$\theta n$ and $\Delta_1(T)\le \Delta$, then there exists an embedding from~$T$ to~$G$. 
\end{problem}

\subsection{Almost Spanning loose trees in linear hypergraphs}
The other direction is the existence of loose trees embedding in linear hypergraphs. We say a hypergraph is \emph{linear} if any two edges share at most one vertex. Recently, Im and Lee~\cite{im2024dirac} considered the Dirac type problem on linear hypergraphs. They showed that in linear $k$-graph $G$, if $\delta_1(G)>1/k$, then there exists an almost perfect matching in~$G$ and this bound is asymptotically tight. Moreover, they showed that in~$G$, there exists an embedding for every spanning $k$-loose tree with bounded vertex degree and $o(|V(G)|)$ leaves. In~\cite{Im2024}, Im, Kim, Lee and Methuku showed that $\delta_1(G)>(1/2-o(1))n$ is sufficient for the existence of embedding.

We note that in general, finding perfect matchings in linear hypergraphs is not possible. Even if the graph is a Steiner triple system, it may not contain a perfect matching. 
Since a binary $k$-loose tree with even depth contains a perfect matching, we can not guarantee the existence for all spanning $k$-loose trees in linear hypergraphs. Hence, it would be interesting to consider the following generalization problem for almost spanning loose trees.
\begin{ques}
 Let $1/n\ll \varepsilon\ll \gamma \ll \Delta,1/k$. Determine the infimum $\delta$ such that if $G$ is a linear $k$-graph on~$n$ vertices, $\overline{\delta}_1(G)\ge \delta+\gamma$ and $T$ is a $k$-loose tree on~$(1-\varepsilon)n$ vertices and $\Delta_1(T)\le \Delta$, then there exists an embedding from~$T$ to~$G$.
\end{ques}
Note that Im, Kim, Lee and Methuku~\cite{Im2024} showed that $\delta \le 1/2$. 
\bibliographystyle{amsplain}
\bibliography{cite}
\end{document}